\documentclass[11pt,reqno]{article}

\usepackage[margin=1in]{geometry}
\usepackage{amsmath,amsthm,amssymb}

\usepackage{hyperref}

\usepackage[abbrev,lite,nobysame]{amsrefs}
\usepackage{times}
\usepackage[usenames,dvipsnames]{color}

\usepackage{marginnote}

\usepackage{mathtools,enumitem,mathrsfs}

\usepackage[compact]{titlesec}

\usepackage[title]{appendix}

\usepackage{comment}

\mathtoolsset{showonlyrefs=true}
    
\newcommand{\rr}{\mathbb{R}}

\newcommand{\lan}{\langle}
\newcommand{\ran}{\rangle}
\newcommand{\be}{\begin{eqnarray*}}
\newcommand{\bel}{\begin{eqnarray}}
\newcommand{\ee}{\end{eqnarray*}}
\newcommand{\eel}{\end{eqnarray}}
\newcommand{\ba}{\begin{aligned}}
\newcommand{\ea}{\end{aligned}}
\newcommand{\de}{\Delta}
\newcommand{\al}{\alpha}

\newcommand{\ep}{\epsilon}
\newcommand{\eps}{\epsilon}

\newcommand{\wt}{\widetilde}

\newcommand{\cc}{-\int_{-1}^1 e^{-\al w}H_-(W) dw}
\newcommand{\bb}{-\int_{-1}^1 e^{\al w} H_+(W) dw}
\newcommand{\ddd}{\int_{-1}^1 e^{\al w}H_-(W) dw}
\newcommand{\aaa}{\int_{-1}^1 e^{-\al w} H_+(W) dw}

\newcommand{\Real}{\mathbb R}

\newcommand{\Complex}{\mathbb C}

\renewcommand{\ep}{\epsilon}

\newcommand{\grad}{\nabla}

\newcommand{\norm}[1]{\left|\left| #1 \right|\right|}
\newcommand{\abs}[1]{\left| #1 \right|}
\newcommand{\set}[1]{\left\{ #1 \right\}}
\newcommand{\brak}[1]{\left\langle #1 \right\rangle} 

\newcommand{\ph}{\mathrm{ph}}

\usepackage{bbm}

\newtheorem{theorem}{Theorem}[section]

\newtheorem{corollary}[theorem]{Corollary}
\newtheorem{lemma}[theorem]{Lemma}
\newtheorem*{lemma*}{Lemma}

\theoremstyle{definition}

\newtheorem{remark}[theorem]{Remark}

\numberwithin{equation}{section}

\begin{document}

\title{Inviscid damping and enhanced dissipation of the boundary layer for 2D Navier-Stokes linearized around Couette flow in a channel}
\author{Jacob Bedrossian\thanks{\footnotesize Department of Mathematics, University of Maryland, College Park, MD 20742, USA \href{mailto:jacob@math.umd.edu}{\texttt{jacob@math.umd.edu}}. J.B. was supported by NSF CAREER grant DMS-1552826 and NSF RNMS \#1107444 (Ki-Net)} \and Siming He\thanks{\footnotesize Department of Mathematics, Duke University, NC 27708, USA \href{mailto:simhe@math.duke.edu}{\texttt{simhe@math.duke.edu}}.}}

\maketitle

\begin{abstract} 
We study the 2D Navier-Stokes equations linearized around the Couette flow $(y,0)^t$ in the periodic channel $\mathbb T \times [-1,1]$ with no-slip boundary conditions in the vanishing viscosity $\nu \to 0$ limit. 
We split the vorticity evolution into the free evolution (without a boundary) and a boundary corrector that is exponentially localized to at most an $O(\nu^{1/3})$ boundary layer. 
If the initial vorticity perturbation is supported away from the boundary, we show inviscid damping of \emph{both} the velocity \emph{and} the vorticity associated to the boundary layer. 
For example, our $L^2_t L^1_y$ estimate of the boundary layer vorticity is \emph{independent of $\nu$}, provided the initial data is $H^1$.
For $L^2$ data, the loss is only logarithmic in $\nu$. 
Note both such estimates are false for the vorticity in the interior.
To the authors' knowledge, this inviscid decay of the boundary layer vorticity seems to be a new observation not previously isolated in the literature.
Both velocity and vorticity satisfy the expected $O(\exp(-\delta\nu^{1/3}\alpha^{2/3}t))$ enhanced dissipation in addition to the inviscid damping. 
Similar, but slightly weaker, results are obtained also for $H^1$ data that is against the boundary initially.
For $L^2$ data against the boundary, we at least obtain the boundary layer localization and enhanced dissipation. 

\end{abstract}

\setcounter{tocdepth}{1}
{\small\tableofcontents}

\section{Introduction}\label{sec:Intro}
We study the 2D Navier-Stokes equations linearized around the Couette flow in the periodic channel $(x,y) \in \mathbb T \times [-1,1]$.
In vorticity form, the equations are
\begin{equation} \label{eq:2DNSE}
\left\{
\begin{array}{l}
\partial_t \omega + y \partial_x \omega  = \nu \Delta \omega, \quad\quad (x,y) \in \mathbb T \times [-1,1]\\
u = \grad^{\perp}(-\Delta)^{-1} \omega := \begin{pmatrix} -\partial_y \\ \partial_x \end{pmatrix} (-\Delta)^{-1} \omega, \\  
u|_{y = \pm 1}  = 0, \\
\omega(0,x,y) = \omega_{in}(x,y).   
\end{array}
\right.
\end{equation}
Fourier transforming $x \mapsto \alpha$, denoting $\widehat{\omega} = \widehat{\omega}(t,\alpha,y)$, gives the following equation for each $\alpha \in \mathbb Z$ (see Section \ref{sec:Note} for Fourier conventions), 
\begin{equation}
\left\{
\begin{array}{l}
\partial_t \widehat{\omega} + i\alpha y \widehat{\omega}  = \nu \left(\partial_{yy} - \alpha^2\right)\widehat{\omega}, \quad\quad y \in [-1,1]\\
\widehat{u} = \begin{pmatrix} -\partial_y \\ i\alpha \end{pmatrix}(\alpha^2 - \partial_{yy})^{-1} \widehat{\omega}, \\  
\widehat{u}|_{y = \pm 1}  = 0, \\
\omega(0,x,y) = \omega_{in}(x,y). 
\end{array}
\right.
\end{equation}
By symmetry, it suffices to consider only the case $\alpha > 0$, which is the convention taken below. 
The case of $y \in \Real$ with no boundary condition on $u|_{y = \pm 1}$, gives the free equations
\begin{equation} \label{eq:free}
\left\{
\begin{array}{l}
\partial_t \widehat{\omega} + i\alpha y \widehat{\omega}  = \nu \left(\partial_{yy} - \alpha^2\right)\widehat{\omega}, \quad\quad y \in \Real \\
  \widehat{u} = \begin{pmatrix} -\partial_y \\ i\alpha \end{pmatrix}(\alpha^2 - \partial_{yy})^{-1} \widehat{\omega} ,\\
  \omega(0,x,y) = \omega_{in}(x,y).   
\end{array}
\right.
\end{equation}
These were solved by Kelvin in 1887 \cite{Kelvin87}; see also \cite{LinZeng11,BMV14} and Section \ref{sec:Free}. 
Two effects are apparent: (A) the enhanced dissipation effect that provides a much accelerated decay time-scale ($O(\nu^{-1/3})$ as opposed to $O(\nu^{-1})$), quantified by the following, 
\begin{align}
\norm{\widehat{\omega}(\alpha,\cdot)}_{L^2_y} \leq e^{-\alpha^2 \nu t - \frac{1}{12}\nu \alpha^{2} t^3} \norm{\widehat{\omega}_{in}(\alpha,\cdot)}_{L^2_y}
\end{align}
and (B) the inviscid damping of the velocity that provides some asymptotic stability independent of $\nu$ (noted first by Orr \cite{Orr07}; see also \cite{LinZeng11})
\begin{subequations} \label{ineq:IDog}
\begin{align}
\abs{\alpha}^{1/2}\norm{\widehat{u}(\alpha,\cdot)}_{L^2_t L^2_y}  & \lesssim \norm{\widehat{\omega}_{in}(\alpha,\cdot)}_{L^2_y}, \\
 \brak{\alpha t} \norm{\widehat{u^1}(t,\alpha,\cdot)}_{L^2_y} + \brak{\alpha t}^2 \norm{\widehat{u^2}(t,\alpha,\cdot)}_{L^2_y} &  \lesssim \norm{\widehat{\omega}_{in}(\alpha,\cdot)}_{H^2_y}. 
\end{align}
\end{subequations}
See Section \ref{sec:Note} for notation conventions regarding $f \lesssim g$ (also, we use $\brak{x} = (1 + \abs{x}^2)^{1/2}$). 
Both effects are caused by the shear sending enstrophy to high frequencies linearly in time and are both associated with the continuous spectrum of linearized Euler. 
The enhanced dissipation is caused by the increased effectiveness of the viscous dissipation at high frequencies; see e.g. \cite{RhinesYoung83,BajerEtAl01,TDG04}. It has recently generated a lot of interest in the mathematics community, both in relation to hydrodynamic stability (see e.g. \cite{BMV14,BGM15I,BGM15II,WZ18,LWZ17,IMM17,WZZ_2DK,Gallay18,Wei18} and the references therein) and in relation to the stirring of passive scalars, where it has also been called relaxation enhancement (\cite{CKRZ08,CZT18,FG19} and the references therein).
The inviscid damping is due to the negative order operator recovering the velocity from the vorticity. This effect is somewhat analogous to Landau damping in plasma physics (see the discussions in \cite{LZ11b,MouhotVillani11,BM15} and the references therein). 
In the absence of boundaries, the nonlinear dynamics near the  Couette flow in 2D and 3D at high Reynolds number is reasonably well-understood but still some questions remain. See \cite{LinZeng11,BM15,BMV14,BGM15I,BGM15II,BVW16,DM18,WZ18} and the review article \cite{BGM_Bull}. 

The presence of boundaries changes the picture significantly, in particular, the mismatch of the viscous and inviscid boundary conditions normally generates vorticity in a boundary layer that does not vanish in the inviscid limit. See \cite{Batchelor,Maekawa2014,MT10,GGN16I,GGN16II,GN17,GVMM18,KVW19,MaekawaMazzucato18} and the references therein.  
Often, the presence of such boundary layers leads to long-wave linear instabilities in channel shear flows even for flows that are stable at $\nu = 0$; see e.g. \cite{GGN16I,GGN16II} in mathematics and \cite{DR81,Yaglom12} in physics. 
This is not true for the Couette flow however, as was proved by Romanov in 1973 \cite{Romanov73}. See \cite{Wasow53} for earlier work and \cite{MarcusPress77} for a more detailed Green's function expansion of the solution. 
Recently, the work \cite{ChenEtAl18} proved enhanced dissipation and inviscid damping results for the linearized 2D Navier-Stokes in a periodic channel (in $x$).
We work also in this setting; the problem of understanding the fine dynamics of long-waves in an infinite channel remains open.

Finally, we remark that in the case $\nu = 0$, the study of Couette flow in a channel is far more advanced. Classical works on the linearized problem are given in \cite{Dikii1960,Case1960}, for a modern treatment of optimal inviscid damping of linearized shear flows close to Couette see \cite{Z17}.
Nonlinear inviscid damping was recently proved in \cite{IonescuJia18} for vorticity initially separated from the boundary (a requirement which is in some sense necessary \cite{Z16}). 
The harder study of non-Couette shears is also beginning to be understood for $\nu =0$, at least in the linearized case; see e.g. \cite{Z16,Z17,WZZ18,WZZ_2DK,WZZ19,J19}.

\subsection{Main results}
\subsubsection{Separated from the boundary}
The first set of results detail the case when the initial perturbation vorticity is compactly supported away from the boundary.
Note that this is not quite the same as the case studied by Maekawa in \cite{Maekawa2014}, since here the background vorticity is still all the way against the boundary.
Intuition suggests results should be stronger if the perturbation is away from the boundary, and our results confirm this by providing the strongest results thus far. 
The proof takes advantage of (A) a variation of a clever idea from \cite{ChenEtAl18} (see Section \ref{sec:OrrSomm} below) for using the free evolution \eqref{eq:free} to estimate resolvents; (B) optimal understanding of the inviscid damping and enhanced dissipation of \eqref{eq:free}; and (C) methods for obtaining precise estimates on Green's functions via Airy function asymptotics inspired by arguments in e.g. \cite{GGN16I,GGN16II,GN17,GN17_Away}. 
See Remark \ref{rmk:ChenEtAl} for more discussions on the refinements to the results of \cite{ChenEtAl18} that we provide in this case. 

Note that when formulating the equations on the vorticity side, the no-slip boundary conditions imply (see also \cite{Romanov73,ChenEtAl18})
\begin{align}
\int_{-1}^1 e^{\pm \al y}\widehat{\omega}(t,\alpha,y)dy = 0. \label{eq:bdyc}  
\end{align}
Throughout the paper, we define 
\begin{align*}
\ep:=\al^{-1/3}\nu^{1/3}.
\end{align*}
Define the following weight for $p \in [1,\infty]$
\begin{align} 
{\eta_{\ep,p}}=\left\{\begin{array}{rrrr}
\ep^{\frac{3}{2}-\frac{3}{2p}},\quad& d(y,\pm 1)\leq \ep^{\frac{3}{2}};\\
\max\left\{\frac{d(x,\pm1)^2}{\ep^{3/2+1/p}},d(y, \pm1)^{1-\frac{1}{p}}\right\}, \quad & \ep^{\frac{3}{2}}\leq d(y,\pm 1) \leq \ep;\\
\ep^{\frac{1}{2}-\frac{1}{p}}
\exp\left(c_\star\frac{d(x,\pm 1)^{{3}/{2}}}{\ep^{{3}/{2}}} \right),\quad & d(y,\pm 1)\geq \ep. \\
\end{array}\right. \label{def:eta1}
\end{align}
Here $c_\star$ is a (small) universal constant independent of $\al, \nu$. 
Define the decay rate (for some small universal constant $\delta$ determined by the proof and any small $\kappa > 0$), 
\begin{align}
\lambda = (1-\kappa)\alpha^2 \nu + \delta \alpha^{2/3}\nu^{1/3}. \label{def:lambda}
\end{align}
\begin{theorem}[Initial separation from the boundary] \label{thm:Sepbd}
Suppose that $\textup{supp} \, \omega_{in} \subset [-1+\delta_0,1-\delta_0]$ for some $\delta_0 > 0$ and satisfies the no-slip condition $\int_{-1}^1 \widehat{\omega}_{in}(\alpha,y) e^{\pm \alpha y} dy = 0$. Then if we write the decomposition
\begin{align}
\widehat{\omega}(t,\alpha,y) = \widehat{\omega}_f(t,\alpha,y) + \widehat{\omega}_b(t,\alpha,y),  
\end{align}
where $\omega_f$ is the solution to \eqref{eq:free}  on $(x,y) \in \mathbb T \times \mathbb R$ and $\omega_b$ is the boundary corrector, $\omega_b$ satisfies the following estimates  $\forall \delta'$ sufficiently small (implicit constants depending on $\delta_0$, $\kappa$, $p$ but \emph{never} on $t$, $\alpha$, or $\nu$), 
\begin{itemize}
\item[(i)] uniform in time weighted estimates, analyticity in $x$, and enhanced dissipation for $p \in [1,\infty)$
\begin{align}
||\eta_{\ep,p}\widehat{\omega}_b(t,\alpha,\cdot)||_{L^p_y}\lesssim  e^{-\delta'\alpha} e^{-\lambda t} \norm{\widehat{\omega}_{in}(\alpha,\cdot)}_{L^2_y}; \label{ineq:SemiGrp}
\end{align}
\item[(ii)] inviscid damping of the vorticity: for $p \in [1,2]$, there holds 
\begin{subequations} \label{ineq:IDv}
\begin{align} 
|| e^{\lambda t}\eta_{\ep,p}\widehat{\omega}_b(t,\alpha,\cdot)||_{L^2_t L^p_y} &\lesssim  e^{-\delta'\alpha} \norm{\partial_y \widehat{\omega}_{in}(\alpha,\cdot)}_{L^2_y}, \label{ineq:IDv1} \\ 
|| e^{\lambda t}\eta_{\ep,p}\widehat{\omega}_b(t,\alpha,\cdot)||_{L^2_t L^p_y} &\lesssim \brak{\ln \nu} e^{-\delta'\alpha} \norm{\widehat{\omega}_{in}(\alpha,\cdot)}_{L^2_y} \label{ineq:IDv2}
\end{align}
\end{subequations}
and for $p \in [2,\infty)$, there holds 
\begin{align}
|| e^{\lambda t}\eta_{\ep,p}\widehat{\omega}_b(t,\alpha,\cdot)||_{L^p_y L^2_t} &\lesssim  e^{-\delta'\alpha} \norm{\partial_y \widehat{\omega}_{in}(\alpha,\cdot)}_{L^2}; \label{ineq:IDv3}
\end{align}
\item[(iii)] inviscid damping and enhanced dissipation of the velocity
\begin{align}
\norm{e^{\lambda t}\widehat{u}_b(t,\alpha,\cdot)}_{L^2_t L^2_y} \lesssim  e^{-\delta'\alpha} \norm{\widehat{\omega}_{in}(\alpha,\cdot)}_{L^2}. \label{ineq:IDu}
\end{align}
\end{itemize}
\end{theorem}

\begin{remark}
Note that $\eta_{\eps,1} \geq 1$ and hence \eqref{ineq:IDv1} and \eqref{ineq:SemiGrp} imply the uniform-in-$\nu$ estimates
\begin{align}
|| e^{\lambda t} \widehat{\omega}_b(t,\alpha,\cdot)||_{L^\infty_t L^1_y} & \lesssim e^{-\delta'\alpha} \norm{\widehat{\omega}_{in}(\alpha,\cdot)}_{L^2_y} \\ 
|| e^{\lambda t}\widehat{\omega}_b(t,\alpha,\cdot)||_{L^2_t L^1_y} & \lesssim e^{-\delta'\alpha} \norm{\partial_y \widehat{\omega}_{in}(\alpha,\cdot)}_{L^2_y}. 
\end{align}
\end{remark}

\begin{remark}
Note that \eqref{ineq:IDv1} provides the improved $L^\infty_y$ inviscid damping estimate at the cost of more regularity 
\begin{align}
\norm{e^{\lambda t}\widehat{u}_b(t,\alpha,\cdot)}_{L^2_t L^\infty_y} \lesssim  e^{-\delta'\alpha} \norm{\partial_y \widehat{\omega}_{in}(\alpha,\cdot)}_{L^2_y}. \label{ineq:IDu2}
\end{align}
\end{remark}

\begin{remark} \label{rmk:L2L2vort}
Due to the special structure of $L^2$, we can relax the regularity used in \eqref{ineq:IDv1} for $p=2$ at the cost of a slightly different weight
\begin{align}
|| e^{\lambda t}\eta^\ast_{\ep,p}\widehat{\omega}_b(t,\alpha,\cdot)||_{L^2_t L^2_y} &\lesssim_p  e^{-\delta'\alpha} \norm{\widehat{\omega}_{in}(\alpha,\cdot)}_{L^2},  
\end{align}
where $\eta^\ast$ is given by
\begin{align} \label{def:eta1star}
{\eta^\ast_{\ep,p}}=\left\{\begin{array}{rrrr}
\ep^{\frac{3}{2}-\frac{3}{2p}},\quad& d(x,\pm 1)\leq \ep^{\frac{3}{2}};\\
d(x,\pm1)^{1-\frac{1}{p}}, \quad & \ep^{\frac{3}{2}}\leq d(x,\pm 1) \leq \ep;\\
\ep^{1-\frac{1}{p}}
\exp\left( c_\star\frac{d(x,\pm 1)^{{3}/{2}}}{\ep^{{3}/{2}}}\right),\quad & d(x,\pm 1)\geq \ep.
\end{array}\right.
\end{align}
Note that the change is only visible in the regions away from the core of the boundary layer. 
\end{remark}

\begin{remark} \label{rmk:ChenEtAl}
In \cite{ChenEtAl18}, the authors study the case of $H^1$ data against the boundary; see Section \ref{sec:Against} for the results we obtain in that case. 
Under the different (basically stronger) assumption of vorticity initially separated from the boundary, Theorem \ref{thm:Sepbd} provides a number of stronger estimates than \cite{ChenEtAl18}. 
Our decomposition into $\omega_f$ and $\omega_b$ is significantly more precise: $\omega_f$ is explicit and the weight in \eqref{def:eta1} provides
 precise localization of the boundary layer.
The inviscid decay of the \emph{vorticity}, provided especially by the $p=1$ case of \eqref{ineq:IDv} is new, which shows that the total vorticity generated at the boundary is bounded and decaying independent of $\nu$. In \cite{ChenEtAl18} an estimate which is suggestive of a slightly weaker version of \eqref{ineq:IDv1} at $p=2$ is given. 
Our results also provide a variety of more minor improvements, such as providing higher $L^p$ estimates and combined inviscid damping and enhanced dissipation estimates in \eqref{ineq:IDv}, \eqref{ineq:IDu}, and \eqref{ineq:IDu2}. This latter estimate, \eqref{ineq:IDu2}, is also an improvement over the $L^\infty_y$ estimate provided in \cite{ChenEtAl18} in terms of damping. 
\end{remark}

\begin{remark}In the $p=\infty$ case, we can also obtain the following estimate:
\begin{align*}
\norm{\frac{\eta^\ast_{\ep,\infty}}{\brak{\mathbf{1}_{d(\cdot,\pm 1) \leq \eps} \ln \frac{d(\cdot,\pm1)}{\ep}} }\widehat{\omega}_b(t,\alpha,\cdot) }_{L_y^\infty}\lesssim_{\delta',\kappa} e^{-\delta'\al-\lambda t}||\widehat{\omega}_{in}(\alpha,\cdot)||_{L_y^2}.
\end{align*}
The proof is analogous to the case $p \in [1,\infty)$, however, the extra logarithmic factor compensates for a divergence in the proof of Lemma \ref{lem:Kpm} below. 
\end{remark}

\begin{remark}
	Note that the weight \eqref{def:eta1} is exponentially large $d(y,\pm 1) \geq \eps$  and that the boundary corrector is analytic-in-$x$ depending only on $\delta$, not on $\nu$. This behavior is analogous to Maekawa's boundary layer decomposition \cite{Maekawa2014}. 
\end{remark}

Many of the works on the inviscid limit on Navier-Stokes have studied the $L^2$ velocity convergence of the viscous problem to the inviscid problem as well as weak convergence of the vorticity to the inviscid evolution with an additional vortex sheet created at the boundary; see the classical work of Kato \cite{Kato84} and the works of \cite{Constantinetall19,LopesFilhoetall08,MazzucatoNiuWang11,MazzucatoTaylor08} and the references therein. See also the review article \cite{MaekawaMazzucato18}. 
The following is a corollary of Theorem \ref{thm:Sepbd}. The $L^2$ convergence of the velocity follows by an easy adaptation of Kato's criterion \cite{Kato84} and the formation of the vortex sheet follows from standard arguments [see e.g. Theorem 3.8 \cite{MaekawaMazzucato18}].     

\begin{corollary}[Vortex sheet in the inviscid limit]
	Let $\omega_{in;\nu} \in L^2$ such that $\textup{supp} \, \omega_{in;\nu} \subset (-1+\delta_0,1-\delta_0)$ for some $\delta_0 > 0$. 
	Let $u_\nu$ and $\omega_\nu$ denote the velocity and vorticity of solutions to \eqref{eq:2DNSE} with $\nu \geq 0$ with no-slip conditions if $\nu > 0$ and no-penetration ($u^2|_{y \pm 1} = 0$) if $\nu = 0$ and that these conditions are satisfied also by the initial data. Suppose that $\lim_{\nu \to 0}\norm{\omega_{in;\nu} - \omega_{in;0}}_{L^2} = 0$. Then for each $T > 0$ fixed,  
	\begin{itemize}
		\item $\lim_{\nu \to 0} \sup_{t \in [0,T]} \norm{u_\nu - u_0}_{L^2} = 0$; 
		\item $\omega^\nu \rightharpoonup \omega^0 - u_0^1(t,x,1) \delta_{y=1} +  u_0^1(t,x,-1) \delta_{y = -1}$,
where the convergence is in $L^\infty(0,T;(H^1)^\ast)$. Here $\delta_{y = \pm 1}$ denotes the Hausdorff measure of the lines $\set{(x,y) \in \mathbb T \times \mathbb R : y = \pm 1}$. 
	\end{itemize}
\end{corollary}

\subsubsection{Against the boundary: $H^1$ case} \label{sec:Against}
Under the assumption of $H^1$ regularity, a variation of the proof of Theorem \ref{thm:Sepbd} also yields analogous results without assuming separation from the boundary. Theorem \ref{thm:RegData} is similar to results in \cite{ChenEtAl18}, however, our decomposition yields more precise boundary layer localization, slightly stronger decay estimates, and the ``nearly'' inviscid damping of the vorticity captured especially by \eqref{ineq:IDv2Reg} in the case $p=1$. Finally, the proof is also a little simpler and more direct, combining ideas from both \cite{ChenEtAl18} and e.g. \cite{GGN16I,GGN16II,GN17,GN17_Away}. 

\begin{theorem}[$H^1$ data touching the boundary] \label{thm:RegData}
Suppose $\omega_{in} \in H^1$ and satisfies the no-slip condition \eqref{eq:bdyc} at $t = 0$. Then if we write the decomposition
	\begin{align}
	\widehat{\omega}(t,\alpha,y) = \widehat{\omega}_f(t,\alpha,y) + \widehat{\omega}_b(t,\alpha,y),  
	\end{align}
	where $\omega_f$ is the solution to \eqref{eq:free} and $\omega_b$ is the boundary corrector, $\omega_b$ satisfies the following (with constants depending on $p,\kappa$), 
	\begin{itemize}
		\item[(i)] uniform in time weighted estimates and enhanced dissipation for $p \in [1,\infty)$
		\begin{align}
||\eta_{\ep,p}\widehat{\omega}_b(t,\alpha,\cdot)||_{L^p_y}\lesssim  e^{-\lambda t}\left(\abs{\alpha}^{1/2}\norm{\widehat{\omega}_{in}(\alpha,\cdot)}_{L^2_y} + \frac{1}{\abs{\alpha}^{1/2}} \norm{\partial_y\widehat{\omega}_{in}(\alpha,\cdot)}_{L^2_y}\right); \label{ineq:SemiGrpReg}
		\end{align}
		\item[(ii)] inviscid damping of the vorticity: for $p \in [1,2]$, there holds ($\eta_{\ep,p}^\ast$ as in Remark \ref{rmk:L2L2vort})
		\begin{subequations} \label{ineq:IDvReg}
			\begin{align}  
|| e^{\lambda t}\eta_{\ep,p}\widehat{\omega}_b(t,\alpha,\cdot)||_{L^2_t L^p_y} &\lesssim \brak{\ln \frac{\nu}{\alpha}} \left(\abs{\alpha}^{1/2} \norm{\widehat{\omega}_{in}(\alpha,\cdot)}_{L^2} + \frac{1}{\abs{\alpha}^{1/2}} \norm{\partial_y\widehat{\omega}_{in}(\alpha,\cdot)}_{L^2}\right) ,\label{ineq:IDv2Reg} \\ 
			|| e^{\lambda t}\eta^\ast_{\ep,p}\widehat{\omega}_b(t,\alpha,\cdot)||_{L^2_t L^2_y} &\lesssim  \left(\abs{\alpha}^{1/2}\norm{\widehat{\omega}_{in}(\alpha,\cdot)}_{L^2} + \frac{1}{\abs{\alpha}^{1/2}}\norm{\partial_y\widehat{\omega}_{in}(\alpha,\cdot)}_{L^2}\right);
			\end{align}
		\end{subequations}
		\item[(iii)] inviscid damping and enhanced dissipation of the velocity
		\begin{align}
		\norm{e^{\lambda t}\widehat{u}_b(t,\alpha,\cdot)}_{L^2_t L^2_y} \lesssim  \left(\abs{\alpha}^{1/2}\norm{\widehat{\omega}_{in}(\alpha,\cdot)}_{L^2_y} + \frac{1}{\abs{\alpha}^{1/2}}\norm{\partial_y\widehat{\omega}_{in}(\alpha,\cdot)}_{L^2_y}\right). \label{ineq:IDuReg}
		\end{align}
	\end{itemize}
\end{theorem}

\subsubsection{Against the boundary: $L^2$ case}
In the general case of arbitrary $L^2$ data, our results are considerably weaker, however, they are still new. 
We are currently unsure how to apply the methods used in Theorems \ref{thm:Sepbd} and \ref{thm:RegData} to obtain good results. 
Instead, we use contour integral methods more similar to those employed in \cite{GN17} and the references therein. 
For $t\geq 1$, we will use the following the weight, which is a slight variation of \eqref{def:eta1} in the limit $p \to \infty$ 
\begin{align}\label{eta_ep}
\widetilde{\eta}_\ep(y) =\left\{\begin{array}{rrr}
\ep^{3/2},\quad& d(y,\pm 1)\leq \ep^{3/2};\\
\left(\frac{d(y,\pm1)}{\ep}\right)^{3}, \quad &d(y,\pm 1)\in [\ep^{3/2}, \ep]; \\
\exp\left( c_\star\frac{d(y,\pm 1)^{3/2}}{\ep^{3/2}}\right),\quad & d(y,\pm 1)\geq \ep.
\end{array}\right.
\end{align}
On the initial layer $t\leq 1$, we apply the following weaker weight that vanishes on the boundary 
\begin{align}\label{weight_eta_ep}
\widetilde{\eta}_\ep^{in}(y) =\left\{\begin{array}{rrr}
\left(\frac{d(y,\pm 1)}{\ep}\right)^{{3}},&\quad d(y,\pm 1)\leq \ep, \\
\exp\left(c_\star\frac{d(y,\pm1)^{3/2}}{\ep^{3/2}} \right),&\quad d(y,\pm 1)\geq \ep.  
\end{array}\right.
 \end{align}
\begin{theorem}[$L^2$ data touching the boundary] \label{thm:NonSepbd} 
Suppose that $\widehat{\omega}_{in}$ satisfies the no-slip condition $\int_{-1}^1 \widehat{\omega}_{in}(\alpha,y) e^{\pm \alpha y} dy = 0$. Then if we write the decomposition
\begin{align}
\widehat{\omega}(t,\alpha,y) = \widehat{\omega}_f(t,\alpha,y) + \widehat{\omega}_b(t,\alpha,y),  
\end{align}
where $\omega_f$ is the solution to \eqref{eq:free}  on $(x,y) \in \mathbb T \times \mathbb R$ and $\omega_b$ is the boundary corrector, $\omega_b$ satisfies the following for all $p \in [1,\infty]$ (implicit constants depending only on $\kappa$ and $p$), 
\begin{subequations} \label{ineq:thm2}
\begin{align}
||\widetilde{\eta}_\ep^{in}\widehat{\omega}_b(t,\alpha,\cdot)||_{L^p_y} &\lesssim \ln ^{{5/4}}\ep^{-1} \sqrt{\ln \ln \ep^{-1}}  e^{-\lambda t }||\widehat{\omega}_{in}(\alpha,\cdot)||_{L^p_y},\quad \forall t\leq1,\\
||\widetilde{\eta}_\ep\widehat{\omega}_b(t,\alpha,\cdot)||_{L^p_y} &\lesssim \ln ^{{5/4}}\ep^{-1} \sqrt{\ln \ln \ep^{-1}} e^{-\lambda t }||\widehat{\omega}_{in}(\alpha,\cdot)||_{L^p_y},\quad \forall t\geq1,
\end{align}
\end{subequations}
\end{theorem}
\begin{remark}
The pre-factor $\ln ^{{5/4}}\ep^{-1} \sqrt{\ln \ln \ep^{-1}}$ is not expected to be sharp. 
\end{remark}

\begin{remark}
In the initial layer $t\leq1$, the  $t\geq 1$ proof derives the following estimate:
\begin{align*}
||\eta_\ep\widehat{\omega}||_{L^p}
\lesssim\frac{1}{t} \ln ^{{5/4}}\ep^{-1} \sqrt{\ln \ln \ep^{-1}}  e^{-\lambda t} ||\widehat{\omega}_{in}||_{L^p}.
\end{align*}
\end{remark}

\subsubsection{Notation and conventions} \label{sec:Note}
We define $\brak{x} = (1 + \abs{x}^2)^{1/2}$. 
The fourier transform in $x \mapsto \alpha \in \mathbb Z$ and $(x,y) \mapsto \mathbb Z \times \Real$ is defined via
\begin{align}
\widehat{f}(\alpha,y) &= \frac{1}{\sqrt{2\pi}} \int_{\mathbb T} e^{-i x \alpha} f(x,y) dx, \\
f(x,y) &= \frac{1}{\sqrt{2\pi}} \sum_{\alpha \in \mathbb Z}  e^{i x \alpha} \hat{f}(k,y), \\
\widehat{f}(\alpha,\eta) &= \frac{1}{2\pi} \int_{\mathbb T \times \Real} e^{-i x \alpha - i y \eta} f(x,y) dx dy, \\
f(x,y) & = \frac{1}{2\pi} \sum_{\alpha \in \mathbb Z} \int_{\Real} e^{i x \alpha + iy \eta} \hat{f}(\alpha,\eta ) d\eta. 
\end{align}
We write $f \lesssim g$ (resectively $f\gtrsim g)$ if there exists a constant $C> 0$ such that $f \leq C g$ (respectively $Cf \geq g$). In the entire paper, these implicit constants will \emph{NEVER} depend on $\nu,\alpha$, or $t$ (and hence also not on $\ep$). If we wish to emphasize the dependence of the implicit constants on certain other parameters $\beta, \gamma,...$ we write $f \lesssim_{\beta,\gamma,...} g$. We write $f \approx g$ if $f \lesssim g$ and $f \gtrsim g$. 

For complex numbers $c \in \Complex$ we denote $c = c_r + i c_i$ with $c_r,c_i \in \Real$. 
For the entire paper, we set the following branch cut convention: for $z \in \Complex$, $\ph z \in (-\pi,\pi)$ we take
\begin{align}
z^{3/2} = \abs{z}^{3/2} e^{\frac{3}{2} i \ph z}. \label{eq:Branch}
\end{align}

\section{Outline and formulation of resolvent estimates} \label{sec:OrrSomm}
In this section we outline the argument and leave the details to be carried out in later sections. 

\subsection{Resolvent formulation and the Orr-Sommerfeld problem}
As usual in linearized fluid mechanics, the analysis requires studying the resolvent.
Define $\Delta_\alpha = \partial_{yy} - \alpha^2$, the linearized operator
\begin{align}
\mathcal{L}_\alpha = -i\alpha y + \nu \Delta_\alpha, 
\end{align}
and the resolvent operator $\mathcal{R}(\alpha,z)$ as
\begin{align}
\mathcal{R}(\alpha,z) = \left(\mathcal{L}_\alpha - z \right)^{-1}, 
\end{align}
defined on the resolvent set
$$\rho(\mathcal{L}_\alpha) := \set{z \in \Complex : \mathcal{L}_\alpha - z\, \textup{ is invertible}}. $$
Recall that the spectrum is defined as $\sigma(\mathcal{L}_\alpha) = \Complex \setminus \rho(\mathcal{L}_\alpha)$. Let $\gamma$ be such that the $\sigma(\mathcal{L}_\alpha) \subset \set{z \in \Complex:  z_r  < \gamma}$. 
The semigroup is then expressed as the Dunford-Pettis integral (see e.g. \cite{EN01}), 
\begin{align}
e^{ \mathcal{L}_\alpha t} \widehat{\omega}_{in}= \frac{1}{2\pi i} \int_{\gamma \pm i \infty} e^{z t} \mathcal{R}(\alpha,z)\widehat{\omega}_{in} dz. 
\end{align}
Here $\int_{\gamma \pm i \infty} f(z) dz$ denotes a usual straight line Bromwich contour extending from $\gamma + i\infty$ to $\gamma - i \infty$.
It is convenient to reformulate the integral to match more closely the approach to the inviscid problem as in e.g. \cite{DR81,WZZ18}. 
To that end, define $c = iz/\alpha$ and write (abusing notation and not renaming $\mathcal{R}$)
\begin{align}
e^{\mathcal{L}_\alpha t} \widehat{\omega}_{in} = - \frac{\alpha}{2\pi } \int_{\frac{i\gamma}{\alpha} \pm \infty} e^{-i\alpha c t} \mathcal{R}(\alpha,c)\widehat{\omega}_{in} dc. \label{eq:MasterContour}
\end{align}
Here $\int_{i\gamma \pm  \infty} f(z) dz$ denotes a rotated straight line Bromwich contour extending from $i\gamma - \infty$ to $i\gamma + \infty$.
Note this simply rescales and switches the convention about the roles of real and imaginary parts of the spectral parameter.
As is standard for both the viscous and inviscid problems \cite{DR81}, we next express $\mathcal{R}$ in terms of the solution of the Orr-Sommerfeld problem (known as the Rayleigh problem in the inviscid case).
Define $\phi$ to be the streamfunction of the resolvent: $\mathcal{R}(\alpha,c)\widehat{\omega}_{in} = \Delta_\alpha \widehat{\phi}$. Then we obtain the family of singular ODEs paramterized by both $\alpha \in \mathbb Z$ and $c \in \Complex$, 
\begin{align}
\Delta_\alpha^2 \widehat{\phi} - \frac{i\alpha}{\nu}(y-c)\Delta_\alpha\widehat{\phi} & = - \frac{\widehat{\omega}_{in}}{\nu} ,\\
\widehat{\phi}\big|_{y = \pm 1 } = \partial_y\widehat{\phi}\big|_{y = \pm 1} & = 0.
\end{align}

\subsection{Derivation of boundary corrector}
We next divide the resolvent into `free' $\mathcal{R}_f$ and `boundary-corrector' $\mathcal{R}_b$ contributions (and correspondingly the streamfunction)
\begin{align}\label{Resolvent_decomposition}
\Delta_\alpha \widehat{\phi} = \mathcal{R}_f + \mathcal{R}_b = \Delta_\alpha \widehat{\phi}_f + \Delta_\alpha \widehat{\phi}_b
\end{align}
where $\mathcal{R}_f$ is chosen as the decaying solution to the Airy problem on the whole line 
\begin{align}
\Delta_\alpha \mathcal{R}_f  - \frac{i\alpha}{\nu}(y - c) \mathcal{R}_f & = - \frac{\widehat{\omega}_{in}}{\nu},
\end{align}
and $\mathcal{R}_b = \Delta_\alpha \widehat{\phi}_b$ where $\widehat{\phi}_b$ solves the homogeneous Orr-Sommerfeld problem with the suitable boundary conditions to ensure that $\widehat{\phi} = \widehat{\phi}_f + \widehat{\phi}_b$ satisfies the no-slip boundary conditions 
\begin{subequations}
\begin{align}
&\Delta_\alpha^2 \widehat{\phi}_b  - \frac{i\alpha}{\nu}(y - c)\Delta_\alpha \widehat{\phi}_b  = 0, \label{def:HOS} \\
&\widehat{\phi}_b\big|_{y = \pm 1}  = -\widehat{\phi}_f\big|_{y \pm 1} ,\\
&\partial_y\widehat{\phi}_b\big|_{y = \pm 1}  = -\partial_y\widehat{\phi}_f\big|_{y \pm 1}.
\end{align}
\end{subequations}
Note that we will always be considering $c \in \Complex$ (usually with non-vanishing real and imaginary parts). 
The decomposition of the resolvent yields a corresponding decomposition of the evolution:
\begin{align}
e^{t \mathcal{L}_\alpha}\widehat{\omega}_{in} = \widehat{\omega}_f(t) + \frac{\alpha }{2\pi } \int_{\frac{i\gamma}{\alpha} \pm \infty} e^{-i\alpha c t} \Delta_\alpha \widehat{\phi}_b dc, \label{eq:EvoD}
\end{align}
where $\widehat{\omega}_f(t,\alpha, y)$ solves the following `free' evolution on \emph{the whole line} $y \in \mathbb R$,
\begin{subequations} \label{def:free}
\begin{align}
&\partial_t \widehat{\omega}_f + i\alpha y \widehat{\omega}_f - \nu \Delta_\alpha \widehat{\omega}_f  = 0,  \quad y \in \mathbb R,\\
&\widehat{\omega}_f(0,\alpha,y)  = \widehat{\omega}_{in}(\alpha,y). 
\end{align}
\end{subequations}
The solution $\omega_f$ of course does not satisfy the boundary conditions for $t > 0$; these are corrected by the latter term in \eqref{eq:EvoD}.
Optimal estimates on $\omega_f$ are straightforward via Fourier transform (see Appendix \ref{sec:Free}), hence the precise analysis of the boundary corrector is the remaining problem.
We denote the boundary corrector $\omega_b$ as
\begin{align}
\widehat{\omega}_b(t,\alpha,y) = \frac{\alpha }{2\pi } \int_{\frac{i\gamma}{\alpha} \pm \infty} e^{-i\alpha c t} \Delta_\alpha \widehat{\phi}_b(\alpha,y) dc. 
\end{align}
The homogeneous Orr-Sommerfeld problem \eqref{def:HOS} has four linearly independent solutions: $e^{\pm \alpha y}$ and $H_{\pm}(Y)$, where $H_{\pm}$ are expressed via Airy functions (see e.g. \cite{Wasow53,DR81}),
\begin{subequations} \label{def:HYetc}
\begin{align}
Y  :&= \frac{y - c - i \alpha \nu}{\eps}, \label{Langer_variable} \\
H_+(Y) & = Ai \left( e^{5i\pi/6} Y \right), \\
H_-(Y) & = Ai \left( e^{i\pi/6} Y \right). 
\end{align}
\end{subequations}
See \eqref{eq:Branch} in Section \ref{sec:Note} for the branch cut convention we are taking for the entire work. 
The variable $Y$ is called the Langer variable. \emph{For the rest of the paper, capital letters $Y,Z,...$  etc denote the corresponding Langer variable (depending also on $\alpha, \nu, c$), for example $Z = \frac{z-c-i\alpha \nu}{\eps}$}. 
The use of such variables is standard for studying the Orr-Sommerfeld problem (see e.g. \cite{DR81,Romanov73,ChenEtAl18,GGN16I,GGN16II} and the references therein). 
The boundary condition \eqref{eq:bdyc} implies the analogous condition on the resolvent (see \cite{Romanov73} or \cite{ChenEtAl18})
\begin{align}
\int_{-1}^1 e^{\pm \al y}\widehat{\omega}(\alpha,y)dy =\int_{-1}^1e^{\pm \al y} \de_\al \widehat{\phi}(\alpha,y) dy =0. 
\end{align}
Therefore, the boundary-corrector resolvent can be expressed in the following form
\begin{align}
\mathcal{R}_b(c,y) = a(c) H_-(Y) + b(c) H_+(Y).
\end{align}
where $a(c)$ and $b(c)$ solve the following linear system (assuming for a moment that the system is solvable)
\begin{align}
\begin{pmatrix}
-\int_{-1}^1 e^{\alpha y} \mathcal{R}_f(y,c) dy \\
-\int_{-1}^1 e^{-\alpha y} \mathcal{R}_f(y,c) dy
\end{pmatrix}
=
\begin{pmatrix}
\int_{-1}^1 e^{\alpha w} H_-(W) dw  &   \int_{-1}^{1} e^{\alpha w} H_+(W) dw \\
\int_{-1}^1 e^{-\alpha w} H_-(W) dw  &   \int_{-1}^{1} e^{-\alpha w} H_+(W) dw
\end{pmatrix}
\begin{pmatrix}
a(c) \\
b(c)
\end{pmatrix}
.
\end{align}
Denoting $D(\alpha,c)$ the determinant of the matrix (known as the the \emph{Evans function}) we can solve this whenever $D(\alpha,c) \neq 0$ and we have
\begin{align}
\begin{pmatrix}
a(c) \\
b(c)
\end{pmatrix}
=
-\frac{1}{D(\alpha,c)}
\begin{pmatrix}
\int_{-1}^1 e^{-\alpha w} H_+(W) dw  &   -\int_{-1}^{1} e^{\alpha w} H_+(W) dw \\
-\int_{-1}^1 e^{-\alpha w} H_-(W) dw  &   \int_{-1}^{1} e^{\alpha w} H_-(W) dw
\end{pmatrix}
\begin{pmatrix}
\int_{-1}^1 e^{\alpha y} \mathcal{R}_f(y,c) dy \\
\int_{-1}^1 e^{-\alpha y} \mathcal{R}_f(y,c) dy
\end{pmatrix}. \label{eq:abdef}
\end{align}
The spectrum of the operator, $\sigma(\mathcal{L}_\alpha)$, is precisely where $D(\alpha,c) = 0$.

\subsection{Evans function estimates}
The first results one needs are quantitative lower bounds on the Evans function wherever one is drawing the contour integrals of the resolvent. 
First estimates were obtained by Romanov \cite{Romanov73}, and more precise estimates (at least for $\abs{\alpha} \gtrsim 1$) were recently obtained by Chen et. al. in \cite{ChenEtAl18}.
Naturally, we require estimates similar to those in \cite{ChenEtAl18}; the specific estimates used in our work are sketched in Appendix \ref{sec:Evans} for the readers' convenience. 
The main estimate we use is essentially the following, see Appendix \ref{sec:Evans} for more precision.
As in \cite{ChenEtAl18}, the proof is a more quantitative variation of Romanov's arguments \cite{Romanov73} using precise Airy function asymptotics. 
See Section \ref{sec:Note} for notation conventions of complex numbers and branch cuts. 

\begin{lemma} \label{lem:ReadableEvans}
The Evans function $D(\al,c)$ is non-zero in the right half-plane $\set{z \in \Complex: z_i \geq -\al\nu-\delta \al^{-1/3} \nu^{1/3}}$. Moreover, for $c \in \Complex$ with $c_i = -\al\nu-\delta \al^{-1/3} \nu^{1/3}$ (where $\delta$ is a universal constant chosen sufficiently small). 
  Then the absolute value of Evans function $D(\al,c)$ is bounded from below by
\begin{align}
|D(\al,c)|\gtrsim&\left|\int_{-1}^1 e^{-\al w}H_-(W)dw\right|\left|\int_{-1}^1 e^{\al w}H_+(W)dw\right|
,\label{D_al_c_lower_bound}
\end{align}
with the following lower bounds depending on the region of the complex plane:
\begin{align}
\left|\int_{-1}^1e^{-\al w}H_-(W)dw\right|\gtrsim &
\frac{\ep e^\al}{\lan\al\ep\ran} \left\{\begin{array}{rr}
 1, & \quad \abs{\frac{-1-c_r}{\eps}} \lesssim 1, \\ 
\brak{\frac{-1-c_r}{\ep}}^{-3/4}\exp\left({\frac{2}{3}Re((\frac{-1-c_r}{\ep}+\delta i)^{3/2}e^{i\pi/4})}\right),&\quad \frac{-1-c_r}{\eps} \lesssim -1, \\ 
\brak{\frac{-1-c_r}{\ep}}^{-3/4}\exp\left({-\frac{2}{3}Re((\frac{-1-c_r}{\ep}+\delta i)^{3/2}e^{i\pi/4})}\right),&\quad \frac{-1-c_r}{\eps} \gtrsim 1, 
\end{array}\right.\nonumber\\
\left|\int_{-1}^1e^{\al w}H_+(W)dw\right|\gtrsim & \frac{\ep e^\al}{\lan\al\ep\ran} \left\{\begin{array}{rr}
1, & \quad \abs{\frac{-1+c_r}{\eps}} \lesssim 1 ,\\ 
\brak{\frac{-1+c_r}{\ep}}^{-3/4}\exp\left({\frac{2}{3}Re((\frac{-1+c_r}{\ep}+\delta i)^{3/2}e^{i\pi/4})}\right),&\quad \frac{-1+c_r}{\ep} \lesssim -1, \\ 
\brak{\frac{-1+c_r}{\ep}}^{-3/4} \exp\left({-\frac{2}{3}Re((\frac{-1+c_r}{\ep}+\delta i)^{3/2}e^{i\pi/4})}\right),&\quad \frac{-1+c_r}{\ep} \gtrsim 1. 
\end{array}\right.\nonumber
\end{align}
\end{lemma}

\begin{remark}
If $\al$ is large in the sense that  $\al\gtrsim \nu^{-1/2}$, then by choosing $c_i=-(1-\kappa)\al \nu$, the estimates above can be improved by dropping the $\frac{1}{\lan\al\ep\ran}$ factor. See Appendix \ref{sec:Evans} for more details.
\end{remark}
\subsection{The proof of Theorem \ref{thm:Sepbd}: Initial vorticity supported away from the boundary}
\textbf{Step 1: Estimates via resolvents:} \\ 
For future notational convenience, define
\begin{align}
\Phi_{\pm}(c) = \frac{\alpha}{2\pi} \int_{-1}^1 e^{\pm \alpha y} \mathcal{R}_f(y,c) dy. \label{def:Phi}
\end{align}
Then, from formula \eqref{eq:abdef}, we can write the boundary corrector resolvent in terms of $\Phi_{\pm}$ for suitably defined kernels $K_{\pm}$,  
\begin{align}\label{omega_b_K+K-}
\frac{\alpha}{2\pi} \mathcal{R}_b(\alpha, y, c) = K_-(y,c)  \Phi_{+}(c) + K_+(y,c)  \Phi_{-}(c).  
\end{align}
We will use a variation of the clever idea due to Chen et. al. \cite{ChenEtAl18}, which leverages the observation that one can obtain good estimates on $\Phi_{\pm}$ through $\omega_f$, the free evolution solving \eqref{def:free}. This works even better when the initial data is away from the boundaries. 

First, consider obtaining pointwise-in-time $L^p$ estimates on the vorticity as in \eqref{ineq:SemiGrp}.
Choosing $\gamma = -\lambda$ in \eqref{eq:EvoD}, integrating in $y$, and applying Minkowski's inequality,
there holds (denoting $c_i = -\lambda/\alpha$)
\begin{align}
\norm{e^{\lambda t} \eta_{\eps,p} \widehat{\omega}_b}_{L^p} & \lesssim \sum_{s \in \set{+,-}} \norm{\eta_{\eps,p} \frac{K_s(\cdot,\cdot + i c_i)}{\brak{\cdot + i c_i}}}_{L^p_y L^2_{c_r}} \norm{\brak{\cdot + i c_i} \Phi_{-s}(\cdot + ic_i)}_{L^2_{c_r}}, \label{ineq:evdt1} 
\end{align}
where we use the notation $-s = +$ if $s = -$ and $-s = -$ if $s = +$.
To obtain the inviscid damping estimates on the vorticity in \eqref{ineq:IDv1}, we apply the following analogous estimate: for $p \in [1,2]$ 
\begin{align}
\norm{e^{\lambda t} \eta_{\eps,p} \widehat{\omega}_b}_{L^2_t L^p_y} \lesssim \norm{e^{\lambda t} \eta_{\eps,p} \widehat{\omega}_b}_{L^p_y L^2_t} & \\ & \hspace{-5cm}\lesssim \sum_{s \in \set{+,-}}\norm{\eta_{\eps,p} \frac{K_s(\cdot,\cdot + i c_i )}{\brak{\cdot + i c_i}}}_{L^p_y L^2_{c_r}} \norm{\brak{\cdot + i c_i}\Phi_{-s}(\cdot + ic_i)}_{L^\infty_{c_r}} \label{ineq:evdt2}
\end{align}
and for \eqref{ineq:IDv2} we instead use the slightly different
\begin{align}
\norm{e^{\lambda t} \eta_{\eps,p} \widehat{\omega}_b}_{L^2_t L^p_y} \lesssim \sum_{s \in \set{+,-}}\norm{\eta_{\eps,p} \frac{K_s(\cdot,\cdot + i c_i)}{\brak{\cdot + i c_i}}}_{L^p_y L^\infty_{c_r}} \norm{\brak{\cdot + i c_i}\Phi_{-s}(\cdot + ic_i)}_{L^2_{c_r}}. \label{ineq:evdt3}
\end{align}
For \eqref{ineq:IDv3} we have similarly
\begin{align}
\norm{e^{\lambda t} \eta_{\eps,p} \widehat{\omega}_b}_{L^p_y L^2_t} \lesssim \sum_{s \in \set{+,-}}\norm{\eta_{\eps,p} \frac{K_s(\cdot,\cdot + i c_i)}{\brak{\cdot + i c_i}}}_{L^p_y L^2_{c_r}} \norm{\brak{\cdot + i c_i}\Phi_{-s}(\cdot + ic_i)}_{L^\infty_{c_r}}. \label{ineq:evdt4}
\end{align}
Consider next the inviscid damping estimate \eqref{ineq:IDu}, which  does not follow from any of the estimates on the vorticity due to the logarithmic loss in \eqref{ineq:IDv2}.
To correct this shortcoming, we prove below in Section \ref{sec:IDtrick} that 
\begin{align}
\norm{e^{\lambda t} \widehat{u}_b}_{L^2_t L^2_y} \lesssim \sum_{s \in \set{+,-}}\norm{ \frac{K_{s}(\cdot,\cdot + ic_i)}{\brak{\cdot + i c_i}}}_{L^\infty_{c_r} L^1_y} \norm{\brak{\cdot + i c_i} \Phi_{-s}(\cdot + ic_i)}_{L^2_{c_r}}, \label{ineq:evdt5}
\end{align} 
in particular, the regularization in the Green's function for the Biot-Savart law permits us to reverse the order of integration on $K_{\pm}$.
This is sufficient to remove the logarithmic loss.  

\noindent
\textbf{Step 2: Estimates on $K_{\pm}$:} \\ 
The following lemma summarizes the estimates available on $K_{\pm}$. 
These are obtained via Lemma \ref{lem:ReadableEvans} on the Evans function  and asymptotics for the Airy functions comprising the kernels (in Section \ref{sec:Airy}). 
\begin{lemma}[Estimates on $K_{\pm}$]\label{lem:Kpm}
For $p \in [1,\infty)$ and $\eta_{\ep,p}$ defined as in \eqref{def:eta1} and $c_i = -(1-\kappa)\alpha \nu - \delta \eps$, 
\begin{align}
\norm{\eta_{\eps,p} \frac{K_{\pm}(\cdot,\cdot + i c_i)}{\brak{\cdot + i c_i}}}_{L^p_y L^2_{c_r} } & \lesssim e^{-\alpha},  \label{ineq:Kpm1}
\end{align}
and for $p \in [1,2]$ there holds 
\begin{align}
\norm{\eta_{\eps,p} \frac{K_{\pm}(\cdot,\cdot + i c_i)}{\brak{\cdot + i c_i}}}_{L^p_y L^\infty_{c_r}} & \lesssim  \abs{\ln \eps} e^{-\alpha}, \label{ineq:LpLinfK} \\ 
\norm{\frac{K_{\pm}(\cdot,\cdot + i c_i)}{\brak{\cdot + i c_i}}}_{L^\infty_{c_r} L^1_y} & \lesssim  e^{-\alpha}. \label{ineq:LinfL1K}  
\end{align}
\end{lemma}

\noindent
\textbf{Step 2: Estimates on $\Phi_{\pm}$} \\
For obtaining good estimates of \eqref{def:Phi}, we use a variation of an idea appearing in \cite{ChenEtAl18}.
First notice that by the definition of $\mathcal{R}_f$, we have  
\begin{align}
\widehat{\omega}_f=\frac{\alpha}{2\pi}\int_{-\frac{i\lambda}{\al}-\infty}^{-\frac{i\lambda}{\al}+\infty}e^{-i\al c t}\mathcal{R}_f(y,c)dc.
\end{align}
Therefore, by Plancherel's theorem for the inverse Laplace transform, we have
\begin{subequations} \label{ineq:ILTtrick}
\begin{align}
\norm{ \Phi_{\pm}(\cdot+ ic_i) }^2_{L^2_{c_r}} & \lesssim \alpha\int_0^\infty \abs{\int_{-1}^1 e^{\lambda t \pm \alpha y} \widehat{\omega}_f(t,\alpha,y) dy}^2 dt ,\\ 
\norm{ (\cdot + ic_i)\Phi_{\pm}(\cdot+ ic_i) }^2_{L^2_{c_r}} & \lesssim \frac{1}{\alpha}\int_0^\infty \abs{\partial_t \int_{-1}^1 e^{\lambda t \pm \alpha y} \widehat{\omega}_f(t,\alpha,y) dy}^2 dt ,\\ 
\norm{ \partial_c \Phi_{\pm}(\cdot+ i c_i) }^2_{L^2_{c_r}} & \lesssim \alpha \int_0^\infty \brak{\alpha t}^2 \abs{\int_{-1}^1 e^{\lambda t \pm \alpha y} \widehat{\omega}_f(t,\alpha,y) dy}^2 dt.
\end{align}
\end{subequations}
Note that $\norm{\Phi_{\pm}(\cdot + i c_i )}_{L^\infty_{c_r}} \lesssim \norm{ \Phi_{\pm}(\cdot+ ic_i) }_{L^2_{c_r}} + \norm{ \partial_{c} \Phi_{\pm}(\cdot+ ic_i) }_{L^2_{c_r}}$.  

Due to to the integral in $y$, the right-hand side admits optimal estimates via the inviscid damping of $\omega_f$ and the separation from the boundary. 
If $g(t) = \brak{\phi,\omega_f(t)}_{L^2}$ with $\phi \in H^1$, it is straightforward to check the inviscid damping-type estimate $\norm{g}_{L^2_t} \lesssim \abs{\alpha}^{-1/2}\norm{\omega_{in}}_{L^2}$ (i.e. $\omega_{in} \in L^2$ implies $\omega_f \in L^2_t \dot{H}^{1/2}_x H^{-1}_y$); see \cite{LinZeng11} and Section \ref{sec:Free}. 
However, $\mathbf{1}_{y \in [-1,1]} e^{\pm \alpha y}$ is not $H^1$ at the boundary (nor uniformly bounded in $\alpha$).
To deal with this, we use the initial support to deduce that $\omega_f$ stays away from the boundary except for a very small (in $\nu$) contribution due to the viscous diffusion. 
The boundary contribution is wiped out by enhanced dissipation long before it can have a detrimental effect on the desired estimate. 
The approximate separation similarly also implies the exponential gain in $\alpha$, (i.e the uniform real-analyticity in $x$).  
The following lemma is proved in Section~\ref{sec:FreeRes}. 
\begin{lemma}\label{lem:IDtrick}
  If $\exists \delta_0> 0$ such that $\textup{supp}\,\omega_{in} \subset (-1+\delta_0,1-\delta_0)$ and $\lambda$ is as in \eqref{def:lambda}, then $\exists \delta' > 0$ (depending on $\delta_0$, $\kappa$) such that the following holds
\begin{subequations} \label{ineq:L2estFrees} 
\begin{align}
\int_0^\infty \abs{\int_{-1}^1 e^{\lambda t \pm \alpha y} \widehat{\omega}_f(t,\alpha,y) dy}^2 dt & \lesssim e^{2(1-\delta') \alpha} \norm{\widehat{\omega}_{in}}_{L^2}^2, \label{ineq:L2estFree1} \\ 
  \int_0^\infty \abs{\partial_t \int_{-1}^1 e^{\lambda t \pm \alpha y} \widehat{\omega}_f(t,\alpha,y) dy}^2 dt & \lesssim e^{2(1-\delta') \alpha} \norm{\widehat{\omega}_{in}}_{L^2} ^2 ,\label{ineq:L2estFree2}\\
  \int_0^\infty \brak{\alpha t}^2 \abs{\int_{-1}^1 e^{\lambda t \pm \alpha y} \widehat{\omega}_f(t,\alpha,y) dy}^2 dt & \lesssim e^{2(1-\delta') \alpha} \left(\norm{\widehat{\omega}_{in}}_{L^2}^2 + \norm{\partial_y \widehat{\omega}_{in}}_{L^2}^2\right) \label{ineq:L2estFree3} \\
  \int_0^\infty \abs{\brak{\alpha t}^2 \partial_t \int_{-1}^1 e^{\lambda t \pm \alpha y} \widehat{\omega}_f(t,\alpha,y) dy}^2 dt & \lesssim e^{2(1-\delta') \alpha} \left(\norm{\widehat{\omega}_{in}}_{L^2}^2 + \norm{\partial_y \widehat{\omega}_{in}}_{L^2}^2\right). \label{ineq:L2estFree4}
\end{align}
\end{subequations} 
\end{lemma}

Lemma \ref{lem:IDtrick} together with \eqref{ineq:ILTtrick}, Lemma \ref{lem:Kpm} (and \eqref{ineq:evdt1}, \eqref{ineq:evdt2}, \eqref{ineq:evdt3}, \eqref{ineq:evdt4}, \eqref{ineq:evdt5}) complete the proof of Theorem \ref{thm:Sepbd}. 

\subsection{The proof of Theorem \ref{thm:RegData}: $H^1$ data against the boundary}

This proof proceeds essentially as in Theorem \ref{thm:Sepbd}, however, we replace Lemma \ref{lem:IDtrick} with the following variation. 
The lack of \eqref{ineq:L2estFree3} is precisely why we cannot currently obtain analogues for some of the corresponding estimates in Theorem \ref{thm:Sepbd}. 
The following lemma is proved in Section~\ref{sec:FreeRes}
\begin{lemma}\label{lem:IDtrickReg}
	For all $\omega_{in} \in H^1$ which satisfies the no-slip condition \eqref{eq:bdyc} at $t=0$, there holds 
	\begin{subequations} \label{ineq:L2estFreesReg} 
		\begin{align}
\int_0^\infty \abs{\int_{-1}^1 e^{\lambda t \pm \alpha y} \widehat{\omega}_f(t,\alpha,y) dy}^2 dt & \lesssim \frac{e^{2\al}}{\alpha^2} \left(\norm{\widehat{\omega}_{in}}_{L^2_y}^2 + \norm{\partial_y\widehat{\omega}_{in}}_{L_y^2}^2\right) \label{ineq:L2estFree1Reg}, \\ 
		\int_0^\infty \abs{\partial_t \int_{-1}^1 e^{\lambda t \pm \alpha y} \widehat{\omega}_f(t,\alpha,y) dy}^2 dt & \lesssim e^{2\al}\left(\norm{\alpha \widehat{\omega}_{in}}_{L^2_y}^2 + \norm{\partial_y\widehat{\omega}_{in}}_{L_y^2}^2\right). \label{ineq:L2estFree2Reg} 
		\end{align}
	\end{subequations} 
\end{lemma}

\subsection{The proof of Theorem \ref{thm:NonSepbd}: $L^2$ data against the boundary} \label{sec:L2OrrSom}
In the $L^2$ case, the first few steps are identical to Theorems \ref{thm:Sepbd} and \ref{thm:RegData}. In the case of rough data near the boundary however, the estimates on $\Phi_{\pm}$ seem much less advantageous. Instead, we take an approach more like that of e.g. \cite{GN17} and the references therein and expand the full Green's function of the Orr-Sommerfeld equation explicitly. 
To this end, we apply the Green's function of the Orr-Sommerfeld equation on the whole real line $\rr$ to represent $\mathcal{R}_f$ in terms of $\omega_{in}$ (see e.g. \cite{Wasow53,DR81}).
Applying Fubini's theorem gives an expansion of the form
\begin{align}\label{om_b}
\widehat{\omega}_b(\alpha,y) =& - \int_{\Gamma_\al} \frac{ 2\epsilon^{-2}}{\pi D(\al,c)}\aaa \int_{-1}^1 \int_{-1}^{r}  e^{-i\al c t+\al r}\widehat{\omega}_{in}(\alpha, z) H_+(Z)H_-(R)H_-(Y)dz dr dc\nonumber\\
&+\text{Similar terms}\nonumber\\
=&: -\sum_{j=1}^8\int_{-1}^1 K_j(y,z)\widehat{\omega}_{in}(\alpha, z)dz,
\end{align}
where $\Gamma_\al$ is a suitably chosen contour in the complex plane. The full representation of $\omega_b$ can be found in Appendix \ref{sec:Green}.
Note that the free resolvent introduced additional powers of $\eps^{-2}$ via the Wronksian of $H_{\pm}$; see e.g. \cite{Wasow53,DR81}. 

\begin{figure}\centering
\includegraphics[scale=0.4]{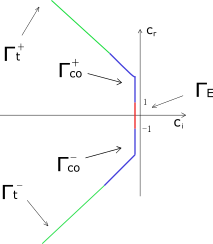} \label{Fig:Contour}
\caption{Contour $\Gamma_\al$ (rotated in the complex plane to match the usual spectral conventions).} 
\end{figure}
To prove estimate \eqref{ineq:thm2}, we show that the operators corresponding to the kernels in the weighted Green's function, $\eta_\ep(y)K_j(y,z)$ as defined in \eqref{om_b} are bounded integral operators $L^p$ to $L^p$ via Schur's test. Specifically, we prove (for $t \geq 1$) 
\begin{subequations}\label{ineq:GenC}
\begin{align}
\sum_{j=1}^8\sup_{y\in[-1,1]}||\eta_\ep(y) K_j(y,\cdot)||_{L_z^1} & \lesssim \lan \al\ep\ran \ln\ep^{-1} e^{-\lambda t},\label{K_1_sup_x_L_1_z_weighted}\\
\sum_{j=1}^8\sup_{z\in[-1,1]}||\eta_\ep(\cdot)K_j(\cdot,z)||_{L_y^1}& \lesssim \lan\al\ep \ran\ln^ {3/2}\ep^{-1}  \ln \ln \ep^{-1} e^{-\lambda t}.\label{K_1_sup_z_L_1_x_weighted}
\end{align}
\end{subequations}
Schur's test and \eqref{om_b} then yield \eqref{ineq:thm2}. Note that the $\lan \al\ep\ran$ factor is dominated by $e^{-\lambda t}$ for $t\geq 1$. Analogous estimates are also made for $t \leq 1$ (see Section \ref{sec:Gen2} for details). 

We decompose the contour integral into two parts: the dominant contribution is close to the continuous spectrum of the linearized Euler equation and the second part is a tail contribution bent into a standard trapezoidal shape. See Figure \ref{Fig:Contour} for a diagram.
The inner contour $\Gamma_E$ is chosen to be the vertical line segment $\{c|c_r\in[-1,1], c_i=-\al\nu-\delta\al^{-1/3}\nu^{1/3}\}$. 
Further from this contribution, we distinguish two cases depending on the distance to the boundary: $d(y,\pm 1)\geq\ep^{3/2}$ and $d(y,\pm 1)\leq \ep^{3/2}$.
In the first case, we use the straight contour $\{c|c_i\equiv -\al \nu-\delta \al^{-1/3}\nu^{1/3}$. 
In the second  case, we further split the contour into two pieces, the connection region $\Gamma_{co}^\pm$ and tail region $\Gamma_t^\pm$.
The straight contour and the regions $\Gamma_{co}^{\pm}$ are treated by a variation of the method used near the continuous spectrum. In the tail region, we use the extra integrability coming from integrating $e^{- \al c_i t}$ in $c_i$ for $t\geq 1$. For $t < 1$, the vanishing of the weight at the boundary allows to minimize the effect of this region for short times. 
See Sections \ref{sec:Gen1} and \ref{sec:Gen2} for details.

\section{Separated from boundary or $H^1$: Estimates on the free resolvent} \label{sec:FreeRes}
\subsection{Proof of Lemma \ref{lem:IDtrick}} 
In this section we prove Lemma \ref{lem:IDtrick}.
Let $\psi_{\pm}$ be a smooth function such that
\begin{align}
\psi_{\pm} =
\begin{cases}
e^{\pm \alpha y} & \quad y \in [-1+\delta_0/8,1-\delta_0/8] \\
0 & \quad y \not\in [-1+\delta_0/16,1-\delta_0/16]. 
\end{cases}
\end{align}
Note  $\forall \delta'$ sufficiently small (depending on $\delta_0$), there holds
\begin{align}
\norm{\partial_y^j \psi_{\pm}}_{L^2} \lesssim_{j,\delta'} e^{(1-\delta')\alpha}. \label{ineq:trivPsi}
\end{align}
Denote
\begin{align}
\widetilde\psi_{\pm} = \psi_{\pm} - e^{\pm \alpha y}\mathbf{1}_{y \in [-1,1]}, 
\end{align}
and note that $\widetilde \psi_{\pm}$ is supported in the complement of the interval $I_{\delta_0} := [-1+\frac{\delta_0}{2}, 1 - \frac{\delta_0}{2}]$.
\textbf{Proof of \eqref{ineq:L2estFree1}:} \\
First, we divide the integral based on the support of the initial data (suppressing the $\alpha$ dependence) 
\begin{align}
\int_{-1}^1 e^{\lambda t} e^{\pm \alpha y} \widehat{\omega}_f(t,y) dy & = \int_{\rr} e^{\lambda t} \psi_{\pm} \widehat{\omega}_f(t,y) dy - \int_{-1}^1 e^{\lambda t} \left(\psi_{\pm} - e^{\pm \alpha y}\mathbf{1}_{y \in [-1,1]}\right) \widehat{\omega}_f(t,y) dy \nonumber \\
& = T_1(t) + T_2(t). \label{ineq:L2Free1_decomp}
\end{align}
By Lemma \ref{lem:FreeDecay} and \eqref{ineq:trivPsi}, for $\delta'$ sufficiently small
\begin{align}
\norm{T_1}_{L^2_t} & \leq \norm{\psi}_{H^1} \norm{e^{\lambda t} \omega_f}_{L^2_t H_y^{-1}} \lesssim e^{(1-\delta')\alpha } \norm{\omega_{in}}_{L^2},
\end{align}
which suffices for the treatment of $T_1$ in \eqref{ineq:L2Free1_decomp}.  

Turn next to $T_2$ in \eqref{ineq:L2Free1_decomp}, where we use that the vorticity is initially zero on the support of the integrand.
First, we observe that $\abs{\omega_f}^2$ is a subsolution of the heat equation,
\begin{align}
\left(\partial_t + 2\nu \alpha^2 - \nu \partial_{yy} \right) \abs{\omega_f}^2 < 0, \label{ineq:subsoln}
\end{align}
which implies by the parabolic comparison principle 
\begin{align}
\abs{\omega_f(t,y)}^2 \leq \frac{1}{(4\pi \nu t)^{1/2}} \int_{\Real} \exp \left(-2\nu \alpha^2 t - \frac{\abs{y - y'}^2}{4\nu t}\right) \abs{\omega_{in}(y')}^2 dy'. \label{ineq:Comp}
\end{align}
By H\"older's inequality and parabolic comparison \eqref{ineq:Comp} we have, 
\begin{align*}
\abs{T_2(t)} 
& \lesssim \abs{\int_{I_{\delta^c_0}} e^{4\lambda t} \abs{\widetilde{\psi}_{\pm}(y)}^4 \abs{\omega_f(t,y)}^2 \frac{1}{(4\pi \nu t)^{1/2}} \int_{\Real} \exp \left(-2\alpha^2\nu  t - \frac{\abs{y - y'}^2}{4\nu t}\right) \abs{\omega_{in}(y')}^2 dy' dy }^{1/4}. 
\end{align*}
Using that $\abs{y-y'} \gtrsim \delta_0$ on the support of the integrand we obtain for some $\delta' > 0$
\begin{align*}
|T_2| & \lesssim_{\delta_0} e^{\alpha - \kappa \alpha^2 \nu t}  e^{-\frac{\delta_0^2}{256\nu t}}  \norm{\widehat{\omega}_{in}}_{L_y^2}^{1/2} \norm{e^{\alpha^2 \nu t + 2\delta \nu^{1/3}\alpha^{2/3}t} \widehat{\omega}_f}_{L_y^2}^{1/2} \\
& \lesssim_{\delta_0} {e^{(1-\delta')\al}}e^{-\frac{\delta_0^2}{600\nu t}} \norm{\omega_{in}}_{L_y^2}^{1/2} \norm{e^{\alpha^2 \nu t + 2\delta \nu^{1/3}\alpha^{2/3}t} \omega_f}_{L_y^2}^{1/2}, 
\end{align*}
where the last line followed  by separately considering the cases $\alpha \nu t < \delta_0$ and $\alpha \nu t \geq \delta_0$. 
Next, we estimate the ${L_t^2}$ norm of $T_2$ in \eqref{ineq:L2Free1_decomp}. 
To this end, we decompose the integral into two regimes: $t\in[0,\nu^{-1/2}]$ and $t\in[\nu^{-1/2}, \infty)$. For $t \in [0,\nu^{-1/2})$, we use that for any $p,b >0 $,  $e^{-\frac{b}{\nu t}}\lesssim_{b,p} \nu^p$. Therefore, by Lemma \ref{lem:FreeDecay} we have the following $\forall p > 0$, 
\begin{align}
\norm{T_2}_{L^2_t}^2 
& \lesssim e^{2(1-\delta')\al}\left(\int_0^{\nu^{-1/2}} + \int_{\nu^{-1/2}}^\infty\right) e^{-\frac{\delta_0^2}{300\nu t}} e^{- \frac{1}{100} \nu \alpha^{2} t^3} \norm{\omega_{in}}_{L_y^2}^2 dt  \nonumber\\
& \lesssim_p {e^{2(1-\delta')\al}}\nu^p \norm{\omega_{in}}_{L_y^2}^2. \label{T_2_L_t2}
\end{align}
Note we also used that for all $\gamma > 0$, $\exists C_\gamma > 0$ such that $\nu^{1/3} \alpha^{2/3} t \leq C_\gamma + \gamma \nu \alpha^2 t^3$. 
This completes the proof of \eqref{ineq:L2estFree1}. 

\textbf{Proof of \eqref{ineq:L2estFree2}:}\\ 
The time derivative estimate is slightly more complicated than \eqref{ineq:L2estFree1}. We begin as before, 
\begin{align}
\frac{d}{dt} &\int_{-1}^1 e^{\lambda t \pm \alpha y} \widehat{\omega}_f(t,y) dy\nonumber\\
& = -\int_{\Real} e^{\lambda t} \psi_\pm i\alpha y \widehat{\omega}_f(t,y) dy +\int_{\Real} e^{\lambda t} \psi_\pm \nu \Delta_\alpha \widehat{\omega}_f(t,y) dy -\int_{-1}^1 e^{\lambda t} \widetilde{\psi}_\pm i\al y\widehat{\omega}_f(t,y) dy  \nonumber\\
& \quad+\int_{-1}^1 e^{\lambda t} \widetilde{\psi}_\pm \nu \Delta_\alpha \widehat{\omega}_f(t,y)+\lambda\int_{-1}^1 e^{\lambda t \pm \alpha y} \widehat{\omega}_f(t,y) dy  dy \nonumber\\
& = :\sum_{j=1}^5 T_j(t).\label{T_j} 
\end{align}
For the $T_5$ term, we note that it is $(1-\kappa)\al^2\nu+\delta \al^{2/3}\nu^{1/3}$ times the integral in \eqref{ineq:L2Free1_decomp}, so by adjusting $\delta'$, $||T_5||_{L^2_t}\lesssim e^{(1-\delta')\al}||\omega_{in}||_2$.
Terms $T_1$ and $T_3$ in \eqref{T_j} are treated as in \eqref{ineq:L2Free1_decomp} in the proof of \eqref{ineq:L2estFree1}. 
Next, consider the treatment of the terms $T_2$ and $T_4$.
For $T_2$ we first integrate by parts in $y$ and then we may treat as in the proof of \eqref{ineq:L2estFree1}. Consider next $T_4$ in \eqref{T_j}. We cannot integrate by parts in $y$ due to the boundary.
Instead we need to use parabolic smoothing to estimate $\Delta_\alpha \widehat{\omega}_f$ and use the boundary separation to avoid losing regularity.
We prove that for $0 < \delta_j = \delta_0 2^{-j}$ and $I_{\delta_j} = [-1 + \delta_j/2,1 - \delta_j/2]$, 
\begin{align}
e^{\lambda t}\norm{\mathbf{1}_{I_{\delta_j}^c} \partial_{y}^j \widehat{\omega}_f(t)}_{L_y^2}  \lesssim  e^{-\kappa\alpha^2 \nu t} e^{-\frac{\delta_j^2}{600^j \nu t}} e^{-\frac{1}{100} \nu \alpha^{2} t^3} \norm{\omega_{in}}_{L^2_y}. \label{ineq:dyjest}
\end{align}
Inequality \eqref{ineq:dyjest} is a consequence of the following general estimate which follows from the heat kernel: for all $\gamma > 0$, 
\begin{align}
\norm{\mathbf{1}_{I_{\gamma/2}^c} \partial_y e^{\nu t \Delta_\alpha} \mathbf{1}_{I_{\gamma}} f}_{L^2} \lesssim  e^{-\frac{\gamma^2}{100\nu t} - \nu \alpha^2 t } \norm{f}_{L^2_y}. \label{ineq:SepSup}
\end{align}
and Duhamel's formula, 
\begin{align}
\widehat{\omega}_f(t) = e^{\nu t \Delta_\alpha/2}\widehat{\omega}_f(t/2) - i \alpha \int_{t/2}^t e^{\nu(t-\tau)\Delta_\alpha }  \left( y \widehat{\omega}_f(\tau) \right) d\tau. 
\end{align}
Indeed, consider
\begin{align*}
\mathbf{1}_{I_{\delta_1}^c} \partial_y\widehat{\omega}_f(t) & =  \mathbf{1}_{I_{\delta_1}^c} \partial_y e^{\nu t \Delta_\alpha/2}\left( (\mathbf{1}_{I_{\delta_0}^c} + \mathbf{1}_{I_{\delta_0}})\widehat{\omega}_f(t/2)\right) \\
& \quad - i \alpha \int_{t/2}^t \mathbf{1}_{I_{\delta_1}^c} \partial_y e^{\nu(t-\tau)\Delta_\alpha }  \left( (\mathbf{1}_{I_{\delta_0}^c} + \mathbf{1}_{I_{\delta_0}}) y\widehat{\omega}_f(\tau) \right) d\tau \\
& = I_{c} + I_{o} + II_{c} + II_{o}. 
\end{align*}
Analogous to the proof of \eqref{ineq:L2estFree1}, estimates on $I_o$ follow from \eqref{ineq:SepSup} and estimates on $I_c$ follow by the comparison principle \eqref{ineq:Comp}.
These imply
\begin{align}
e^{2 \lambda t}\norm{I_{c} + I_{o}}_{L_y^2}^2 \lesssim_{\delta_0} e^{\alpha -\frac{\delta^2_0}{400 \nu t} - \kappa \nu \alpha^2 t } \norm{\widehat{\omega}_{in}}_{L_y^2} \norm{e^{\frac{1}{2}\alpha^2 \nu t - 2\delta \nu^{1/3} \alpha^{2/3} t} \widehat{\omega}_{f}\left(\frac{t}{2}\right)}_{L_y^2}. 
\end{align}
This is consistent with \eqref{ineq:dyjest} as in the proof of \eqref{T_2_L_t2} above by the exact formula in Appendix \ref{sec:Free}.  
Applying the same arguments for each $\tau$ in the integral in $II_{c}$ and $II_{o}$ yields \eqref{ineq:dyjest} for $j=1$.
The case $j>1$ follows by iteration. 

With \eqref{ineq:dyjest} proved, now consider the estimate of $T_4$ in \eqref{T_j}.
We apply an argument similar to that used to treat \eqref{T_2_L_t2}.
By \eqref{ineq:dyjest} (and again applying estimates based on \eqref{ineq:Comp} for the first term), 
\begin{align}
\abs{T_4} 
& \lesssim e^{\alpha} e^{\lambda t} \nu \left(\alpha^2 \norm{ \mathbf{1}_{I_{\delta_2}^c} \widehat{\omega}_f}_{L^2} +  \norm{\mathbf{1}_{I_{\delta_2}^c} \partial_{yy} \widehat{\omega}_f}_{L^2} \right)\nonumber \\ 
& \lesssim \nu e^{(1-\delta')\alpha} e^{-\frac{\delta_2}{600^2 \nu t}} e^{-\frac{1}{100} \nu \alpha^{2} t^3} \left(1 +  \alpha^2\right) \norm{\omega_{in}}_{L^2_y}. \label{ineq:dyjest2}
\end{align}
Therefore, similar to the argument in \eqref{T_2_L_t2} we obtain
\begin{align*}
\norm{T_4}_{L^2_t}^2 & \lesssim \nu e^{2(1-\delta') \al}\norm{\widehat{\omega}_{in}}_{L_y^2}^2.
\end{align*}
This completes the proof of \eqref{ineq:L2estFree2}.

\textbf{Proof of \eqref{ineq:L2estFree3}:}\\
We proceed as in the proof of \eqref{ineq:L2estFree1},
\begin{align}
\int_{-1}^1 \brak{\alpha t} e^{\lambda t} e^{\pm \alpha y} \widehat{\omega}_f(t,y) dy & = \int_{\rr}e^{\lambda t} \brak{\alpha t} \psi_{\pm} \widehat{\omega}_f(t,y) dy - \int_{-1}^1 e^{\lambda t} \brak{\alpha t} \widetilde{\psi}_{\pm} \widehat{\omega}_f(t,y) dy \nonumber \\
& = T_1(t) + T_2(t). \label{ineq:L2Free1_decomp3}
\end{align}
The term $T_2$ can be treated as in \eqref{ineq:L2estFree1}, as the separation of support gains arbitrary powers of $t$.
Consider next $T_1$. By Fourier transforming in $\Real$ (denoting $y \mapsto \eta$), Plancherel's theorem, shifting the integral in $\eta$, 
\begin{align*}
\norm{T_1}_{L^2_t} & = \left(\int_0^\infty \abs{ \int_{-\infty}^\infty  e^{\lambda t} \brak{\alpha t} \widehat{\psi_{\pm}}(\eta) \widehat{\omega}_f(t,\eta) d\eta }^2 dt \right)^{1/2} \\
& \lesssim \norm{\psi_{\pm}}_{H^2} \left( \int_0^\infty \int_{-\infty}^\infty  e^{2\lambda t} \frac{\brak{\alpha t}^2}{ \brak{\eta}^{4}}  \abs{\widehat{\omega}_f(t,\eta)}^2 d\eta dt \right)^{1/2} \\
& \lesssim_{\delta_0} e^{(1-\delta')\alpha} \left(\int_0^\infty \int_{-\infty}^\infty  \frac{\brak{\alpha t}^2}{\brak{\eta}^2 \brak{\eta-\alpha t}^{4}}  e^{2\lambda t} \brak{\eta}^2\abs{\widehat{\omega}_f(t,\eta-\alpha t)}^2 d\eta dt \right)^{1/2}.
\end{align*}
Therefore, using the exact formula for the free evolution, (see \eqref{ineq:ffree}), 
\begin{align*}
\norm{T_1}_{L^2_t} & \lesssim_{\delta_0} e^{(1-\delta')\alpha} \left(\int_{-\infty}^\infty \left(\int_0^\infty\frac{\brak{\alpha t}^2}{\brak{\eta}^2 \brak{\eta-\alpha t}^{4}}  dt\right)  \brak{\eta}^2\abs{\omega_{in}(\eta)}^2 d\eta \right)^{1/2} \\
& \lesssim_{\delta_0} e^{(1-\delta')\alpha} \norm{\omega_{in}}_{H^1}. 
\end{align*}

\textbf{Proof of \eqref{ineq:L2estFree4}:}\\
We proceed as in the previous estimates,
\begin{align}
  \brak{\alpha t}  \partial_t \int_{-1}^1 e^{\lambda t \pm \alpha y} \widehat{\omega}_f(t,\alpha,y) dy & = -\brak{\alpha t} \int_{\Real} e^{\lambda t} \psi_\pm i\alpha y \widehat{\omega}_f(t,y) dy +\brak{\alpha t} \int_{\Real} e^{\lambda t} \psi_\pm \nu \Delta_\alpha \widehat{\omega}_f(t,y) dy \nonumber \\ & \quad  - \brak{\alpha t}\int_{-1}^1 e^{\lambda t} \widetilde{\psi}_\pm i\al y\widehat{\omega}_f(t,y) dy  \nonumber\\
& \quad+\brak{\alpha t}\int_{-1}^1 e^{\lambda t} \widetilde{\psi}_\pm \nu \Delta_\alpha \widehat{\omega}_f(t,y)+\lambda\brak{\alpha t} \int_{-1}^1 e^{\lambda t \pm \alpha y} \widehat{\omega}_f(t,y) dy  dy \nonumber\\
& = :\sum_{j=1}^5 T_j(t).
\end{align}
The treatment of $T_1$ and $T_5$ are essentially the same as \eqref{ineq:L2estFree3} (after adjusting $\delta'$). The treatment of $T_3$ and $T_4$ follow as in \eqref{ineq:L2estFree2} as the separation from the boundary gains arbitrary powers of $t$. 
The term $T_2$ follows as in \eqref{ineq:L2estFree3} after integrating by parts in $y$ (and adjusting $\delta'$). 
We omit the repetitive details for brevity. 

\subsection{Proof of Lemma \ref{lem:IDtrickReg}}
\textbf{Proof of \eqref{ineq:L2estFree1Reg}:} \\ 
Denote the \emph{profile} $\widehat{g}_f = e^{i\alpha y t} \widehat{\omega}_f(t,\alpha,y)$, which satisfies
\begin{align}
\partial_t \widehat{g}_f = \nu \left( (\partial_{yy} - i\alpha t)^2 -  \alpha^2 \right) \widehat{g}_f =: \nu \Delta_{\alpha,t} \widehat{g}_f. \label{def:Datg}
\end{align}
To see \eqref{ineq:L2estFree1Reg}, first notice that 
\begin{align}
\int_{-1}^1 e^{\lambda t \pm \alpha y} \widehat{\omega}_f(t,\alpha,y) dy = \int_{-1}^1 e^{\lambda t \pm \alpha y - i\alpha y t} \widehat{g}_f(t,\alpha,y) dy, 
\end{align}
is suitably bounded for $t \lesssim 1$, so without loss of generality, assume $t > 1$. 
Integrating by parts in $y$
\begin{align*}
\int_{-1}^1 e^{\lambda t \pm \alpha y - i\alpha y t} \widehat{g}_f(t,\alpha,y) dy & = \frac{e^{\lambda t \pm \alpha y - i \alpha y t}}{(\pm\alpha-i\alpha t)} \widehat{g}_f(t,\alpha,y)\bigg|_{y = \pm 1} \\ & \quad - \frac{1}{(\pm \alpha - i\alpha t)}\int_{-1}^1 e^{\lambda t \pm \alpha y - i\alpha y t} \partial_y \widehat{g}_f(t,\alpha,y) dy. 
\end{align*}
We therefore have by Sobolev embedding, 
\begin{align}
\abs{\int_{-1}^1 e^{\lambda t \pm \alpha y} \widehat{\omega}_f(t,\alpha,y) dy} \lesssim \frac{e^{\alpha}}{\brak{\alpha t}} \left(\norm{e^{\lambda t}\widehat{g}_f(t,\alpha,\cdot)}_{L^2_y} + \norm{e^{\lambda t}\partial_y \widehat{g}_f(t,\alpha,\cdot)}_{L^2_y} \right)
\end{align}
By the exact formula from Appendix \ref{sec:Free} we see $\norm{e^{\lambda t}\partial_y \widehat{g}_f(t,\alpha,\cdot)}_{L^2_y} \lesssim \norm{\partial_y\widehat{\omega}_{in}(\alpha,\cdot)}_{L^2_y}$, which implies \eqref{ineq:L2estFree1Reg}.

\textbf{Proof of \eqref{ineq:L2estFree2Reg}:} \\ 
The proof of \eqref{ineq:L2estFree2Reg} is a little more complicated. We have
\begin{align}
\frac{d}{dt}\int_{-1}^1 e^{\lambda t \pm \alpha y} \widehat{\omega}_f(t,\alpha,y) dy&  = \lambda \int_{-1}^1 e^{\lambda t \pm \alpha y - i\alpha y t} \widehat{g}_f(t,\alpha,y) dy \nonumber \\
& \quad - \int_{-1}^1 e^{\lambda t \pm \alpha y} (i\alpha y) \widehat{\omega}_f(t,\alpha,y) dy - \int_{-1}^1 e^{\lambda t \pm \alpha y} \nu\Delta_\alpha \widehat{\omega}_f(t,\alpha,y) dy.\label{eq:pelican} 
\end{align}
The second term is treated as in \eqref{ineq:L2estFree1Reg}; indeed note that note that for $t > 1$ we have 
\begin{align}
 \abs{\int_{-1}^1 e^{\lambda t \pm \alpha y} (i\alpha y) \widehat{\omega}_f(t,\alpha,y) dy} \lesssim \frac{e^{\alpha}}{\brak{t}} \left(\norm{\widehat{\omega}_{in}}_{L^2} + \norm{\partial_y \widehat{\omega}_{in}}_{L^2} \right). 
\end{align}
To treat the first term in \eqref{eq:pelican}, first note that by maximal regularity, (note $\alpha^{2/3}\nu^{1/3} \leq C_\kappa + \frac{\kappa}{4} \alpha^2 \nu$) 
\begin{align}
\int_0^1 \abs{\lambda \int_{-1}^1 e^{\lambda t \pm \alpha y} \widehat{\omega}_f(t,\alpha,y) dy}^2 dt  & \lesssim_\kappa \int_0^1 \abs{ \left(1 + (1-\frac{1}{2}\kappa)\alpha^2\nu\right) \int_{-1}^1 e^{\lambda t \pm \alpha y} \widehat{\omega}_f(t,\alpha,y) dy}^2 dt \\ & \lesssim  e^{\alpha} \left(\norm{\alpha \widehat{\omega}_{in}}_{L^2}^2  + \norm{\partial_y\widehat{\omega}_{in}}_{L^2}^2 \right). 
\end{align}
For $t > 1$, we can treat the first term as in \eqref{ineq:L2estFree1Reg} above (using the decay and parabolic regularization in Appendix \ref{sec:Free} to absorb the $\lambda$ in front). 

Finally, consider the last term in \eqref{eq:pelican}. 
By $\kappa > 0$ and maximal regularity we have
\begin{align*}
\int_{0}^1 \abs{\int_{-1}^1 e^{\lambda t \pm \alpha y} \nu\Delta_\alpha \widehat{\omega}_f(t,\alpha,y) dy}^2 dt \lesssim e^{\alpha} \nu \norm{e^{\lambda t} \Delta_\alpha \widehat{\omega}_f}_{L^2_t(0,1;L^2_y)}^2 \lesssim e^{\alpha} \left(\norm{\alpha \widehat{\omega}_{in}}_{L^2}^2  + \norm{\partial_y\widehat{\omega}_{in}}_{L^2}^2 \right). 
\end{align*}
Consider then $t > 1$ and recall \eqref{def:Datg}, 
\begin{align}
\int_{-1}^1 e^{\lambda t \pm \alpha y} \nu\Delta_\alpha \widehat{\omega}_f(t,\alpha,y) dy= \int_{-1}^1 e^{\lambda t \pm \alpha y - i\alpha yt}  \nu \Delta_{\alpha,t} \widehat{g}_f dy. \label{eq:gomf}
\end{align}
By the exact formula in Appendix \ref{sec:Free} and applying the Fourier transform in both variables: for some $\delta'' > 0$ there holds 
\begin{align}
e^{\lambda t} \nu \left( \abs{\eta-\alpha t}^2 + \alpha^2  \right) \widehat{g}_f(t,\alpha,\eta) & \leq \nu \left( \abs{\eta-\alpha t}^2 + \alpha^2  \right)\exp\left(- \delta'' \nu \alpha^2 t^3 - \delta''\nu \eta^2 t \right) \abs{\widehat{\omega}_{in}(\alpha,\eta)} \nonumber\\
& \lesssim \frac{\eta^2 + \alpha^2 t^2+\al^2}{\eta^2 t + \alpha^2 t^3} \abs{\widehat{\omega}_{in}(\alpha,\eta)}, 
\end{align}
and hence for $t > 1$ we have by \eqref{eq:gomf}, 
\begin{align}
\abs{\int_{-1}^1 e^{\lambda t \pm \alpha y} \nu\Delta_\alpha \widehat{\omega}_f(t,\alpha,y) dy} \lesssim \frac{e^\alpha}{t} \norm{\omega_{in}}_{L^2}, 
\end{align}
which completes the proof of \eqref{ineq:L2estFree2Reg}.

\section{Separated from boundary or $H^1$: Estimates on the kernels} \label{sec:Ksep} 
In this section, we prove Lemma \ref{lem:Kpm}. 
Recall from \eqref{omega_b_K+K-} the kernels 
\begin{align}
K_+(Y,c)=&-\frac{1}{D(\al,c)}\left(\int_{-1}^1 e^{-\al w}H_+(W)dw\right)H_-(Y)+\frac{1}{D(\al,c)}\left(\int_{-1}^1 e^{-\al w}H_-(W)dw\right)H_+(Y)\nonumber\\
=&K_{+,1}(Y,c)+K_{+,2}(Y,c)\label{K_+}\\
K_-(Y,c)=&\frac{1}{D(\al,c)}\left(\int_{-1}^1e^{\al w}H_+(W)dw\right)H_-(Y)-\frac{1}{D(\al,c )}\left(\int_{-1}^1e^{\al w}H_-(W)d w\right)H_+(Y)\nonumber\\
=&K_{-,1}(Y,c)+K_{-,2}(Y,c).\label{K_-}
\end{align} 
Recall also the weight defined in \eqref{def:eta1}. 

\subsection{Proof of estimate \eqref{ineq:Kpm1}} \label{sec:kpm1pf}
We will separate low and high frequencies.
Define $C_0 = \left(\frac{2 \delta}{\kappa}\right)^{3/4}$; the low frequency case is $\alpha^2 \nu \leq C_0$ and the high frequency case is $\alpha^2 \nu > C_0$.

\textbf{The low frequency case $\alpha^2 \nu \leq C_0$:}  \\ 
By Lemma \ref{lem:Evans_function_est}, the following holds (recall here $c = c_r + ic_i$ with $c_i = -\lambda/\alpha$), 
\begin{align}
&\norm{\frac{{\eta_{\ep,p}}}{\lan c\ran}K_{+,1}}^p_{L_y^pL_{c_r}^2}\nonumber\\
 &\lesssim\int_{-1}^1\eta_{\ep,p}^p\left(\int_{-\infty}^\infty \frac{1}{\lan c\ran^{2}\ep^{2}e^{2\al}} \frac{|H_-(Y)|^2}{\big|A_0(\frac{-1-c_r}{\ep}+\delta i)\big|^2}dc_r\right)^{p/2} dy\nonumber\\
 &\lesssim \left(\int_{-1}^{-1+\ep^{3/2}}+\int_{-1+\ep^{3/2}}^{-1+\ep}+\int_{-1+\ep}^1\right)\eta_{\ep,p}^p\left(\int_{-1}^\infty \mathbf{1}_{y\leq c_r}\frac{1}{\lan c\ran^{2}\ep^{2}e^{2\al}} \frac{|H_-(Y)|^2}{\big|A_0(\frac{-1-c_r}{\ep}+\delta i)\big|^2}dc_r\right)^{p/2} dy\nonumber\\
&\quad +\left(\int_{-1}^{-1+\ep^{3/2}}+\int_{-1+\ep^{3/2}}^{-1+\ep}+\int_{-1+\ep}^1\right)\eta_{\ep,p}^{p}\left(\int_{-\infty }^1\mathbf{1}_{y\geq c_r}\frac{1}{\lan c\ran^{2}\ep^{2}e^{2\al}} \frac{|H_-(Y)|^2}{\big|A_0(\frac{-1-c_r}{\ep}+\delta i)\big|^2}dc_r\right)^{p/2} dy\nonumber\\
&=: e^{-p\al}  (T_{+,1}+T_{+,2}+T_{+,3}+T_{-,1}+T_{-,2}+T_{-,3}).\label{T+123T-123}
\end{align}
We estimate the first term $T_{+,1}$ in \eqref{T+123T-123}. We distinguish between two possible cases: $p\in[1,2]$ and $p\in[2,\infty)$. In the first case $p\in[1,2]$, we estimate using the definition of the weight $\eta_{\ep,p}\equiv \ep^{\frac{3}{2}-\frac{3}{2p}}$ in the region $y+1\in[0,\ep^{\frac{3}{2}}]$  \eqref{def:eta1}, the $H_\pm$ estimates \eqref{estimate_of_H_pm}, the $A_0$ estimate \eqref{A_0_lower_bound} and Lemma \ref{Gain_in_poly_power}, we have that
\begin{align*}
T_{+,1}\lesssim &\left(\int_{-1}^{-1+\ep^{3/2}} dy\right)^{1-\frac{p}{2}}\left(\int_{-1}^{-1+\ep^{3/2}} \eta_{\ep,p}^{2}\int_{-1}^\infty \mathbf{1}_{Y_r\leq 0}\frac{1}{\lan c\ran^{2}\ep^{2}} \frac{|H_-(Y)|^2}{\big|A_0(\frac{-1-c_r}{\ep}+\delta i)\big|^2}dc_r dy\right)^{p/2}\\
\lesssim &\ep^{\frac{3}{4}p}\left(\int_{-1}^\infty\frac{1}{\ep\lan c\ran^{2}}\int_{\frac{-1-c_r}{\ep}}^{\min\{\frac{1-c_r}{\ep},0\}} \frac{\lan \frac{-1-c_r}{\ep}+\delta i\ran^{3/2}}{\lan Y\ran^{1/2}} \mathbf{1}_{Y_r\leq 0} e^{\frac{4}{3}Re(Y^{3/2}e^{i\pi/4}-(\frac{-1-c_r}{\ep}+\delta i)^{3/2}e^{i\pi/4})}dY_r dc_r\right)^{p/2}.
\end{align*}
Now we apply the fact that $|Y_r|\leq \big|\frac{-1-c_r}{\ep}\big|$, $\mathrm{ph}(Y)\in[\pi/2,\pi)$ to bound the exponential factor, and hence
\begin{align}
T_{+,1}\lesssim&\ep^{\frac{3}{4}p}\left(\int_{-1}^\infty\frac{1}{\ep\lan c\ran^{2}} \left\lan \frac{-1-c_r}{\ep}\right\ran^{1/2} dc_r\right)^{p/2}
\lesssim1.\label{T+1leq2}
\end{align}
For the $p\in[2,\infty)$ case, we use H\"older inequality's and the definition of $\eta_{\ep,p}^p\equiv\ep^{\frac{3}{2}p-\frac{3}{2}}$ for $y+1\in[0,\ep^{3/2})$ \eqref{def:eta1} to obtain (denoting $\theta=\frac{1}{2}+\frac{3}{4p}$), 
\begin{align*}
T_{+,1}\lesssim &\ep^{\frac{3}{2}p-\frac{3}{2}}\left(\int_{-\infty}^\infty\frac{1}{\lan c\ran^{\frac{2(1-\theta) p}{p-2}}}dc_r\right)^{\frac{p-2}{2}}\int_{-1}^\infty\int_{\frac{-1-c_r}{\ep}}^{\frac{1-c_r}{\ep}}\mathbf{1}_{Y_r\leq 0}\frac{|H_-(Y)|^p}{\ep^{p-1}\lan c\ran^{\theta p}\big|A_0(\frac{-1-c_r}{\ep}+\delta i)\big|^p}dY_r dc_r.
\end{align*}
Combining the $H_\pm$ estimates \eqref{estimate_of_H_pm}, the $A_0$ estimate \eqref{A_0_lower_bound} and Lemma \ref{Gain_in_poly_power}, and applying an argument similar to that used in the $p\in[1,2]$ case,  we obtain
\begin{align}
T_{+,1}\lesssim & \ep^{\frac{3}{2}p-\frac{3}{2}}\int_{-1}^\infty\frac{1}{\ep^{p-1}\lan c\ran^{\theta p}}\int_{\frac{-1-c_r}{\ep}}^{\min\{\frac{1-c_r}{\ep},0\}} \frac{\lan \frac{-1-c_r}{\ep}+\delta i\ran^{3p/4}}{\lan Y\ran^{p/4}}e^{\frac{2p}{3}Re(Y^{3/2}e^{i\pi/4}-(\frac{-1-c_r}{\ep}+\delta i)^{3/2}e^{i\pi/4})}dY_r dc_r\nonumber\\
\lesssim&\ep^{\frac{3}{2}p-\frac{3}{2}}\int_{-1}^\infty\frac{1}{\ep^{p-1}\lan c\ran^{\theta p}} \left\lan \frac{-1-c_r}{\ep}+\delta i\right\ran^{p/2-1/2}
dc_r\nonumber\\
\lesssim&\int_{-1}^\infty\frac{1}{\lan c\ran^{(\theta -1/2)p+1/2}}dc_r\lesssim 1.\label{T+1}
\end{align}
Next we estimate the $T_{+,2}$ term in \eqref{T+123T-123}. Combining the $H_\pm$ estimates \eqref{estimate_of_H_pm}, the $A_0$ estimate \eqref{A_0_lower_bound} and Lemma \ref{Gain_in_poly_power}, we have that
\begin{align*}
T_{+,2}\lesssim &\int_{-1+\ep^{3/2}}^{-1+\ep} \eta_{\ep,p}^{p}\left(\int_{-1}^\infty\frac{\mathbf{1}_{Y_r\leq 0}}{\ep^{2}\lan c\ran^{2}} \frac{\lan \frac{-1-c_r}{\ep}+\delta i\ran^{3/2}}{\lan Y\ran^{1/2}}e^{\frac{4}{3}Re(Y^{3/2}e^{i\pi/4}-(\frac{-1-c_r}{\ep}+\delta i)^{3/2}e^{i\pi/4})}dc_r\right)^{p/2}dy.
\end{align*}
Next we need to gain from the exponential factor. On the support of the integrand, we have the following inequality (proved, for example, by differentiation), 
\begin{align}\label{Com_monotone}
Re((Y_r+\delta i)^{3/2}e^{i\pi/4})\leq 5\delta\mathbf{1}_{|-1-c_r|\ep^{-1}\leq 5\delta}+Re\left((Y_r+\delta i)\bigg(\frac{-1-c_r}{\ep}+\delta i\bigg)^{1/2}e^{i\pi/4}\right). 
\end{align}
Applying this monotonicity relation, there exists a universal small constant $C$ such that the following holds 
\begin{align*}
T_{+,2}\lesssim&\int_{-1+\ep^{3/2}}^{-1+\ep}\eta_{\ep,p}^{p}\bigg(\int_{-1}^\infty\frac{1}{{\ep^2\lan c\ran ^{2}}}\left\lan\frac{-1-c_r}{\ep}+\delta i\right\ran \nonumber\\
&\quad\quad\times e^{\frac{{4}}{3}Re[(-(\frac{-c_r-1}{\ep}+\delta i)^{3/2}e^{i\pi/4}+(\frac{ -c_r-1}{\ep}+\frac{1+y}{\ep}+\delta i)( \frac{-c_r-1}{\ep}+\delta i)^{1/2}e^{i\pi/4}]} d{ c_r} \bigg)^{p/2}dy\nonumber\\
\lesssim &\int_{-1+\ep^{3/2}}^{-1+\ep}\eta_{\ep,p}^{p}\left(\int_{-1}^\infty\frac{1}{\ep^{2}\lan c\ran ^{2}}\left\lan\frac{-1-c_r}{\ep}+\delta i\right\ran e^{\frac{4}{3}\frac{y+1}{\ep}Re( (\frac{-1-c_r}{\ep}+\delta i)^{1/2}e^{i\pi/4})}dc_r\right)^{p/2}dy\nonumber\\
\lesssim &\int_{-1+\ep^{3/2}}^{-1+\ep}\eta_{\ep,p}^p\left(\int_{-1}^\infty\frac{1}{\ep}\left\lan\frac{(1+c_r)(y+1)^2}{\ep^3}\right\ran e^{-C^{-1}\big|\frac{(1+c_r)(y+1)^2}{\ep^3}\big|^{1/2}}\frac{\ep}{(y+1)^2} d c_r\right)^{p/2} dy. 
\end{align*} 
Changing variables  $y \mapsto u :=\frac{(c_r+1)(y+1)^2}{\ep^3}$ (recall definition \eqref{def:eta1}), we obtain
\begin{align}\label{T2+est}
T_{2,+}\lesssim &\int_{-1+\ep^{3/2}}^{-1+\ep}\eta_{\ep,p}^p\frac{\ep^{\frac{3}{2}p}}{(y+1)^{2p}}dy \lesssim 1.
\end{align}

Next we estimate the $T_{+,3}$ term in \eqref{T+123T-123}. Combining the definition \eqref{def:eta1}, the $H_\pm$ estimates \eqref{estimate_of_H_pm}, the $A_0$ estimate \eqref{A_0_lower_bound}, Lemma \ref{Gain_in_poly_power} and the monotonicity relation \eqref{Com_monotone} (which holds on the support of the integrand), with an argument similar to that used to estimate $T_{+,2}$ implies there exists a universal constant $C$ such that 
\begin{align*}
T_{+,3}\lesssim &\int_{-1+\ep}^{1}\eta_{\ep,p}^{p}\left(\int_{y}^\infty\frac{1}{\ep^{2}\lan c\ran^{2}} \frac{\lan \frac{-1-c_r}{\ep}+\delta i\ran^{3/2}}{\lan Y\ran^{1/2}}e^{\frac{4}{3}Re(Y^{3/2}e^{i\pi/4}-(\frac{-1-c_r}{\ep}+\delta i)^{3/2}e^{i\pi/4})}dc_r\right)^{p/2}dy \nonumber\\
\lesssim&\int_{-1+\ep}^1\ep^{-1+\frac{p}{2}}e^{pc_\star \frac{(y+1)^{3/2}}{\ep^{3/2}}}\bigg(\int_{y}^\infty\frac{1}{{\ep^2\lan c\ran ^{2}}}\left\lan\frac{-1-c_r}{\ep}+\delta i \right\ran ^{3/2}e^{-C^{-1}|\frac{y+1}{\ep}||\frac{-1-c_r}{\ep}|^{1/2}} d{ c_r} \bigg)^{p/2}dy. 
\end{align*}
Observe that the constraint $c_r\geq y$ yields the relation $|-1-c_r|\geq y+1$. Combining this fact and $y+1\geq \ep$, we choose $c_\star$ small enough in the weight $\eta_{\ep,p}$ to obtain that
\begin{align*}
T_{+,3}\lesssim &\int_{-1+\ep}^{1}\ep^{-1+\frac{p}{2}}e^{pc_\star \frac{(y+1)^{3/2}}{\ep^{3/2}}}\left(\int_{y}^\infty\frac{1}{\ep^{2}\lan c\ran ^{2}}\left\lan\frac{-1-c_r}{\ep}+\delta i\right\ran^{3/2} e^{-\frac{C^{-1}}{2}|\frac{y+1}{\ep}|^{3/2} -\frac{C^{-1}}{2}|\frac{-1-c_r}{\ep}|^{1/2}}dc_r\right)^{p/2}dy\nonumber\\
\lesssim &\ep^{-1}\int_{-1+\ep}^1e^{-p(C^{-1}/2-c_\star)(\frac{1+y}{\ep})^{\frac{3}{2}}} dy\lesssim 1.
\end{align*}
This completes the estimation of all the $T_{+,\cdot}$ terms in \eqref{T+123T-123}.
 
Next we estimate the $T_{-,1}$ term in \eqref{T+123T-123}. As in the treatment of the $T_{+,1}$ term in \eqref{T+123T-123}, we distinguish two cases:  $p\in[2,\infty)$ or $p\in[1,2]$. In the case $p\in[2,\infty)$, applying H\"older's inequality and Fubini's theorem,  we obtain (denoting $\theta =\frac{1}{2}+\frac{3}{4p}$)
\begin{align*}
T_{-,1}\lesssim &\ep^{\frac{3}{2}p-\frac{3}{2}}\lan\al\ep\ran^{p}\left(\int_{-\infty}^\infty\frac{1}{\lan c\ran^{\frac{2(1-\theta) p}{p-2}}}dc_r\right)^{\frac{p-2}{2}}\int_{-\infty}^1\int_{\frac{-1-c_r}{\ep}}^{\frac{1-c_r}{\ep}}\mathbf{1}_{Y_r\geq 0}\frac{|H_-(Y)|^p}{\ep^{p-1}\lan c\ran^{\theta p}\big|A_0(\frac{-1-c_r}{\ep})\big|^p}dY_r dc_r.
\end{align*}
In the $c_r\leq y$ case, $|A_0(\ep^{-1}(-1-c_r+\delta\ep i))|^{-1}$ becomes super-exponentially large and $|H_-(Y)|$ becomes super-exponentially small, hence we use a different asymptotic expansion here. Combining the definition of the weight $\eta_{\ep,p}$ \eqref{def:eta1}, the $H_\pm$ estimates \eqref{estimate_of_H_pm}, the $A_0$ estimate \eqref{A_0_lower_bound} yields
\begin{align*}
T_{-,1}\lesssim & \ep^{\frac{3}{2}p-\frac{3}{2}}\int_{-1}^1\frac{1}{\ep^{p-1}\lan c\ran^{\theta p}}\int_{\max\{\frac{-1-c_r}{\ep},0\}}^{\frac{1-c_r}{\ep}} \frac{1}{\lan Y\ran^{p/4}}e^{-\frac{2p}{3}Re(Y^{3/2}e^{i\pi/4})}dY_r dc_r\nonumber\\
&+\ep^{\frac{3}{2}p-\frac{3}{2}}\int_{-\infty}^{-1}\frac{1}{\ep^{p-1}\lan c\ran^{\theta p}}\int_{\max\{\frac{-1-c_r}{\ep},0\}}^{\frac{1-c_r}{\ep}} \frac{\lan \frac{-1-c_r}{\ep}+\delta i\ran^{3p/4}}{\lan Y\ran^{p/4}}e^{\frac{2p}{3}Re((\frac{-1-c_r}{\ep}+\delta i)^{3/2}e^{i\pi/4}-Y^{3/2}e^{i\pi/4})}dY_r dc_r.
\end{align*}
The first term is treated by direct integration. 
For the second term, we gain a fraction power of $ \left\lan \frac{-1-c_r}{\ep}+\delta  i\right\ran$ by integrating $Y_r$ first.
Indeed, Lemma \ref{Gain_in_poly_power} implies
\begin{align}
T_{-,1}\lesssim&\ep^{\frac{1}{2}p-\frac{1}{2}} +\ep^{\frac{3}{2}p-\frac{3}{2}}\int_{-1}^\infty\frac{1}{\ep^{p-1}\lan c\ran^{\theta p}} \left\lan \frac{-1-c_r}{\ep}+\delta i\right\ran^{p/2-1/2} dc_r \nonumber\\
\lesssim & \ep^{\frac{1}{2}p-\frac{1}{2}} +  \ep^{\frac{1}{2}p-\frac{1}{2}}\int_{-1}^\infty\frac{1}{\lan c\ran^{(\theta -1/2)p+1/2}}dc_r\lesssim 1.\label{T-1}
\end{align}
This completes the case $p \in [2,\infty)$.
Combining the treatments of the $T_{+,1}$ term for $p\in[1,2]$ and  $T_{-,1}$ for  $p\in[2,\infty)$ gives the desired estimates on $T_{-,1}$ for $p\in[1,2]$.
This completes the estimation of the $T_{-,1}$ term in \eqref{T+123T-123}. 

For the remaining $T_{-,2}$ and $T_{-,3}$ terms in \eqref{T+123T-123}, we separate into two cases, $c_r\in[-1,1]$ and $c_r\in(-\infty,-1)$, based on the following consideration. For the $T_{-,\cdot}$ terms, the domain of $c_r$ is restricted such that the argument of $H_-$ has positive real parts, i.e. $Y_r=\frac{y-c_r}{\ep}\geq 0$, so the factor $|H_-(Y)|$ is exponentially decreasing. As a result, as long as the growth of the $|A_0|^{-1}$ is properly controlled, the estimates follow. In the case $c_r\in[-1,1]$, since $-1-c_r$ is less than $0$ and $\delta$ is chosen small enough, there holds $\left|A_0\left(\frac{-1-c_r}{\ep}+\delta i\right)\right|\gtrsim 1$. For the case $c_r\in(-\infty,-1]$, we use the lower bound on $\left|A_0\left(\frac{-1-c_r}{\ep}+\delta i\right)\right|$ in \eqref{A_0_lower_bound} and a similar argument to the treatment of $T_{+,2}$ and $T_{+,3}$ in \eqref{T+123T-123}.  
Consider first $T_{-,2}$ in \eqref{T+123T-123}, we decompose the integral into two parts:
\begin{align}
T_{-,2}=
&\int_{-1+\ep^{3/2}}^{-1+\ep}\eta_{\ep,p}^p\left(\int_{-1}^1\mathbf{1}_{y\geq c_r}\frac{1}{\lan c\ran^2\ep^2}\frac{|H_-(Y)|^2}{\big|A_0(\frac{-1-c_r}{\ep}+\delta i)\big|^2}dc_r
\right)^{p/2}dy\nonumber\\
&+\int_{-1+\ep^{3/2}}^{-1+\ep}\eta_{\ep,p}^p\left(\int_{-\infty}^{-1}\mathbf{1}_{y\geq c_r} \frac{1}{\lan c\ran ^{2}\ep^2}\frac{|H_-(Y)|^2}{\big|A_0(\frac{-1-c_r}{\ep}+\delta i)\big|^2}dc_r\right)^{p/2}dy\nonumber\\
=:&T_{-,21}+T_{-,22}.\label{T-2-12}
\end{align}
To estimate the $T_{-,21}$ term in \eqref{T-2-12}, we combine the definition of the weight $\eta_{\ep,p}$ \eqref{def:eta1}, the $H_\pm$ estimates \eqref{estimate_of_H_pm} and the estimate $A_0 \gtrsim 1$ for $c_r\geq -1$ \eqref{A_0_lower_bound} to get 
\begin{align*}
T_{-,21}\lesssim&\int_{-1+\ep^{3/2}}^{-1+\ep}\eta_{\ep,p}^p\left(\int_{-1}^1\mathbf{1}_{Y_r\geq 0}\frac{1}{\lan c\ran ^2\ep^2}e^{-\frac{4}{3}Re(Y^{3/2}e^{i\pi/4})}dc_r\right)^{p/2}dy=\int_{-1+\ep^{3/2}}^{-1+\ep}\eta_{\ep,p}^p \ep^{-p/2}dy\lesssim 1.
\end{align*} 
To estimate the $T_{-,22}$ term in \eqref{T-2-12}, we combine the definition of the weight $\eta_{\ep,p}$ \eqref{def:eta1}, the $H_\pm$ estimates \eqref{estimate_of_H_pm} and the $A_0$ estimate \eqref{A_0_lower_bound} and use the same argument as in the estimate of the $T_{+,2}$ term \eqref{T2+est} to get
\begin{align*}
T_{-,22}\lesssim &\int_{-1+\ep^{3/2}}^{-1+\ep}\eta_{\ep,p}^p\left(\int_{-\infty}^{-1}\frac{1}{\ep^{2}\lan c\ran^2}\left\lan\frac{-1-c_r}{\ep}+\delta i\right\ran e^{\frac{4}{3}Re[(\frac{-c_r-1}{\ep}+\delta i)^{3/2}-(\frac{y-c_r}{\ep}+\delta i)^{3/2})e^{i\pi/4}]} d{ c_r}\right)^{p/2}dy\nonumber\\
\lesssim&\int_{-1+\ep^{3/2}}^{-1+\ep} \eta_{\ep,p}^p\left(\int_{-\infty}^{-1} \frac{1}{\ep^{2}\lan c\ran^2}\left\lan\frac{-1-c_r}{\ep}+\delta i\right\ran e^{-\frac{4}{3}\frac{y+1}{\ep}Re[(\frac{-1-c_r}{\ep}+\delta i)^{1/2}e^{i\pi/4}]} d{c_r}\right)^{p/2}dy\\
\lesssim & 1.
\end{align*}
For the last term $T_{-,3}$ in \eqref{T+123T-123}, we apply the definition \eqref{def:eta1} and decompose the integral into four parts:
\begin{align}
T_{-,3}\lesssim &\int_{-1+\ep}^1e^{pc_\star\frac{(y+1)^{3/2}}{\ep^{3/2}}}\ep^{\frac{p}{2}-1}\left(\int_{-\infty}^{-1}\mathbf{1}_{y\geq c_r}\mathbf{1}_{-1-c_r\leq \frac{y+1}{4}} \frac{1}{\lan c\ran ^{2}\ep^2}\frac{|H_-(Y)|^2}{\big|A_0(\frac{-1-c_r}{\ep}+\delta i)\big|^2}dc_r\right)^{p/2}dy\nonumber\\&
+\int_{-1+\ep}^1e^{pc_\star\frac{(y+1)^{3/2}}{\ep^{3/2}}}\ep^{\frac{p}{2}-1}\left(\int_{-\infty}^{-1}\mathbf{1}_{y\geq c_r}\mathbf{1}_{-1-c_r>\frac{y+1}{4}} \frac{1}{\lan c\ran ^{2}\ep^2}\frac{|H_-(Y)|^2}{\big|A_0(\frac{-1-c_r}{\ep}+\delta i)\big|^2}dc_r\right)^{p/2}dy\nonumber\\
&+\int_{-1+\ep}^1e^{pc_\star\frac{(y+1)^{3/2}}{\ep^{3/2}}}\ep^{\frac{p}{2}-1}\left(\int_{-1}^1\mathbf{1}_{y\geq c_r}\mathbf{1}_{|y-c_r|\leq \frac{y+1}{2}} \frac{1}{\lan c\ran^2\ep^2}\frac{|H_-(Y)|^2}{\big|A_0(\frac{-1-c_r}{\ep}+\delta i)\big|^2}dc_r
\right)^{p/2}dy\nonumber\\
&+\int_{-1+\ep}^1e^{pc_\star\frac{(y+1)^{3/2}}{\ep^{3/2}}}\ep^{\frac{p}{2}-1}\left(\int_{-1}^1\mathbf{1}_{y\geq c_r}\mathbf{1}_{|y-c_r|>\frac{y+1}{2}} \frac{1}{\lan c\ran^2\ep^2}\frac{|H_-(Y)|^2}{\big|A_0(\frac{-1-c_r}{\ep}+\delta i)\big|^2}dc_r
\right)^{p/2}dy\nonumber\\
=:&T_{-,31}+T_{-,32}+T_{-,33}+T_{-,34}.\label{T-3-1234}
\end{align}
To estimate the $T_{-,31}$ term in \eqref{T-3-1234}, the definition of the weight $\eta_{\ep,p}$ \eqref{def:eta1}, the $H_\pm$ estimates \eqref{estimate_of_H_pm}and the $A_0$ estimate \eqref{A_0_lower_bound} to get
\begin{align*}
T_{-,31}\lesssim &\ep^{\frac{p}{2}-1}\int_{-1+\ep}^1 e^{pc_\star \frac{(y+1)^{3/2}}{\ep^{3/2}}}\\
&\quad \times\left(\int_{-\infty}^{-1}\ep^{-2}\mathbf{1}_{Y_r\geq 0} \mathbf{1}_{-1-c_r\leq \frac{y+1}{4}}\left\lan\frac{-1-c_r}{\ep}+\delta i\right\ran^{\frac{3}{2}} e^{\frac{4}{3}Re[(\frac{-c_r-1}{\ep}+\delta i)^{3/2}-Y^{3/2})e^{i\pi/4}]} d{ c_r}\right)^{p/2}dy.
\end{align*}
We apply the constraints $-c_r-1\leq \frac{y+1}{4}$, $c_r\leq- 1$ and $y+1\geq \ep$ to control the exponential factor in the integral and obtain that, for a universal constant $C$, the following estimate holds 
\begin{align}
T_{-,31} \lesssim&\ep^{\frac{p}{2}-1}\int_{-1+\ep}^1 e^{pc_\star \frac{(y+1)^{3/2}}{\ep^{3/2}}}\bigg(\int_{-\infty}^{-1}\ep^{-2} \left\lan\frac{-1-c_r}{\ep}+\delta i\right\ran^{\frac{3}{2}}e^{-C^{-1}\frac{(y+1)^{3/2}}{\ep^{3/2}}-\frac{1}{16}Re[(1+\frac{-1-c_r}{\ep}+\delta i)^{3/2}e^{i\pi/4}]} d{c_r}\bigg)^{p/2}dy \nonumber \\
\lesssim &\ep^{-1}\int_{-1+\ep}^1e^{-p(C^{-1}-c_\star)(\frac{1+y}{\ep})^{\frac{3}{2}}} dy\lesssim 1,\label{T-31}
\end{align}
where the last line followed by choosing the constant $c_\star$ in the weight $\eta_{\ep,p}$ small enough. 
To estimate the $T_{-,32}$ term in \eqref{T-3-1234}, we combine the $H_\pm$ estimates \eqref{estimate_of_H_pm} and the $A_0$ estimate \eqref{A_0_lower_bound} and use the same argument as in the estimate of the $T_{+,2}$ term \eqref{T2+est} to get
\begin{align*}
T_{-,32}
\lesssim&\ep^{\frac{p}{2}-1}\int_{-1+\ep}^1 e^{pc_\star \frac{(y+1)^{3/2}}{\ep^{3/2}}}\left(\int_{-\infty}^{-1} \ep^{-2}\mathbf{1}_{-1-c_r\geq \frac{y+1}{4}}\left\lan\frac{-1-c_r}{\ep}+\delta i\right\ran^{\frac{3}{2}}e^{-\frac{4}{3}Re[(\frac{y+1}{\ep})(\frac{-1-c_r}{\ep}+\delta i)^{1/2}e^{i\pi/4}]} d{c_r}\right)^{p/2}dy.
\end{align*}
The constraint that $\frac{-1-c_r}{\ep}\geq \frac{y+1}{4\ep}$ and $y+1\geq \ep$ yields that there exists a universal constant $C$ such that the following estimate holds
\begin{align}
T_{-,32}\lesssim&\ep^{\frac{p}{2}-1}\int_{-1+\ep}^1 e^{pc_\star \frac{(y+1)^{3/2}}{\ep^{3/2}}}\left(\int_{-\infty}^{-1} \ep^{-2}\left\lan\frac{-1-c_r}{\ep}+\delta i\right\ran^{\frac{3}{2}}e^{-C^{-1}\frac{(y+1)^{3/2}}{\ep^{3/2}}-\frac{1}{16}Re[(\frac{-1-c_r}{\ep}+\delta i )^{1/2}e^{i\pi/4}]} d{c_r}\right)^{p/2}dy \nonumber \\
\lesssim &\ep^{-1}\int_{-1+\ep}^1e^{-p(C^{-1}-c_\star)(\frac{1+y}{\ep})^{\frac{3}{2}}} dy\lesssim 1, \label{T-32}
\end{align}
where the last estimate followed by choosing the $c_\star$ small compared to $C^{-1}$. 
Next, we estimate the $T_{-,33}$ term in \eqref{T-3-1234}, combining $y-c_r\leq \frac{y+1}{4}$ together with $-1\leq c_r\leq y$ implies $|1+c_r|\geq\frac{y+1}{2}$, the $H_\pm$ estimates \eqref{estimate_of_H_pm}  and the $A_0$ estimate \eqref{A_0_lower_bound} to get that there exists a universal constant $C$ such that 
\begin{align}
T_{-,33}\lesssim &\ep^{\frac{p}{2}-1}\int_{-1+\ep}^1 e^{pc_\star \frac{(y+1)^{3/2}}{\ep^{3/2}}}\left(\int_{-1}^{y}\ep^{-2}\mathbf{1}_{|1+c_r|\geq \frac{y+1}{2}}\left\lan\frac{-1-c_r}{\ep}+\delta i\right\ran^{\frac{3}{2}} e^{-\frac{4}{3}Re[(\frac{-c_r-1}{\ep}+\delta i)^{3/2}e^{i\pi/4}]} d{ c_r}\right)^{p/2}dy \nonumber \\
\lesssim &\ep^{-1}\int_{-1+\ep}^1e^{-p(C^{-1}-c_\star)(\frac{1+y}{\ep})^{\frac{3}{2}}} dy\lesssim 1, \label{T-33}
\end{align}
where the last line followed by choosing the constant $c_\star$ sufficiently small. 
Finally, we estimate the $T_{-,34}$ term in \eqref{T-3-1234}, we combine the fact that $Y_r=\frac{y-c_r}{\ep}\geq\frac{y+1}{2\ep}$ in this region, the $H_\pm$ estimates \eqref{estimate_of_H_pm} and the $A_0$ estimate $|A_0|\gtrsim 1$ \eqref{A_0_lower_bound} to get
\begin{align}
T_{-,34}
\lesssim &\ep^{-1+\frac{p}{2}}\int_{-1+\ep}^1 e^{pc_\star \frac{(y+1)^{3/2}}{\ep^{3/2}}}\left(\int_{-1}^{y}\ep^{-2} e^{-C^{-1}(\frac{y+1}{\ep})^{3/2}-\frac{1}{3}Re[Y^{3/2}e^{i\pi/4}]} d{ c_r}\right)^{p/2}dy\nonumber\\
\lesssim &\ep^{-1}\int_{-1+\ep}^1e^{-p(C^{-1}-c_\star)(\frac{1+y}{\ep})^{\frac{3}{2}}} dy\lesssim 1.\label{T-34}
\end{align}
Combining all of \eqref{T-3-1234}, \eqref{T-31}, \eqref{T-32}, \eqref{T-33} and \eqref{T-34} completes the estimate of $T_{-,3}$. 
This completes the proof of \eqref{ineq:Kpm1} for $K_{+,1}$. 

We next give a brief sketch of how to treat the remaining terms in \eqref{K_+} and \eqref{K_-}.
First, the treatment of $K_{+,1}$ extends analogously to cover $K_{-,1}$.
The treatment of the $K_{\ast,2}$ terms are analogous with some minor changes.
First, the decomposition of the $y$ integrals will be into A) $y\in[1-\ep^{3/2}, 1]$; B) $y\in[1-\ep,1-\ep^{3/2}]$; C) $y\in[-1,1-\ep]$. 
Seceond, we apply the fact that $Ai(y)=\overline{Ai(\overline{y})}$ to rewrite $H_+(Y)=Ai(e^{5\pi i/6}Y)$ as $\overline{Ai(e^{- 5\pi i/6}\overline{Y})}=\overline{Ai\left(e^{\pi i/6}\left(\frac{c_r-y}{\ep}+\delta i\right)\right)}$. Finally we observe the following monotonicity :
\begin{align}
e^{\frac{2}{3}Re(e^{i\pi/4}(\frac{c_r-y}{\ep}+\delta i)
^{3/2})}\lesssim e^{\frac{2}{3}Re(e^{i\pi/4}(\frac{c_r-1}{\ep}+\delta i)^{1/2}(\frac{c_r-y}{\ep}+\delta i))},\quad c_r-1\leq c_r-y\leq 0 .
\end{align}
After applying these observations, the treatment of the $K_{\ast,2}$ terms follow analogously to the $K_{\ast,1}$ terms with only minor changes.
The details are omitted for the sake of brevity.

\textbf{The high frequency case $\alpha^2 \nu > C_0$:} \\
First, note that $C_0$ is chosen such that $(1-\frac{\kappa}{2})\alpha^2 \nu \geq (1-\kappa) \alpha^2 \nu + \delta \alpha^{2/3}\nu^{1/3}$.
Therefore, by proving the slightly stronger estimate in this region using $\lambda = (1-\frac{\kappa}{2})\alpha^2 \nu$ we obtain the desired result.
Hence, it suffices to use this slightly stronger $\lambda$ here.
After one makes this shift, the difference between low and high frequency cases is mainly a matter of different complex phase book-keeping. After this shift, the range of interest for $c$ is $c\in\{c=c_r+c_i|c_r\in\rr,\, c_{i}=-(1-\frac{\kappa}{2})\al\nu\}$.
Since $\al$ is large, the imaginary part of the Langer variables \eqref{Langer_variable} becomes $Y_i=-\frac{1}{2}\kappa\al^{4/3}\nu^{2/3}<0$, jumping the phase from $(0, \pi)$ to $(-\pi,0)$ and creating a minor inconvenience in using the asymptotic expansions of the Airy functions.
To this end, we replace inequalities \eqref{estimate_of_H_pm_high_mode} and \eqref{ineq:powergain_2} with \eqref{estimate_of_H_pm} and \eqref{ineq:powergain} when estimating the homogeneous solutions $H_\pm$ and the Evans function $D(\al, c)$. Furthermore, we use the inequalities \eqref{Evans_function_lower_bound_2}, \eqref{H_+_int_high_mode}, and \eqref{H_-_int_high_mode} instead of \eqref{Evans_function_lower_bound_1}, \eqref{H_+_int}, and \eqref{H_-_int} when we carry out the estimates. From there, a straightforward adaptation of the arguments above to estimate $K_{\pm}$  yields the desired estimates.

\subsection{Inviscid damping estimate \eqref{ineq:IDu}, \eqref{ineq:LinfL1K}, and \eqref{ineq:LpLinfK}} \label{sec:IDtrick}
Next, we prove the $L_t^2L_y^2$ inviscid damping estimate of the $u_\al$ in \eqref{ineq:IDu}, \eqref{ineq:LinfL1K}, and \eqref{ineq:LpLinfK}. 

First, we justify \eqref{ineq:evdt5}. Denote $Green_{\alpha, \de_D} = \begin{pmatrix}-\partial_y \\ -i\alpha \end{pmatrix} (-\Delta_{\alpha,D})^{-1}$, where $(-\Delta_{\alpha,D})^{-1}$ denotes the operator $(\alpha^2 - \partial_{yy})^{-1}$ when equipped with Dirichlet boundary conditions.
As in the previous section, we separate between low and high frequency cases.

\textbf{The low frequency case $\alpha^2 \nu \leq C_0$:}  \\ 
Standard elliptic estimates yield (defining $c= c_r - i \lambda/\alpha$ as above), 
\begin{align*}
||e^{\lambda t} \widehat{u_b}(t,\alpha,\cdot)||_{L_t^2L_y^2} ^2 \lesssim&\int_\rr\int_{-1}^1 \left(\int Green_{\de_D} \mathcal{R}(y, c)dy\right)^2dydc_r\\
\lesssim&\int_\rr||\mathcal{R}(y,c)||_{L_y^1}^2dc_r\lesssim\int_\rr \left(\int \frac{1}{\lan c\ran}|K_{\pm}| dy\right)^2(\lan c\ran \mathcal{R}_f[c])^2 dc_r\\
\lesssim& \left(\sup_{c\in\Gamma}\int_{-1}^1 \frac{1}{\lan c\ran}|K_{\pm}(y,c)|dy\right) ^2 \norm{\widehat{\omega}_{in}(\alpha,\cdot)}_{L^2_y}
\end{align*}where the supremum is taken over $\Gamma:=\{c=c_r-(\delta\ep+ (1-\kappa)\al \nu) i,\ c_r\in(-\infty,\infty)\}$ in the last line.
We further decompose the $K_\pm$ as $K_{+,1}$, $K_{+,2}$, $K_{-,1}$ and $K_{-,2}$ as in \eqref{K_+}, \eqref{K_-}.
The estimates are all analogous, so we simply consider the case $K_{+,1}$. 

Analogous to \eqref{T+123T-123}, we apply Lemma \ref{lem:Evans_function_est} and decompose the integral
\begin{align}
&\sup_{c\in \Gamma}\int_{-1}^1\frac{|K_{+,1}(Y,c)|}{\lan c\ran }dy\nonumber\\
 &\lesssim \left(\sup_{c_r\in[-1,\infty)}\int_{-1}^{-1+\ep^{3/2}}+\sup_{c_r\in[-1,\infty)}\int_{-1+\ep^{3/2}}^{-1+\ep}+\sup_{c_r\in[-1,\infty)}\int_{-1+\ep}^1\right) \mathbf{1}_{y\leq c_r}\frac{1}{\lan c\ran\ep e^{\al}} \frac{|H_-(Y)|}{\big|A_0(\frac{-1-c_r}{\ep}+\delta i)\big|} dy\nonumber\\
&\ +\left(\sup_{c_r\in(-\infty,1]}\int_{-1}^{-1+\ep^{3/2}}+\sup_{c_r\in(-\infty,1]}\int_{-1+\ep^{3/2}}^{-1+\ep}+\sup_{c_r\in(-\infty,1]}\int_{-1+\ep}^1\right)\mathbf{1}_{y\geq c_r}\frac{1}{\lan c\ran\ep e^{\al}} \frac{|H_-(Y)|}{\big|A_0(\frac{-1-c_r}{\ep}+\delta i)\big|} dy\nonumber\\
&=:e^{-\al} (T_{+,1}^1+T_{+,2}^1+T_{+,3}^1+T_{-,1}^1+T_{-,2}^1+T_{-,3}^1).\label{T+123T-123_1}
\end{align} 
We estimate each term of \eqref{T+123T-123_1} and show that they are of order $1$.

First we estimate the $T_{\pm,1}^1$ terms in \eqref{T+123T-123_1}. Combining the $H_\pm$ estimates \eqref{estimate_of_H_pm} and the $A_0$ estimate \eqref{A_0_lower_bound}, we apply a similar argument to the estimate of $T_{+,1}$ term in \eqref{T+123T-123} to obtain that
\begin{align*}
T_{+,1}^1\lesssim \sup_{c_r\in[-1,\infty)}\int_{-1}^{-1+\ep^{3/2}}\frac{\lan \frac{-1-c_r}{\ep}+\delta i \ran^{1/2} }{\ep\lan c\ran }dy\lesssim \int_{-1}^{-1+\ep^{3/2}}\frac{1}{\ep^{3/2}}dy=1.
\end{align*}
The estimate of the $T_{-,1}^1$ is similar to that of $T_{-,1}$ in \eqref{T+123T-123}. We omit the details for the sake of brevity.

Next, we estimate the $T_{\cdot, 2}^1$ terms in \eqref{T+123T-123_1}. 
Combining the $H_\pm$ estimates \eqref{estimate_of_H_pm}, the $A_0$ estimate \eqref{A_0_lower_bound} and Lemma \ref{Gain_in_poly_power}, we have that there exists a universal constant $C$ such that
\begin{align*}
T_{+,2}^1\lesssim&\sup_{c_r\in[-1,\infty)}\frac{1}{\ep}\int_{-1+ \ep^{3/2}}^{-1+\ep}\frac{1}{\lan -1-c_r\ran}\mathbf{1}_{|1+c_r|\geq \ep}e^{-C^{-1}|\frac{y+1}{\ep}||\frac{-c_r-1}{\ep}|^{1/2}}\left| \frac{-1-c_r}{\ep}\right| ^{1/2}dy\\
&+\sup_{c_r\in[-1,\infty)}\frac{1}{\ep}\int_{-1+\ep^{3/2}}^{-1+\ep} \mathbf{1}_{|1+c_r|\leq \ep}d{y}\lesssim 1. 
\end{align*}
The estimate of $T_{-,2}^1$, is similar to that off $T_{-,2}$ in \eqref{T+123T-123}. We omit the details for the sake of brevity. 
 
Finally, we estimate the $T_{\pm,3}^1$ terms in \eqref{T+123T-123_1}. The estimate of $T_{+,3}^1$ is similar to that of $T_{+,3}$ in \eqref{T+123T-123}.
Combining the $H_\pm$ estimates \eqref{estimate_of_H_pm}, the $A_0$ estimate \eqref{A_0_lower_bound}, and the fact that $1+c_r\geq\ep$ on the support of the integrand,  there exists a universal constant $C$ such that
\begin{align*}
T_{+,3}^1\lesssim &\sup_{c_r\in[-1,\infty)}\frac{1}{\ep} \int_{-1+\ep}^1\frac{\mathbf{1}_{y\leq c_r}}{\lan -1-c_r\ran}\left| \frac{-1-c_r}{\ep}\right| ^{1/2}\left|\frac{y+1}{\ep}\right|^{1/4}e^{-\frac{(y+1)^{3/2}}{4\ep^{3/2}}}e^{-C^{-1}|\frac{y+1}{4\ep}||\frac{-c_r-1}{\ep}|^{1/2}} dy\lesssim 1.
\end{align*}
The estimate of the $T_{-,3}^1$ is similar to the estimate of $T_{-,3}$ in \eqref{T+123T-123}. One can decompose the supremum in $c_r$ into the $c_r\leq -1$ case and the $c_r\in[-1,1]$ cases and apply different  estimates of $H_\pm$ and $A_0$ to get the desired estimates. We omit the details for the sake of brevity. 

\textbf{The high frequency case $\alpha^2 \nu > C_0$:} \\
As in Section \ref{sec:kpm1pf} we use the slightly stronger $\lambda = (1-\frac{\kappa}{2})\alpha^2 \nu$.
The $c$ of interest has  $c_i=-(1-\frac{1}{2}\kappa)\al\nu$ and the Langer variables have imaginary part $-\frac{1}{2}\kappa \al^{4/3}\nu^{2/3}$.
By replacing inequalities \eqref{estimate_of_H_pm}, \eqref{ineq:powergain},  \eqref{Evans_function_lower_bound_1}, \eqref{H_+_int}, and \eqref{H_-_int} with inequalities \eqref{estimate_of_H_pm_high_mode},  \eqref{ineq:powergain_2},  \eqref{Evans_function_lower_bound_2}, \eqref{H_+_int_high_mode}, and \eqref{H_-_int_high_mode},  one proves the estimates with similar methods that that used in the low frequency case. The details are omitted for the sake of brevity.

This concludes the proof of \eqref{ineq:LinfL1K}.
The proof of \eqref{ineq:LpLinfK} follows from a combination of the arguments used to prove \eqref{ineq:LinfL1K} and those used to prove \eqref{ineq:Kpm1}.
We omit the repetitive details for the sake of brevity.
This concludes the proof of Theorem \ref{thm:Sepbd} and Theorem \ref{thm:RegData}.

\section{$L^2$ case: Resolvent estimates near the continuous spectrum} \label{sec:Gen1}
\label{Sec:NearSpec_1}
Here and in Section \ref{sec:Gen2} we prove estimates \eqref{ineq:GenC} for the kernels defining the boundary vorticity in \eqref{om_b}; see Appendix \ref{sec:Green}.
 
We treat low and high frequencies the same for $t \geq 1$ by using the stronger $\lambda = \alpha \nu + \delta \alpha^{2/3}\nu^{1/3}$. 
As the treatments of all the kernels are essentially the same, we focus only on $K_1$ and briefly explain the others.
We divide the desired estimates into two contributions based on $\Gamma_\alpha$ (see Section \ref{sec:L2OrrSom}): 
\begin{align}
K_1(y,z) &= \left(\int_{\Gamma_E} + \int_{\Gamma_\alpha \setminus \Gamma_E}\right) \frac{2 e^{-i \al c t}}{\pi D(\al,c)}\bigg(\aaa\bigg)\int_{z}^1 \epsilon^{-2} e^{\al r}H_-(R)drH_+(Z)H_-(Y)d c dr \\
& = K_{1;nc} + K_{1;fc},
\end{align}
where $\Gamma_E = \{c|c_r\in[-1,1], c_i=-\al\nu-\delta\al^{-1/3}\nu^{1/3}\}$ (``nc'' means ``near continous spectrum'' and ``fc'' means ``far from continuous spectrum'').
In this section we estimate $K_{\ast;nc}$, in Section \ref{sec:Gen2} we estimate $K_{\ast;fc}$. 

\noindent
\textbf{Step 1: Weighted estimate $\sup_y||\widetilde{\eta}_\ep K_{1;nc}(y,\cdot)||_{L^1_z}$} \\  
The first step we estimate the kernel corresponding to $\mathcal{R}_f$, dividing into three contributions based on $c_r$
\begin{align}
\ep^{-2}&\int_{-1}^1 \int_{z}^{1} e^{\al r}|H_-(R)H_+(Z)|drdz=\ep^{-2}\int_{-1}^1 \int_{z}^{1} (\mathbf{1}_{z\leq  c_r\leq  r}+\mathbf{1}_{ z\leq  r\leq  c_r}+\mathbf{1}_{c_r\leq  z\leq r})e^{\al r}|H_-(R)H_+(Z)|drdz\nonumber\\
=:&I_1+II_1+III_1.\label{RZ_factor_2} 
\end{align}
We estimate the first term $I_1$ in \eqref{RZ_factor_2}. Applying the estimates on the $H_\pm$ functions \eqref{estimate_of_H_pm} yields
\begin{align*}
I_1 \lesssim& e^\al\int_{\frac{-1- c_r}{\ep}}^{\frac{1- c_r}{\ep}}\int_{\frac{-1- c_r}{\ep}}^{\frac{1- c_r}{\ep}} \frac{1}{\lan Z\ran^{1/4}\lan R\ran^{1/4}}  e^{-\frac{2}{3}Re(Z^{3/2}{e^{i\pi/4}})}e^{-\frac{2}{3}Re(R^{3/2}e^{i\pi/4})}dR_r dZ_r
\lesssim e^\al.  
\end{align*}
Next we estimate $II_1$. Applying Lemma \ref{Gain_in_poly_power} gives 
\begin{align*}
II_1 \lesssim&\int_{\frac{-1- c_r}{\ep}}^{0} \int_{Z_r}^{0} e^{\al}\frac{1}{\lan R\ran^{1/4}\lan Z\ran^{1/4}}e^{\frac{{2}}{3}Re(R^{3/2}e^{i\pi/4})}e^{-\frac{{2}}{3}Re(Z^{3/2}e^{i\pi/4})}dR_rdZ_r\\
\lesssim&\int_{\frac{-1-   c_r}{\ep}}^{0}\frac{e^\al}{\lan Z\ran}dZ\lesssim e^\al \ln \left\lan \frac{1+ |c_r|}{\ep}\right\ran.
\end{align*}
{
Finally, we estimate the $III_1$ term in \eqref{RZ_factor_2} using Lemma \ref{Gain_in_poly_power} and estimate \eqref{estimate_of_H_pm}, 
\begin{align*}
III_1 \lesssim&\int_{\frac{-1-c_r}{\ep}}^{\frac{1-c_r}{\ep}}\frac{e^\al}{\lan Z\ran^{1/4}}\int_{Z_r}^{\frac{1-c_r}{\ep}} \mathbf{1}_{0\leq Z_r\leq R_r} \frac{1}{\lan R\ran^{1/4}}e^{-\frac{{2}}{3}Re(e^{i\pi/4}R^{3/2})}e^{\frac{{2}}{3}Re(e^{i\pi/4}Z^{3/2})}dR_r dZ_r \\
\lesssim& \int_{0}^{\frac{1-c_r}{\ep}}\frac{e^\al}{\lan Z\ran} e^{-\frac{{2}}{3}Re(e^{i\pi/4}Z^{3/2})} e^{\frac{{2}}{3}Re(e^{i\pi/4}Z^{3/2})}dZ_r\lesssim e^{\al}\ln\left\lan\frac{1+|c_r|}{\ep}\right\ran.
\end{align*}}
Combining the estimates of $I_1,II_1,III_1$ and \eqref{RZ_factor_2}, 
with the Evans function estimate \eqref{Evans_function_lower_bound_1}, Airy function estimate \eqref{H_+_int}, and $\ln \ep^{-1}(1+|c|_r)\lesssim \ln \ep^{-1}$ for $c_r\in [-1,1]$, we obtain
\begin{align*}
\sup_{y\in[-1,1]}&||\widetilde{\eta}_\ep(y)K_{1;nc}(y,\cdot)||_{L_z^1}\lesssim {\ep}^{-1} \ln {\ep}^{-1} \sup_{y\in[- 1,1]}\int_{-{1}}^{{1}}\frac{\widetilde{\eta}_\ep(y) e^{-i \al c t}}{ |A_0(\frac{-1-c_r }{\ep}+\delta i)|}|H_-(Y)| d{ c_r}.
\end{align*}
We further decompose the integral above into three regimes
\begin{align}
\sup_{y\in[-1,1]}||\widetilde{\eta}_\ep(y)K_{1;nc}(y,z)||_{L_z^1}\lesssim & {\ep}^{-1} \ln{\ep}^{-1} \sup_{d(y,- 1)\leq \ep^{3/2}}\int_{-{1}}^{{1}}\frac{\widetilde{\eta}_\ep(y) e^{\al c_i t}}{ |A_0(\frac{-1-c_r }{\ep}+\delta i)|}|H_-(Y)| d{ c_r}\nonumber\\
&+\ep^{-1} \ln{\ep}^{-1} \sup_{d(y,-1)\in[\ep^{3/2},\ep]}\int_{-{1}}^{{1}}\frac{\widetilde{\eta}_\ep(y) e^{\al c_i t}}{ |A_0(\frac{-1-c_r }{\ep}+\delta i)|}|H_-(Y)| d c_r\nonumber\\
&+{\ep}^{-1} \ln{\ep}^{-1} \sup_{d(y,- 1)\geq \ep}\int_{-{1}}^{{1}}\frac{\widetilde{\eta}_\ep(y) e^{\al c_i t}}{ |A_0(\frac{-1-c_r }{\ep}+\delta i)|}|H_-(Y)| d{ c_r}\nonumber\\
=:&T_1+ T_2+ T_3.\label{T1T2T3}
\end{align}
The treatment of $T_1$ in \eqref{T1T2T3} is similar to the estimate of $T_{+,1}$ in \eqref{T+1}, \eqref{T+1leq2} in Section \ref{sec:Ksep}.
Recalling the definition of $\wt\eta_\ep$ in \eqref{eta_ep}, the estimates of $A_0$ \eqref{A_0_lower_bound}, and $H_\pm$ \eqref{estimate_of_H_pm}, we estimate $T_1$ as follows 
\begin{align}
T_1\lesssim &\sup_{d(y,-1)\leq \ep^{3/2}}{\ln{\ep}^{-1} \ep^{1/2}}{}\int_{{y}}^{{1}}\frac{e^{\al c_i t}\lan \frac{-c_r-1}{\ep}+\delta i\ran ^{3/4} }{\lan Y\ran ^{1/4}}e^{-\frac{{2}}{3}Re((\frac{- c_r-1}{\ep}+\delta i)^{3/2}e^{i\pi/4})}e^{\frac{{2}}{3}Re(Y^{3/2}e^{i\pi/4})} d{ c_r}\nonumber\\
&+\sup_{d(y,-1)\leq \ep^{3/2}}{\ep^{1/2} \ln{\ep}^{-1}}\int_{-{1}}^{{y}}e^{\al c_i t}e^{-\frac{{2}}{3}Re((\frac{y- c_r}{\ep}+\delta i)^{3/2}e^{i\pi/4})} d{ c_r}\nonumber\\
\lesssim&\ln \ep^{-1}\ep^{1/2}\int_{{y}}^{{1}}e^{\al c_i t}\left\lan \frac{-c_r-1}{\ep}+\delta i\right\ran ^{1/2} d{c_r}+ e^{-\lambda t}\nonumber\\
\lesssim&\ln \ep^{-1} e^{-\lambda t}.\label{T1}
\end{align}
Next we estimate $T_2$ in \eqref{T1T2T3}. The argument is similar to that used to estimate $T_{+,2}$, i.e., \eqref{T2+est}, in Section \ref{sec:Ksep}. Combining  \eqref{A_0_lower_bound} and \eqref{estimate_of_H_pm}, the definition of $\widetilde{\eta}_\ep$ \eqref{eta_ep}, 
we have that $\exists C > 0$ (universal) such that
\begin{align}
T_2 &\lesssim {\ep}^{-1} \ln \ep^{-1}\sup_{y+1\in[\ep^{3/2},\ep]}\left(\frac{y+1}{\ep}\right)^{3}\int_{y}^{1}e^{\al c_i t}\left\lan \frac{-1-c_r}{\ep}+\delta i\right\ran^{1/2}e^{-C^{-1}|\frac{y+1}{\ep}||\frac{-1-c_r}{\ep}|^{1/2}} d{ c_r}\nonumber\\
&\quad\quad+{\ep}^{-1} \ln \ep^{-1}\sup_{y+1\in[\ep^{3/2},\ep]}\left(\frac{y+ 1}{\ep}\right)^{3}\int_{-1}^{y}e^{\al c_i t}e^{-\frac{{2}}{3}Re((\frac{y- c_r}{\ep}+\delta i)^{3/2}e^{i\pi/4})} d{c_r}\nonumber\\
&\lesssim \ln \ep^{-1}\sup_{y+1\in[\ep^{3/2},\ep]} \int_{y}^{1}e^{\al c_i t}\left\lan \frac{(1+c_r)(y+1)^2}{\ep^{3}}\right\ran^{1/2}e^{-C^{-1}|\frac{y+1}{\ep}||\frac{1+c_r}{\ep}|^{1/2}} \frac{(y+1)^2}{\ep^{3}}dc_r 
+\ln \ep^{-1}e^{-\lambda t} \nonumber\\
&\lesssim \ln \ep^{-1}e^{-\lambda t}.\label{T2}
\end{align}
Next we estimate the term $T_3$ in \eqref{T1T2T3}, which are similar to those made on $T_{\pm,3}$ in \eqref{T+123T-123}.
Dividing into three contributions:
\begin{align}
T_3  & = {\ep}^{-1} \ln{\ep}^{-1} \sup_{d(y,- 1)\geq \ep}\int_{-{1}}^{{1}}\left(\mathbf{1}_{-1 \leq c_r \leq y} (\mathbf{1}_{\abs{c_r - y} \leq \frac{y+1}{2}} + \mathbf{1}_{\abs{c_r - y} > \frac{y+1}{2}}) + \mathbf{1}_{c_r \geq y} \right) \frac{\widetilde{\eta}_\ep(y) e^{\al c_i t}}{ |A_0(\frac{-1-c_r }{\ep}+\delta i)|}|H_-(Y)| d{ c_r}\nonumber\\
& = :T_{31} + T_{32} + T_{33}.  
\end{align}
By \eqref{A_0_lower_bound} and \eqref{estimate_of_H_pm}, the definition of  $\widetilde{\eta}_\ep$ \eqref{eta_ep},  that $|c_r -y|\leq \frac{y+1}{2}$ implies that $|c_r+1|\geq\frac{y+1}{2}$, we have (choosing also $c_\star$ small), 
\begin{align*}
T_{31} 
\lesssim&{\ep}^{-1} \ln{\ep}^{-1} \sup_{d(y,\pm 1)\geq \ep} e^{c_\star \frac{(y+1)^{3/2}}{\ep^{3/2}}}\int_{-1}^{y}\mathbf{1}_{|c_r+1|\geq\frac{y+1}{2}}e^{\al c_i t}\left\lan\frac{-1-c_r}{\ep}\right\ran^{3/4} e^{-\frac{{2}}{3}Re((\frac{- c_r-1}{\ep}+\delta i)^{3/2}e^{i\pi/4})} d{c_r}\\
\lesssim & \ln \ep^{-1} e^{-\lambda t}. 
\end{align*}
Using $Y_r=\frac{y-c_r}{\ep}\geq \frac{y+1}{2\ep}$, the lower bound $|A_0|\gtrsim 1$ \eqref{A_0_lower_bound} and \eqref{estimate_of_H_pm}, we have (choosing $c_\star$ small),  
\begin{align*}
T_{32}
\lesssim& {\ep}^{-1} \ln \ep^{-1}\sup_{d(y,\pm 1)\geq\ep}e^{c_\star \frac{(y+1)^{3/2}}{\ep^{3/2}}}\int_{-1}^{y}e^{\al c_i t}e^{-C^{-1}(\frac{y+1}{\ep})^{\frac{3}{2}}}e^{-\frac{1}{3}Re (Y^{3/2} e^{i\pi/4})} d{c_r}  \lesssim  \ln \ep^{-1}e^{-\lambda t}. 
\end{align*}
For $T_{33}$, since $c_r\geq y$, we have $|-1-c_r|\geq y+1$. We apply the same argument used to treat $T_{+,3}$ in \eqref{T+123T-123} in Section \ref{sec:Ksep}.
Combining  \eqref{A_0_lower_bound}, \eqref{estimate_of_H_pm}, \eqref{eta_ep} and choosing $c_\star$ completes the desired estimate, which is omitted for the sake of brevity.  
Combining the estimate of $T_3$ with \eqref{T1T2T3}, \eqref{T1} and \eqref{T2} yields  \eqref{K_1_sup_x_L_1_z_weighted}.

\noindent
\textbf{Step 2: Weighted estimate $\sup_{z\in[-1,1]}||\widetilde{\eta}_\ep K_{1;nc}(\cdot,z)||_{L^1_y}$} \\  
As above, we estimate the contribution associated to $\mathcal{R}_f$ first. Subdividing into the three natural contributions: 
\begin{align}
\ep^{-2}\int_{z}^1 e^{\al r}H_-(R)dr H_+(Z)= &\ep^{-2}\int_z^1 (\mathbf{1}_{z\leq  c_r\leq  r} +\mathbf{1}_{z\leq  r\leq  c_r} +\mathbf{1}_{c_r\leq  z\leq  r}) e^{\al r}H_-(R)dr H_+(Z)\nonumber\\
=: & I_2+II_2+III_2.\label{factor_RZ}
\end{align}
We estimate the term $I_2$ in \eqref{factor_RZ}. Since the function $H_+(Z)$ is small when $ Z_r\leq 0$ and $H_-(R)$ is small when $ R_r\geq 0$, this term is small. Indeed, applying \eqref{estimate_of_H_pm} implies 
\begin{align*}
I_2
\lesssim&{\ep^{-1}} e^\al\int_{0}^{\frac{1-c_r}{\ep}} \frac{ \mathbf{1}_{z\leq  c_r\leq  r} }{\lan Z\ran^{1/4}\lan R\ran^{1/4}} e^{-\frac{2}{3}Re(Z^{3/2}{e^{i\pi/4}})}e^{-\frac{2}{3}Re(R^{3/2}e^{i\pi/4})}dR_r\\
\lesssim& \ep^{-1}e^\al \frac{1}{\lan Z\ran^{1/4}}e^{-\frac{{2}}{3}Re(Z^{3/2}e^{i\pi/4})}\lesssim \ep^{-1}e^\al \frac{1}{\lan Z\ran}.
\end{align*}
For $II_2$ we apply Lemma \ref{Gain_in_poly_power}: 
\begin{align*}
II_2
\lesssim &\ep^{-1}\frac{e^\al}{\lan Z\ran^{1/4}}\int_{Z_r}^{0} \frac{\mathbf{1}_{z\leq  r\leq  c_r}}{\lan R\ran^{1/4}}e^{-\frac{{2}}{3}Re(R^{3/2}e^{-3i\pi/4})}dR_r e^{-\frac{{2}}{3}Re(Z^{3/2}e^{i\pi/4})}\lesssim\ep^{-1}e^{\al}\frac{1}{\lan Z\ran}.
\end{align*}
The estimate on $III_2$ is similar and is omitted for brevity. 
Combining the above with \eqref{factor_RZ}, we have 
\begin{align}
\sup_{z\in[-1,1]}&||\widetilde{\eta}_\ep K_{1;nc}(y,z)||_{L_y^1}\nonumber\\
\lesssim &e^{-\lambda t}\sup_{z\in[-1,1]}\ep^{-1}e^{\al}&\norm{\int_{-1}^1\frac{\widetilde{\eta}_\ep}{| D(\al,c)|\lan Z\ran}\bigg(\aaa\bigg)\abs{H_-(Y)}dc_r}_{L_y^1}. \label{K_1_sup_z_L_1_y_weighted_1}
\end{align} 
Dividing the $y$ integral into $d(y,\pm 1) \leq \eps^{3/2}$, $d(y,\pm 1) \in (\eps^{3/2},\eps)$, and $d(y,\pm1) \geq \eps$ allows to apply the same method
we used to prove \eqref{K_1_sup_x_L_1_z_weighted}. The details are omitted for brevity.

\textbf{Comments on remaining terms: } \\ 
The first four terms, $K_1,..,K_4$ are roughly analogous.
To treat the latter four terms, one needs to take advantage of the fact that $Ai(y)=\overline{Ai(\overline{y})}$ to rewrite $H_+(Y)=Ai(e^{5\pi i/6}Y)$ as $\overline{Ai(e^{- 5\pi i/6}\overline{Y})}=\overline{Ai\left(e^{\pi i/6}\left(\frac{c_r-y}{\ep}+\delta i\right)\right)}$. Note also the inequality
\begin{align}
e^{\frac{2}{3}Re(e^{i\pi/4}(\frac{c_r-y}{\ep}+\delta i)
^{3/2})}\lesssim e^{\frac{2}{3}Re(e^{i\pi/4}(\frac{c_r-1}{\ep}+\delta i)^{1/2}(\frac{c_r-y}{\ep}+\delta i))},\quad c_r-1\leq c_r-y\leq 0 .
\end{align}
Using these observations, treating the last four terms is then analogous to the first four; the details are omitted for brevity.

\section{$L^2$ case: Resolvent estimates in the tails and connection zones} \label{sec:Gen2}
\subsection{General set-up of the tail  and connection zones ($t\geq 1$)}
For $t\ geq 1$, as in Section \ref{sec:Gen1}, we use $\lambda = \alpha^2 \nu + \delta \alpha^{2/3}\nu^{1/3}$. 
Recall from Section \ref{sec:L2OrrSom} that the contours defining the $K_j$ in \eqref{om_b} are chosen to depend on $y$. Indeed, if $d(y,\pm 1) \geq \ep^{3/2}$ we use the straight contour
\begin{align}
\Gamma_{\mathrm{str}}=\{c|c_i\equiv-\al\nu - \delta \alpha^{-1/3}\nu^{1/3}\}. 
\end{align}
To treat these contributions one then deduces
\begin{align}
\sup_{y\in[-1 + \eps^{3/2},1 - \eps^{3/2}]}||\eta_\ep(y) K_{j;fc}(y,\cdot)||_{L_z^1} & \lesssim \ln\ep^{-1} e^{-\lambda t}, \\ 
\sup_{z\in[-1,1]}||\eta_\ep(\cdot)K_{j;fc}(\cdot,z)||_{L_y^1(d (y,\pm 1) \geq \eps^{3/2}) }& \lesssim \ln \ep^{-3/2}  \ln \ln \ep^{-1} e^{-\lambda t}.
\end{align}
These estimates are roughly analogous to that performed in Section \ref{sec:Gen1}, and we omit the details for the sake of brevity. 

If $\mathrm{d}(y,\pm 1)\leq \ep^{3/2}$, i.e. the core of the boundary layer, then we further adjust the contour.
We first define the connection region $\Gamma_{co}^\pm$:
\begin{align}
\Gamma_{co}^+=&\cup_{j=1}^2\Gamma_j^+;\\
\Gamma_{co}^-=&\cup_{j=1}^2\Gamma_j^-,
\end{align}
(the subscript `$co$' denotes `connection') where the $\Gamma_j^\pm$ sub-segments are defined as follows for an angle $\theta$ such that  $\abs{\tan \theta} < \frac{1}{100}$, 
\begin{align*}
\Gamma_1^+=&\{c|c_r\in[1,200], c_i\equiv-\al\nu-\delta \al^{-1/3}\nu^{1/3}\};\\
\Gamma_1^-=&\{c|c_r\in[-200,-1], c_i\equiv-\al\nu-\delta\al^{-1/3}\nu^{1/3}\};\\
\Gamma_2^+=&\{c| c_i+\al\nu+\delta\al^{-1/3}\nu^{1/3}=-\tan(\theta)(c_r-200)_+, \quad c_r\in [200,-1000\ln (\al^{-1}\nu)]\};\\
\Gamma_2^-=&\{c|c_i+\al\nu+\delta\al^{-1/3}\nu^{1/3}=-\tan(\theta)(-c_r-200)_+, \quad c_r\in[1000\ln (\al^{-1}\nu),-200]\}.
\end{align*}
Note the $\Gamma_j^\pm$ are straight lines. Next we define the tail part of the contour:
\begin{align*}
\Gamma_t^+=&\{c| c_i+\al\nu+\delta\al^{-1/3}\nu^{1/3}=-\tan(\theta)(c_r-200)_+, \quad c_r\in [-1000\ln (\al^{-1}\nu),\infty)\};\\
\Gamma_t^-=&\{c|c_i+\al\nu+\delta\al^{-1/3}\nu^{1/3}=-\tan(\theta)(-c_r-200)_+, \quad c_r\in(-\infty,1000\ln (\al^{-1}\nu)]\}.
\end{align*}
Here the subscript `$t$' denotes the `tail' and the superscripts `$+,-$' denote the upper half and lower half of the tails, respectively.
Recall the sketch of the contour in Figure \ref{Fig:Contour} in Section \ref{sec:L2OrrSom}. 

We next sketch how to estimate the resolvent in the connection and tail regions. We first consider the $t\geq1$ case.
First, since the $\Gamma_1^\pm$ regions are close to the continuous spectrum of the Euler equation, we can estimate the resolvent in the same way in Section \ref{Sec:NearSpec_1}; hence these contributions are omitted. For the second piece, $\Gamma_2^\pm$ of $\Gamma_{co}$, we choose the length so that the loss is essentially logarithmic (see e.g. \eqref{ineq:lnln} below). In the tail region $\Gamma_t^\pm$,  the factor $e^{\al c_i t}$ in the resolvent provides integrability in $c$ for $t\geq 1$.

For the initial time layer, i.e., $t\leq 1$, we choose again the straight contour $\Gamma_{\al}:=\{c|c_i\equiv -\al\nu-\delta\al^{-1/3}\nu^{1/3}\}$. Here, we use the weaker weight $\widetilde\eta_\ep^{in}$ \eqref{weight_eta_ep} to compensate for the loss of integrability as $c$ approaches infinity in the boundary layer contribution.
Away from the boundary layer, the estimates proceed as in the case $t \geq1$.

In the remainder of the section, we provide a few more details on the above sketch. 

\subsection{Connection and tail estimates}\label{Sec:connection_tail}
Here we focus on the case $d(y,\pm 1) \leq \eps^{3/2}$.
We decompose the proof into three steps: $\Gamma^{\pm}_1$, $\Gamma^{\pm}_2$, $\Gamma_t^{\pm}$, that is we write
\begin{align}
K_{1;fc}(y,z) & = \sum_{+,-} \left(\int_{\Gamma_1^{\pm}} + \int_{\Gamma_2^{\pm}} + \int_{\Gamma_t^{\pm}}\right)  \\ & \quad \quad \times \frac{ 2 e^{-i \al c t}}{\pi D(\al,c)}\bigg(\aaa\bigg)\int_{z}^1 \epsilon^{-2} e^{\al r}H_-(R)drH_+(Z)H_-(Y)d c dr \\
& =\sum_{+,-}\left( K^{\pm}_{1,1} + K^{\pm}_{1,2} + K^{\pm}_{1,t}\right). 
\end{align}
As in Section \ref{sec:Gen1}, we only treat the case $K_1$; the other kernels are similar and are omitted for the sake of brevity. 

\noindent
\textbf{Step 1: Estimate on $K_{1,1}^{\pm}$.}\\ 
Recall that we consider $d(y,\pm 1)\leq \ep^{3/2}$, hence $\widetilde\eta_\ep\equiv\ep^{3/2}$.
As a result, it suffices to prove: 
\begin{align}\label{K_11_Schur}
\frac{\ep^{3/2}}{\ln \ep^{-1}}\sup_{d(y,\pm1)\leq \ep^{3/2}}||K^\pm_{1,1}(y,\cdot)||_{L_z^1}+ \frac{1}{\ln \ep^{-1}}\sup_{z\in[-1,1]}||K^\pm_{1,1}(\cdot,z)||_{L^1_y(d(y,\pm 1) \leq \ep^{3/2})}\lesssim 1.
\end{align}
As in Section \ref{sec:Gen1}, the first step is find estimates corresponding to the free resolvent.
In this case, these are: $\forall c\in \Gamma_{co}\cup \Gamma_t$,
\begin{align}
&\epsilon^{-2}\int_{-1}^1\int_{z}^1  e^{\al r}|H_-(R)||H_+(Z)|drdz\lesssim e^\al\left\lan\frac{1+|c_r|}{\ep}\right\ran;\label{FRz}\\
&\ep^{-2}\int_{z}^1 e^{\al r}|H_-(R)| |H_+(Z)|dr\lesssim\frac{e^\al}{\lan Z\ran}.\label{FRy}
\end{align}
After replacing Lemma \ref{Gain_in_poly_power} by Lemma \ref{lem:Connection_Zone_Gain_in_poly_power},  the proofs of \eqref{FRz} and \eqref{FRy} are similar to the estimates of $I_1,I_2, I_3$ and $II_1,II_2, II_3$ terms in \eqref{RZ_factor_2} and \eqref{factor_RZ}. Hence, the arguments are omitted for the sake of brevity.
The remaining estimates required  for \eqref{K_11_Schur} are essentially the same as that found in Section \ref{sec:Gen1} and are hence omitted for the sake of brevity. 

\noindent
\textbf{Step 2: Estimate on $K_{1,2}^{\pm}$.}\\
For simplicity, fix $K_{1,2}^+$; the case of $K_{1,2}^-$ is analogous. 
First note the following, 
\begin{equation}
\frac{1}{\lan Y\ran^{1/4}}e^{\frac{2}{3}Re(Y^{3/2}e^{i\pi/4})}\lesssim \frac{e^{\frac{2}{3}Re(( \frac{-1-c_r}{\ep}+\frac{-c_i-\al\nu}{\ep}i)^{3/2}e^{i\pi/4})}}{\lan \frac{-1-c_r}{\ep}+\frac{-c_i-\al\nu}{\ep}i\ran^{1/4}},\quad\forall c\in \Gamma_2^+.\label{Mono_Q}
\end{equation}
This is proved by using monotonicity in the exponent and applying $ \lan Y\ran\approx\left\lan \frac{-1-c_r}{\ep}+\frac{-c_i-\al\nu}{\ep}i\right\ran$ for $c\in \Gamma_2^+$.
As in \eqref{K_11_Schur}, it suffices to prove
\begin{align}
\frac{\ep^{3/2}}{\ln ^{{3/2}}\ep^{-1} \ln \ln \eps^{-1}}\sup_{d(y,\pm1)\leq \ep^{3/2}}||K^\pm_{1,1}(y,\cdot)||_{L_z^1} & \nonumber \\ \quad +  \frac{1}{\ln \ep^{-1}}\sup_{z\in[-1,1]}||K^\pm_{1,1}(\cdot,z)||_{L^1_y(d(y,\pm 1) \leq \ep^{3/2})} & \lesssim \lan\al\ep\ran. \label{K12_Schur} 
\end{align}
The proof is similar to the arguments in Section \ref{sec:Gen1} with some minor modifications.
Consider simply the first term, which is the one that requires a more different treatment. 

Combining \eqref{FRz}, \eqref{A_0_lower_bound}, \eqref{Mono_Q}, Lemma \ref{lem:Evans_function_est} and Lemma \ref{lem:Connection_Zone_Gain_in_poly_power}, we obtain that
\begin{align}
&\sup_{ \mathrm{dist}(y,\pm 1)\leq\ep^{3/2}}||K_{1,2}^+(y,z)||_{L_z^1} \nonumber \\
&\leq\sup_{ \mathrm{dist}(y,\pm 1)\leq\ep^{3/2}} \int_{{200}}^{{-3000\ln\ep}}\frac{e^\alpha}{ e^\al|A_0(\frac{-1-c_r }{\ep}+\frac{-c_i(c_r)-\al\nu}{\ep} i)|}|H_-(Y)|\ln\left\lan \frac{ |c_r|+1}{\ep}\right\ran {\ep}^{-1}d{ c_r} \nonumber \\
 &\lesssim\sup_{ \mathrm{dist}(y,\pm 1)\leq\ep^{3/2}}\int_{{200}}^{-{3000\ln\ep}}\frac{\lan \frac{-c_r-1}{\ep}+\frac{-c_i-\al\nu}{\ep} i\ran ^{3/4} }{\lan Y\ran^{1/4}}\nonumber \\
 &\quad\quad\quad\quad\quad\quad\quad\quad\quad\times e^{-\frac{{2}}{3}Re((\frac{- c_r-1}{\ep}+\frac{-c_i-\al\nu}{\ep} i)^{3/2}e^{i\pi/4})}e^{\frac{{2}}{3}Re(Y^{3/2}e^{\pi i/4})}\ln\left\lan \frac{ |c_r|+1}{\ep}\right\ran {\ep}^{-1}d{ c_r}\nonumber \\
&\lesssim \int^{{-3000\ln \ep}}_{{200}}\brak{\frac{-c_r-1}{\ep}+\frac{-c_i-\al\nu}{\ep} i}^{1/2} \ln\left\lan \frac{ |c_r|+1}{\ep}\right\ran{\ep}^{-1} d{ c_r} \nonumber \\
&\lesssim \left(\frac{-\ln \ep}{\ep}\right)^{3/2}\ln\ln{\ep}^{-1}, \label{ineq:lnln}
\end{align} 
This is consistent with \eqref{K12_Schur}.

\textbf{Step 3: tail region $\Gamma_t^\pm$  estimates.} \\ 
Since the tail contour $\Gamma_t^\pm$ has the same slope as the connection contour $\Gamma_2^\pm$, the main estimates are similar and so are omitted for the sake of brevity.
The primary gain here comes from the $e^{\al c_i t}$.  
The easest way to use this is to note that for $t\geq 1$, 
\begin{align*} 
e^{\al c_i t}\leq e^{\al c_i}\lesssim e^{15\al\ln \ep}e^{-\frac{\al c_i}{2}}\lesssim\ep^{15\al}e^{-\frac{\al |c_r|}{200}}.
\end{align*}
Combining this with the arguments in  $\Gamma_2^{\pm}$ and are hence omitted for the sake of brevity. 

\subsection{Initial time layer estimate $t\leq 1$}
In this subsection, we estimate the semigroup in the initial time layer.
As in Section \ref{sec:Ksep}, we treat the high modes $\alpha^2 \nu \geq C_0$ differently from the low modes $\alpha^2 \nu < C_0$.
We will focus on the low mode treatment $\alpha^2 \nu < C_0$; the high mode treatments are done analogously using the same shift as in Section \ref{sec:Ksep} and we  hence omit this argument for the sake of brevity. 

Recall \eqref{om_b} and the choice $\Gamma_\al=\{c|c_i=-(1-\kappa)\al\nu-\delta\al^{-1/3}\nu^{1/3}\}$ for $t\leq 1$.
The problem is for $\abs{c_r} \gtrsim 1$ and hence we divide the kernels as follows: 
\begin{align}
K_1(y,z) & = \left(\int_{-10}^{10} + \int_{10}^\infty + \int_{-\infty}^{-10} \right)  \\ & \quad \quad \times e^{-i \al c t}\frac{ 2}{\pi D(\al,c)}\bigg(\aaa\bigg)\int_{z}^1 \epsilon^{-2} e^{\al r}H_-(R)drH_+(Z)H_-(Y)d c_r dr \\
& =: K_{1,E}^{in} + K_{1}^{in,+} + K_{1}^{in,-};
\end{align}
we sketch only the case $K_1$; the other $K_j$ are analogous.  

To estimate the initial time-layer via Schur's test, it suffices to deduce 
\begin{align}\label{Initial_larer_Lp_goal}
\sup_{y \in [-1,1]}||\widetilde\eta_\ep^{in}K_1^{in,\pm}(y,\cdot)||_{L_z^1}+\sup_{z \in [-1,1]}||\widetilde\eta_\ep^{in}K_1^{in,\pm}(\cdot,z)||_{L_y^1}\lesssim 1.
\end{align}
To see why the vanishing weight $\eta_\ep^{in}$ translates to a gain, we treat as an example $\sup_{y \in [-1,1]}||\eta_{\ep}^{in}K_1^{in,+}(y,\cdot)||_{L_z^1}$; the other estimates are similar or easier. 
Using \eqref{weight_eta_ep}, \eqref{A_0_lower_bound}, Lemma \ref{lem:Evans_function_est}, and Lemma \ref{Gain_in_poly_power} we deduce 
\begin{align*}
\sup_y&||\widetilde\eta_\ep^{in}K_{1}^{in,+}(y,z)||_{L_z^1}\\
\leq&\left(\sup_{y+1\leq \ep}+\sup_{y+1\geq \ep}\right)\frac{\widetilde\eta_\ep^{in}(y)}{\eps}\int_{{10}}^{\infty}\frac{ e^\alpha}{ e^\al|A_0(\frac{-1-c_r }{\ep}+\delta i)|}|H_-(Y)|\ln\left\lan \frac{ |c_r|+1}{\ep}\right\ran d c_r\\
\lesssim &\sup_{y+1\leq \ep}\frac{\widetilde\eta_\ep^{in}(y)}{\eps} \int_{{10}}^{\infty}\frac{\lan \frac{-c_r-1}{\ep}+\delta i\ran ^{3/4} }{\lan Y\ran^{1/4}}e^{-\frac{{2}}{3}Re((\frac{- c_r-1}{\ep}+\delta i)^{3/2}e^{i\pi/4})}e^{\frac{{2}}{3}Re(Y^{3/2}e^{\pi i/4})}\ln\left\lan \frac{ |c_r|+1}{\ep}\right\ran d{ c_r}+ 1\\
\lesssim &\sup_{y+1\leq \ep}\left(\frac{y+1}{\ep}\right)^3\int_{{10}}^{\infty}{\left\lan \frac{-c_r-1}{\ep}+\delta i\right\ran ^{1/2} }e^{\frac{{2}}{3}Re(\frac{y+1}{\ep}(\frac{- c_r-1}{\ep}+\delta i)^{1/2}e^{i\pi/4})}\ln\left\lan \frac{ |c_r|+1}{\ep}\right\ran {\ep}^{-1}d{ c_r}+1\\
\lesssim &1 .
\end{align*}
The remaining estimates are analogous and are hence omitted for the sake of brevity.





\appendix
\section{Technical lemmas}\label{sec:Appendix}
\subsection{The free evolution} \label{sec:Free}
Let us briefly recall the properties of the initial value problem 
\begin{align}
&\partial_t \omega + y \partial_x \omega = \nu \Delta \omega, \\ 
&\omega(0) = \omega_{in}, \\ 
& u = \grad^{\perp}\Delta^{-1}\omega. 
\end{align}
Via Fourier transform in both variables $(x,y) \mapsto (\alpha,\eta)$ we derive (see \cite{BMV14} for a modern treatment), 
\begin{align}
\widehat{\omega}(t,\alpha,\eta) & = \widehat{\omega}_{in}(\alpha,\eta + \alpha t) \exp \left( -\nu \alpha^2 t - \nu\int_0^t\abs{\eta + \alpha(t-\tau)}^2 d\tau \right). \label{ineq:ffree}
\end{align}
The following lemma quantifies the enhanced dissipation and inviscid damping. 
\begin{lemma} \label{lem:FreeDecay}
There holds the pointwise-in-time enhanced dissipation estimate
\begin{align}
\norm{\widehat{\omega}(t,\alpha,\cdot)}_{L^2_y} \lesssim e^{-\alpha^2 \nu t - \frac{1}{12} \nu \alpha^2 t^3} \norm{\widehat{\omega}_{in}(\alpha,\cdot)}_{L^2_y},  \label{ineq:FreePtWsDec}
\end{align}
there holds the time-averaged inviscid damping and enhanced dissipation 
\begin{align}
\norm{e^{\nu \alpha^2 t+ \frac{1}{12}\nu \alpha^2 t^3} \widehat{u}(t,\alpha,\cdot) }_{L^2_t L^2_y} \lesssim \frac{1}{\abs{\alpha}^{1/2}}\norm{\widehat{\omega}_{in}(\alpha,\cdot)}_{L^2_y}, \label{ineq:VelDecayFree}
\end{align}
and, more generally, the following $L^2_t H^{-s}_y$ estimate for $s > 1/2$: 
\begin{align}
\norm{e^{\nu \alpha^2 t+ \frac{1}{12}\nu \alpha^{2} t^3} \widehat{\omega}(t,\alpha,\cdot)}_{L^2_t H^{-s}_y} \lesssim \frac{1}{\abs{\alpha}^{1/2}} \norm{\widehat{\omega}_{in}(\alpha,\cdot)}_{L^2}. \label{ineq:H-s}
\end{align} 
\end{lemma}
\begin{proof}
Estimate \eqref{ineq:FreePtWsDec} is immediate from the formula \eqref{ineq:ffree}. 
Estimate \eqref{ineq:VelDecayFree} follows from \eqref{ineq:H-s}. Finally, we observe by \eqref{ineq:ffree}, 
\begin{align*}
\int_0^\infty \int_{-\infty}^\infty e^{2\alpha^2 \nu t + \frac{1}{6}\nu \alpha^2 t^3} \brak{\eta}^{-2s} \abs{\widehat{\omega}(t,\alpha,\eta)}^2 d\eta & \lesssim \int_0^\infty \int_{-\infty}^\infty \brak{\eta}^{-2s} \abs{\widehat{\omega}_{in}(\alpha,\eta+ \alpha t)}^2 d\eta dt \\
& \lesssim \int_{-\infty}^\infty  \left( \int_0^\infty \brak{\eta - \alpha t}^{-2s} dt\right) \abs{\widehat{\omega}_{in}(\alpha,\eta)}^2 d\eta \\
& \lesssim \frac{1}{\abs{\alpha}} \norm{\omega_{in}}_{L^2}^2, 
\end{align*}
which completes the proof of \eqref{ineq:H-s}. 
\end{proof}

\subsection{Airy functions estimates} \label{sec:Airy}
Recall the definition of the homogeneous solutions of the Orr-Sommerfeld problem in terms of Airy functions \eqref{def:HYetc}. 
Standard asymptotics for the Airy functions gives the following

\begin{lemma}\label{lem:estimate_of_Airy}
  The homogeneous solutions $H_-(z)$ and $H_+(z)$ satisfy the following for $z=z_r+\delta i \subset \set{z \in \Complex : 0 < z_i  < K + \frac{1}{100}\abs{z_r}}$,  
\begin{subequations}\label{estimate_of_H_pm}
\begin{align}
|H_-(z)|\lesssim_K & \frac{1}{\lan z\ran^{1/4}}e^{-\frac{{2}}{3}Re(e^{i\pi/4}z^{3/2})}, \quad  z_r\geq 0;\\
|H_-(z)|\lesssim_K & \frac{1}{\lan z\ran^{1/4}}e^{\frac{{2}}{3}Re(e^{i\pi/4}z^{3/2})}, \quad z_r\leq 0;\\
|H_+(z)|\lesssim_K & \frac{1}{\lan z\ran^{1/4}}e^{\frac{{2}}{3}Re(e^{i\pi/4}z^{3/2})},\quad  z_r\geq 0;\\
|H_+(z)|\lesssim_K & \frac{1}{\lan z\ran^{1/4}}e^{-\frac{{2}}{3}Re(e^{i\pi/4}z^{3/2})}, \quad  z_r\leq 0.
\end{align}
\end{subequations}
If $z\in\{z\in\mathbb{C}|z_i<0\}$, then the solutions $H_-(z)$ and $H_+(z)$ satisfy the following
\begin{subequations}\label{estimate_of_H_pm_high_mode}
\begin{align}
|H_-(z)|\lesssim_K & \frac{1}{\lan z\ran^{1/4}}e^{-\frac{{2}}{3}Re(e^{i\pi/4}z^{3/2})}, \quad  z_r\in \rr;\\
|H_+(z)|\lesssim_K & \frac{1}{\lan z\ran^{1/4}}e^{\frac{{2}}{3}Re(e^{i\pi/4}z^{3/2})}, \quad  z_r\in\rr.
\end{align}
\end{subequations}
\end{lemma}
\begin{proof}
Recall the definition of the $H_\pm$ in \eqref{def:HYetc}. 
For $\abs{z} \lesssim 1$, \eqref{estimate_of_H_pm} follows since $\abs{H_{\pm}} \lesssim 1$ on bounded sets.
For $\abs{z}$ larger, we apply the asymptotic expansion for $|z|\rightarrow \infty$, $|\mathrm{ph} z|<\pi-\zeta$, for any $\zeta>0$ (see e.g. \cite{Miller06}):
\begin{align*}
Ai(z)=\frac{1}{2z^{1/4}\sqrt{\pi}}e^{-2z^{3/2}/3}(1+O(|z|^{-3/2})). 
\end{align*}
Note that for the second and fourth inequalities, one must be careful to treat the branch cut \eqref{eq:Branch} correctly. For the estimates \eqref{estimate_of_H_pm_high_mode}, the proof is the same. One only needs to be careful with the fact that now the phase is  $\mathrm{ph}(z)\in(-\pi,0)$ since $z_r<0$.  
\end{proof}

In addition to the standard Airy functions, we also require the integrated version used by Romanov~\cite{Romanov73}, 
\begin{equation}\label{A_0}
A_0(z)=\int_{ze^{i \pi/6}}^{i0+\infty} Ai(t)dt, 
\end{equation}
where for definiteness we take the contour to be the straight line connecting $ze^{i\pi/6}$ to $0$ and then the ray connecting $0$ to $i0 + \infty$. 
The following property of $A_0$ proved by Romanov~\cite{Romanov73} is crucial. 
\begin{lemma}\label{lem_0}[Lemma 2 and following remark, \cite{Romanov73}]
  The function $A_0(z)$ has no zeros in the sector $-\pi\leq \mathrm{ph} (z)\leq 2\pi/3$ and in the half-plane $\{z| z_i \leq \delta_0\}$ for some universal constant $\delta_0>0$. Moreover, the following quantity $a$ is strictly positive
\begin{equation}\label{defn_of_a}
a=a(\delta)=-\max_{ z_i\leq \delta}Re \frac{A_0'(z)}{A_0(z)}>0, \quad\delta\in [0,\delta_0).
\end{equation}
\end{lemma}

We also have the following lower bound of $A_0$, which also follows from [(3.4) \cite{Romanov73}]; see also \cite{Wasow53}. 
\begin{lemma}[Integrated Airy function asymptotics] \label{lem:IntAiry} 
The integrated Airy function \eqref{A_0} has the following lower bounds for
$z=z_r+i z_i  \subset \set{z \in \Complex : 0 < z_i  < K + \frac{1}{100}\abs{z_r}} \cup \set{z \in \Complex: z_i < 0}$,
\begin{subequations}\label{A_0_lower_bound}
\begin{align}
|A_0(z)|\gtrsim_K& \frac{1}{\lan z\ran^{3/4}}e^{\frac{2}{3}Re(z^{3/2}e^{i\pi/4})},\quad z_r\leq 0;\\ 
|A_0(z)|\gtrsim_K &\frac{1}{\lan z\ran^{3/4}}e^{-\frac{2}{3}Re(z^{3/2}e^{i\pi/4})},\quad z_r\geq0.
\end{align}
\end{subequations}
\end{lemma}

\subsection{Power-gain lemmas}
The following lemma is used often in our paper.
Similar estimates are used in e.g. \cite{GN17} and the references therein, however, here we use a different range of complex parameters, so we include a sketch for the sake of completeness. 
\begin{lemma}\label{Gain_in_poly_power}
\textbf{Case 1:}
Consider the spectral parameter $c$ lying on the vertical line $\Gamma_\al:=\{c=c_r-(\al \nu+\delta \ep)i, c_r\in \rr\}$. From \eqref{Langer_variable}, the Langer variables $R,Z$ have the explicit form 
\[
R=\frac{r-c_r}{\ep}+\delta i=:R_r+\delta i,\quad Z=\frac{z-c_r}{\ep}+\delta i=:Z_r+\delta i.
\]
For $\delta$ small enough, the following estimates hold (recall also \eqref{eq:Branch})
\begin{subequations} \label{ineq:powergain}
\begin{align}
\int_0^{Z_r}&\frac{1}{\lan R\ran^{1/4}}e^{\frac{2}{3}Re((R_r+\delta i)^{3/2}e^{i\pi/4})}dR_r\lesssim \frac{1}{\lan Z\ran^{3/4}}e^{\frac{2}{3}Re ((Z_r+\delta i)^{3/2}e^{\pi i/4})},\quad Z_r\geq 0.\label{Gain_in_poly_power_1}\\
\int_{Z_r}^0&\frac{1}{\lan R\ran^{1/4}}e^{\frac{2}{3}Re((R_r+\delta i)^{3/2}e^{i\pi/4})}dR_r\lesssim \frac{1}{\lan Z\ran^{3/4}}e^{\frac{2}{3}Re ((Z_r+\delta i)^{3/2}e^{\pi i/4})},\quad Z_r\leq 0,\label{Gain_in_poly_power_2}\\
\int_{Z_r}^\infty&\frac{1}{\lan R\ran^{1/4}}e^{-\frac{2}{3}Re((R_r+\delta i)^{3/2}e^{i\pi/4})}dR_r\lesssim \frac{1}{\lan Z\ran^{3/4}}e^{-\frac{2}{3}Re ((Z_r+\delta i)^{3/2}e^{\pi i/4})},\quad Z_r\geq 0,\label{Gain_in_poly_power_3}\\
\int_{-\infty}^{Z_r}&\frac{1}{\lan R\ran^{1/4}}e^{-\frac{2}{3}Re((R_r+\delta i)^{3/2}e^{i\pi/4})}dR_r\lesssim \frac{1}{\lan Z\ran^{3/4}}e^{-\frac{2}{3}Re ((Z_r+\delta i)^{3/2}e^{\pi i/4})},\quad Z_r\leq 0.\label{Gain_in_poly_power_4}
\end{align}
\end{subequations}
\textbf{Case 2:} In the high mode case, i.e., $\al^2\nu\geq C_0$, we consider the spectral parameter $c$ lying on the vertical line $\Gamma_\al:=\{c=c_r-(1-\kappa)\al \nu i, c_r\in \rr\}$. From \eqref{Langer_variable}, the Langer variables $R,Z$ have the explicit form 
\[
R=\frac{r-c_r}{\ep}-\kappa\al^{4/3}\nu^{2/3} i,\quad Z=\frac{z-c_r}{\ep}-\kappa\al^{4/3}\nu^{2/3} i.
\]
The following estimates hold (recall also \eqref{eq:Branch})
\begin{subequations} \label{ineq:powergain_2}
\begin{align}
\int_0^{Z_r}&\frac{1}{\lan R\ran^{1/4}}e^{\frac{2}{3}Re((R_r-\kappa\al^{4/3}\nu^{2/3} i)^{3/2}e^{i\pi/4})}dR_r\lesssim \frac{1}{\lan Z\ran^{3/4}}e^{\frac{2}{3}Re ((Z_r-\kappa\al^{4/3}\nu^{2/3} i)^{3/2}e^{\pi i/4})},\quad Z_r\geq 0.\label{Gain_in_poly_power_5}\\
\int_{Z_r}^0&\frac{1}{\lan R\ran^{1/4}}e^{-\frac{2}{3}Re((R_r-\kappa\al^{4/3}\nu^{2/3} i)^{3/2}e^{i\pi/4})}dR_r\lesssim \frac{1}{\lan Z\ran^{3/4}}e^{-\frac{2}{3}Re ((Z_r-\kappa\al^{4/3}\nu^{2/3} i)^{3/2}e^{\pi i/4})},\quad Z_r\leq 0,\label{Gain_in_poly_power_6}\\
\int_{Z_r}^\infty&\frac{1}{\lan R\ran^{1/4}}e^{-\frac{2}{3}Re((R_r-\kappa\al^{4/3}\nu^{2/3} i)^{3/2}e^{i\pi/4})}dR_r\lesssim \frac{1}{\lan Z\ran^{3/4}}e^{-\frac{2}{3}Re ((Z_r-\kappa\al^{4/3}\nu^{2/3} i)^{3/2}e^{\pi i/4})},\quad Z_r\geq 0,\label{Gain_in_poly_power_7}\\
\int_{-\infty}^{Z_r}&\frac{1}{\lan R\ran^{1/4}}e^{\frac{2}{3}Re((R_r-\kappa\al^{4/3}\nu^{2/3} i)^{3/2}e^{i\pi/4})}dR_r\lesssim \frac{1}{\lan Z\ran^{3/4}}e^{\frac{2}{3}Re ((Z_r-\kappa\al^{4/3}\nu^{2/3} i)^{3/2}e^{\pi i/4})},\quad Z_r\leq 0.\label{Gain_in_poly_power_8}
\end{align}
\end{subequations}
\end{lemma}
\begin{proof}
\textbf{Proof of \eqref{ineq:powergain}:} \\ 
The inequality is immediate for $\abs{Z_r} \leq 4$, hence, consider next the case $\abs{Z_r} \geq 4$. 
In this case, the following holds on the integration interval: 
\begin{align}\label{proof_of_Gain_in_poly_power_1}
e^{\frac{2}{3}Re((R_r+\delta i)^{3/2}e^{\pi i/4})}\leq e^{\frac{2}{3}Re((Z_r+\delta i)^{1/2}(R_r+\delta i)e^{i\pi/4})},\quad Z_r\geq R_r\geq 2;
 \end{align}
this follows by e.g. differentiation. 
Next, we decompose the integral in \eqref{Gain_in_poly_power_1} into two parts:
\begin{align}
\int_0^{Z_r}\frac{1}{\lan R\ran^{1/4}}e^{\frac{2}{3}Re((R_r+\delta i)^{3/2}e^{i\pi/4})}dR_r&=\left(\int_0^{Z_r-\sqrt{Z_r}}+\int_{Z_r-\sqrt{Z_r}}^{Z_r}\right)\frac{1}{\lan R\ran^{1/4}}e^{\frac{2}{3}Re((R_r+\delta i)^{3/2}e^{i\pi/4})}dR_r \nonumber \\ & =:T_1+T_2.\label{proof_of_gain_in_poly_power_T_1_T_2}
\end{align}
First consider $T_2$. Combining \eqref{proof_of_Gain_in_poly_power_1} and the following relation on the integration interval, 
\begin{equation*}
\lan R\ran
\geq (1+|Z_r-\sqrt{Z_r}|^2+\delta^2)^{1/2}\geq (1+\frac{1}{4}|Z|^2)^{1/2}\geq \frac{1}{2}\lan Z\ran, 
\end{equation*}
together with $\delta\leq \frac{1}{100}\leq \frac{1}{400}Z_r$, we have 
\begin{align}
T_2\lesssim& \int_{Z_r-\sqrt{Z_r}}^{Z_r}\frac{1}{\lan Z\ran^{1/4}}e^{\frac{2}{3}Re((Z_r+\delta i)^{1/2}(R_r+\delta i)e^{i\pi/4})}dR_r
\lesssim\frac{1}{\lan Z\ran^{1/4}}\frac{1}{Re((Z_r+\delta i)^{1/2}e^{i\pi/4})}e^{\frac{2}{3}Re((Z_r+\delta i)^{3/2}e^{i\pi/4})}\nonumber\\
\lesssim& \frac{1}{\lan Z\ran^{3/4}}e^{\frac{2}{3}Re((Z_r+\delta i)^{3/2}e^{i\pi/4})}.\label{proof_of_gain_in_poly_power_T_2}
\end{align}
This completes the treatment of the $T_2$ term.
The $T_1$ term in \eqref{proof_of_gain_in_poly_power_T_1_T_2}, using \eqref{proof_of_Gain_in_poly_power_1} and the fact that $\delta\leq \frac{1}{100}\leq \frac{1}{400} \abs{Z_r} $ is
easier and is hence omitted for the sake of brevity. 
Hence \eqref{Gain_in_poly_power_1} follows. 

The remaining inequalities \eqref{Gain_in_poly_power_1}--\eqref{Gain_in_poly_power_4}  hold by analogous arguments once the suitable analogue of \eqref{proof_of_Gain_in_poly_power_1} is similarly deduced via differentiation.
We omit the details for the sake of brevity. 

\textbf{Proof of \eqref{ineq:powergain_2}:}\\
In this case, $\Gamma_\al:=\{c=c_r-(1-\kappa)\al\nu i,\quad c_r\in\rr\}$. 
We prove the following monotonicity inequalities adapted for use in high frequencies: 
\begin{align*}
e^{\frac{2}{3}Re((Y_r-\kappa\al^{4/3}\nu^{2/3}i)^{3/2}e^{i\pi/4})}\leq&
e^{\frac{2}{3}Re((W_r-\kappa\al^{4/3}\nu^{2/3}i)^{3/2}e^{i\pi/4})},\quad W_r\geq Y_r\geq 0;\\
e^{\frac{2}{3}Re((Y_r-\kappa\al^{4/3}\nu^{2/3}i)^{3/2}e^{i\pi/4})}\leq&
e^{\frac{2}{3}Re((W_r-\kappa\al^{4/3}\nu^{2/3}i)^{1/2}(Y_r-\kappa\al^{4/3}\nu^{2/3}i)e^{i\pi/4})},\quad W_r\geq Y_r\geq 0;\\
e^{-\frac{2}{3}Re((Y_r-\kappa\al^{4/3}\nu^{2/3}i)^{3/2}e^{i\pi/4})}\leq&
e^{-\frac{2}{3}Re((W_r-\kappa\al^{4/3}\nu^{2/3}i)^{3/2}e^{i\pi/4})},\quad W_r\leq Y_r\leq 0.  \\
e^{-\frac{2}{3}Re((Y_r-\kappa\al^{4/3}\nu^{2/3}i)^{3/2}e^{i\pi/4})}\leq&
e^{-\frac{2}{3}Re((W_r-\kappa\al^{4/3}\nu^{2/3}i)^{1/2}(Y_r-\kappa\al^{4/3}\nu^{2/3}i)e^{i\pi/4})},\quad W_r\leq Y_r\leq 0. 
\end{align*}
Consider the case $W_r\leq Y_r\leq 0$ i.e., the third and fourth inequalities. First, 
\begin{align*}
-\frac{d}{dw}Re((w-\kappa\al^{4/3}\nu^{2/3}i)^{1/2}(Y_r-\kappa\al^{4/3}\nu^{2/3}i)e^{i\pi/4})=-\frac{1}{2}Re((w-\kappa\al^{4/3}\nu^{2/3}i)^{-1/2}(Y_r-\kappa\al^{4/3}\nu^{2/3}i)e^{i\pi/4}).
\end{align*}
Since $W_r\leq Y_r$, $\mathrm{ph}(W_r-\kappa\al^{4/3}\nu^{2/3}i)\leq \mathrm{ph}(Y_r-\kappa\al^{4/3}\nu^{2/3}i)$, so that the derivative is negative and the fourth inequality follows.
The third inequality follows by a similar argument. 
For the $W_r\geq Y_r\geq 0$ case, we have
\begin{align*}
\frac{d}{dw}Re((w-\kappa\al^{4/3}\nu^{2/3}i)^{1/2}(Y_r-\kappa\al^{4/3}\nu^{2/3}i)e^{i\pi/4})=\frac{1}{2}Re((w-\kappa\al^{4/3}\nu^{2/3}i)^{-1/2}(Y_r-\kappa\al^{4/3}\nu^{2/3}i)e^{i\pi/4}).
\end{align*}
Since $W_r\geq Y_r\geq 0$, we have $0\leq-\frac{1}{2}\mathrm{ph}(W_r-\kappa\al^{4/3}\nu^{2/3}i)\leq -\frac{1}{2}\mathrm{ph}(Y_r-\kappa\al^{4/3}\nu^{2/3}i)$ and hence the derivative is positive. This completes the proof of the second inequality; the first inequality follows  similarly. 
Combining these inequalities with the argument used to prove \eqref{ineq:powergain} yields the inequalities \eqref{ineq:powergain_2}. 
\end{proof}
By a similar argument, one can prove the following lemma (used in Section \ref{sec:Gen2}). 

\begin{lemma}\label{lem:Connection_Zone_Gain_in_poly_power}
Consider the spectral parameter $c$ lying on the half-lines as in Section \ref{sec:Gen2},
\begin{align*}
\Gamma^+=&\{c| c_i+\al\nu+\delta\al^{-1/3}\nu^{1/3}=-\tan(\theta)(c_r-200)_+, \quad c_r\in [200,\infty)\};\\
\Gamma^-=&\{c|c_i+\al\nu+\delta\al^{-1/3}\nu^{1/3}=-\tan(\theta)(-c_r-200)_+, \quad c_r\in(-\infty,-200]\}.
\end{align*}
The Langer variables are 
\[R=\frac{r-c_r}{\ep}-\frac{c_i+\al \nu}{\ep}i =R_r+iR_i,\quad Z=\frac{z-c_r}{\ep}-\frac{c_i+\al \nu}{\ep}i=Z_r+iZ_i.\] 
Then the following estimates are satisfied (recall also \eqref{eq:Branch})
\begin{subequations}
\begin{align}
\int_{\frac{-1-c_r}{\ep}}^{Z_r}&\frac{1}{\lan R\ran^{1/4}}e^{\frac{2}{3}Re((R_r+ iR_i)^{3/2}e^{i\pi/4})}dR_r\lesssim \frac{1}{\lan Z\ran^{3/4}}e^{\frac{2}{3}Re ((Z_r+ iZ_i)^{3/2}e^{\pi i/4})},\quad Z_r\geq \max\left\{0,\frac{-1-c_r}{\ep}\right\};\label{Gain_in_poly_power_1_1}\\
\int_{Z_r}^{\frac{1-c_r}{\ep}}&\frac{1}{\lan R\ran^{1/4}}e^{\frac{2}{3}Re((R_r+iR_i )^{3/2}e^{i\pi/4})}dR_r\lesssim \frac{1}{\lan Z\ran^{3/4}}e^{\frac{2}{3}Re ((Z_r+ iZ_i)^{3/2}e^{\pi i/4})},\quad Z_r\leq \min\left\{0,\frac{1-c_r}{\ep}\right\};\label{Gain_in_poly_power_2_1}
\\
\int_{Z_r}^\infty&\frac{1}{\lan R\ran^{1/4}}e^{-\frac{2}{3}Re((R_r+iR_i)^{3/2}e^{i\pi/4})}dR_r\lesssim\frac{1}{\lan Z\ran^{3/4}}e^{-\frac{2}{3}
Re((Z_r+iZ_i)^{3/2}e^{i\pi/4})},\quad Z_r\geq 0;\\
\int_{-\infty}^{Z_r}&\frac{1}{\lan R\ran^{1/4}}e^{-\frac{2}{3}Re((R_r+iR_i)^{3/2}e^{i\pi/4})}dR_r\lesssim\frac{1}{\lan Z\ran^{3/4}}e^{-\frac{2}{3}Re((Z_r+iZ_i)^{3/2}e^{i\pi/4})},\quad Z_r\leq 0.
\end{align}
\end{subequations}
\end{lemma}

\subsection{Detailed Green's function} \label{sec:Green}
The following denotes the full expression of the boundary resolvent $\mathcal{R}_b$ directly as an integral operator on $\widehat{\omega}_{in}$; see \ref{om_b}:
\begin{align}\omega_b=&-\int_{\Gamma_\al}\frac{ 2\epsilon^{-2}}{\pi D(\al,c)}\left(\aaa\right)\int_{-1}^1{\bigg[\int_{z}^1  e^{\al r}H_-(R)dr}\bigg] H_+(Z)H_-(Y) \widehat{\omega}_{in}(\alpha,z) dzdc\nonumber\\
&-\int_{\Gamma_\al}\frac{ 2\epsilon^{-2}}{\pi D(\al,c)}\left(\aaa\right)\int_{-1}^1 {\bigg[\int_{-1}^z e^{\al r} H_+(R)dr\bigg]} H_-(Z)H_-(Y) \widehat{\omega}_{in}(\alpha,z)dzdc\nonumber\\
&-\int_{\Gamma_\al}\frac{ 2\epsilon^{-2}}{\pi D(\al,c)}\left(\bb\right)\int_{-1}^1 {\bigg[\int_{z}^1e^{-\al r}H_-(R)dr\bigg]}  H_+(Z)H_-(Y) \widehat{\omega}_{in}(\alpha,z)dzdc\nonumber\\
&-\int_{\Gamma_\al}\frac{ 2\epsilon^{-2}}{\pi D(\al,c)}\left(\bb\right)\int_{-1}^1 {\bigg[\int_{-1}^ze^{-\al r}H_+(R)dr \bigg]}  H_-(Z)H_-(Y) \widehat{\omega}_{in}(\alpha,z)dzdc\nonumber\\
&-\int_{\Gamma_\al}\frac{ 2\epsilon^{-2}}{\pi D(\al,c)}\left(\cc\right)\int_{-1}^1{\bigg[\int_{z}^1  e^{\al r}H_-(R)dr}\bigg] H_+(Z)H_+(Y)\widehat{\omega}_{in}(\alpha,z) dzdc\nonumber\\
&-\int_{\Gamma_\al}\frac{ 2\epsilon^{-2}}{\pi D(\al,c)}\left(\cc\right)\int_{-1}^1 {\bigg[\int_{-1}^z e^{\al r} H_+(R)dr\bigg]} H_-(Z)H_+(Y) \widehat{\omega}_{in}(\alpha,z) dzdc\nonumber\\
&-\int_{\Gamma_\al}\frac{ 2\epsilon^{-2}}{\pi D(\al,c)}\left(\ddd\right)\int_{-1}^1 {\bigg[\int_{z}^1e^{-\al r}H_-(R)dr\bigg]}  H_+(Z)H_+(Y) \widehat{\omega}_{in}(\alpha,z) dzdc\nonumber\\
&-\int_{\Gamma_\al}\frac{ 2\epsilon^{-2}}{\pi D(\al,c)}\left(\ddd\right)\int_{-1}^1 {\bigg[\int_{-1}^ze^{-\al r}H_+(R)dr \bigg]}  H_-(Z)H_+(Y) \widehat{\omega}_{in}(\alpha,z) dzdc\nonumber\\
=:&\sum_{j=1}^8\int_{-1}^{1}K_j(y,z)  \widehat{\omega}_{in}(\alpha,z)dz.\label{K1-K8}
\end{align}

\section{Estimates on the Evans function}\label{sec:Evans}
\subsection{Evans function estimates} 
The main goal of this section is to prove the following lemma, which in term implies Lemma \ref{lem:ReadableEvans}.
\begin{lemma}\label{lem:Evans_function_est}
There exists a universal $\nu_0$ such that for $\nu \in (0,\nu_0]$, the Evans function $D(\alpha,c)$ is non-zero except in the region
\begin{align}
\set{z \in \Complex\big| \ z_i < -\alpha \nu - \delta \alpha^{-1/3} \nu^{1/3}},\label{Straight_contour}
\end{align} 
for $\delta<\delta_0$ (a sufficiently small universal constant). 

Moreover, the following bounds are satisfied.
\begin{itemize}
\item[(i)] For any $c\in\{c\in \mathbb{C}|c=c_r-i\al \nu-i\delta \al^{-1/3}\nu^{1/3},\quad \forall c_r\in(-\infty,\infty)\}$:
\begin{align}\label{Evans_function_lower_bound_1}
|D(\al, c)|\gtrsim \bigg|\int_{-1}^1 e^{-\al w}H_-(W)dw\bigg|\bigg|\int_{-1}^1 e^{\al w}H_+(W)dw\bigg|.
\end{align}
\noindent
Moreover, there holds 
\begin{align}
\left|\int_{-1}^1e^{\al w}H_+(W)dw\right|\gtrsim \frac{\ep e^\al}{\lan \al\ep\ran}\left|A_0\left(\frac{-1+ c_r}{\ep}+\delta i\right)\right|\label{H_+_int},\\
\left|\int_{-1}^1e^{-\al w}H_-(W)dw\right|\gtrsim \frac{\ep e^\al}{\lan\al\ep\ran}\left|A_0\left(\frac{-1-c_r}{\ep}+i\delta\right)\right|\label{H_-_int}. 
\end{align}

\item[(ii)] For $\al^2\nu\geq C_0$ and $c\in\{c\in \mathbb{C}|c=c_r-i(1-\kappa)\al \nu,\quad \forall c_r\in(-\infty,\infty)\}$:
\begin{align}\label{Evans_function_lower_bound_2}
|D(\al, c)|\gtrsim \bigg|\int_{-1}^1 e^{-\al w}H_-(W)dw\bigg|\bigg|\int_{-1}^1 e^{\al w}H_+(W)dw\bigg|.
\end{align}
\noindent
Moreover,  there holds  
\begin{align}
\left|\int_{-1}^1e^{\al w}H_+(W)dw\right|\gtrsim {\ep e^\al}\left|A_0\left(\frac{-1+ c_r}{\ep}+\delta i\right)\right|\label{H_+_int_high_mode},\\
\left|\int_{-1}^1e^{-\al w}H_-(W)dw\right|\gtrsim \ep e^\al\left|A_0\left(\frac{-1-c_r}{\ep}+i\delta\right)\right|\label{H_-_int_high_mode}. 
\end{align}
\end{itemize} 
\end{lemma}
\begin{proof}
Consider case (i) first. 
First, define the variables $d,C_j^\star$
\begin{subequations}
\begin{align}\label{defn_C_j_star}
d=&-ic+\al\nu, \\ 
C_1^\star& :=\frac{-1-id}{\ep}=\frac{-1-c_r}{\ep}+\delta i,\\
C_2^\star& :=\frac{-1-i\overline{d}}{\ep}=\frac{-1+c_r}{\ep}+\delta i.
\end{align}
\end{subequations}
By definition \eqref{Langer_variable} for the spectral parameter $c$ on the vertical line $c_i\equiv -\al\nu-\delta\ep$ we get
\begin{align}
Z:=\frac{z-id}{\ep}=\frac{z-c_r}{\ep}+\delta i.\label{Langer_variable_straight_line}
\end{align}
Define also the function $u(z,t)$ 
\begin{equation}\label{defn_u}
u(z,t)=\frac{A_0(z+t)}{A_0(z)}=\exp\bigg(\int_0^t \frac{A_0'(z+s)}{A_0(z+s)}ds\bigg).
\end{equation} 
\noindent
\textbf{Step 1: Rephrasing the lower bound \eqref{Evans_function_lower_bound_1}.}
First, we show that the lower bound of the Evans function $|D(\al,c)|$  \eqref{Evans_function_lower_bound_1} is follows from the follwing: 
\begin{align}
e^{2\al}&\frac{1}{1+q}\left|1-e^{-2\al}u\left(C_j^\star,\frac{2}{\ep}\right)\right|\geq \frac{1}{1+q}\left|1-e^{2\al}u\left(C_j^\star,\frac{2}{\ep}\right)\right|\nonumber\\
&+\al\ep\int_0^{2/\ep}(e^{2\al-\al\ep t}+e^{\al\ep t})\left|u\left(C_j^\star,t\right)\right|dt+\frac{q}{1+q}(1+e^{2\al-2a/\ep})\label{proof_of_lem_1_1}
\end{align}
for some fixed universal constant $q>0$.
To this end, we first note that the lower bound \eqref{Evans_function_lower_bound_1} is implied by the relations:
\begin{align}
\frac{1}{1+q}\bigg|\int_{-1}^1 e^{-\al w}Ai(e^{i\pi/6}\frac{w-id}{\ep})dw\bigg|\geq
\bigg|\int_{-1}^1 e^{\al w}Ai(e^{i\pi/6}\frac{w-id}{\ep})dw\bigg|;\label{3.16}\\
\frac{1}{1+q}\bigg|\int_{-1}^1 e^{\al w}Ai(e^{i5\pi/6}\frac{w-id}{\ep})dw\bigg|\geq
\bigg|\int_{-1}^1 e^{-\al w}Ai(e^{i5\pi/6}\frac{w-id}{\ep})dw\bigg|.\label{3.17}
\end{align}
Secondly, by substitution  ($x=-1+\ep t$ in \eqref{3.16} and $x=1-\ep t$ in \eqref{3.17}) and the fact that $A_i(w)=\overline{Ai(\overline{w})}$, \eqref{3.16} and \eqref{3.17} are equivalent to
\begin{align}
\frac{1}{1+q}e^{2\al}\bigg |\int_0^{2/\ep}e^{-\al\ep t} Ai\left(e^{i\pi/6}\left(\frac{-1-id}{\ep}+t\right)\right)dt\bigg|\geq \bigg |\int_0^{2/\ep}e^{\al\ep t} Ai\left(e^{i\pi/6}\left(\frac{-1-id}{\ep}+t\right)\right)dt\bigg|\\
\frac{1}{1+q}e^{2\al}\bigg |\int_0^{2/\ep}e^{-\al\ep t} Ai\left(e^{i\pi/6}\left(\frac{-1-i\overline{d}}{\ep}+t\right)\right)dt\bigg|\geq \bigg |\int_0^{2/\ep}e^{\al\ep t} Ai\left(e^{i\pi/6}\left(\frac{-1-i\overline{d}}{\ep}+t\right)\right)dt\bigg|, 
\end{align}
which in turn hold provided the following is satisfied:
\begin{equation}
e^{2\al}\frac{1}{1+q}\bigg |\int_0^{2/\ep}e^{-\al\ep t} Ai(e^{i\pi/6}(C_j^\star+t))dt\bigg|\geq \bigg |\int_0^{2/\ep}e^{\al\ep t} Ai(e^{i\pi/6}(C_j^\star+t))dt\bigg|,\quad j=1,2, \label{Proof_of_lemma_1_2}
\end{equation}
where the quantities $C_j^\star$ \eqref{defn_C_j_star} take values in the domain
\begin{align}\label{Domain_Z_star}
Re(C_j^\star)\leq& \max\left\{\frac{-1- c_r}{\ep},\frac{-1+ c_r}{\ep}\right\},\quad Im(C_j^\star)=\frac{- c_i-\al \nu}{\ep}= \delta,\quad{\al\geq 0}.
\end{align}
According to Lemma \ref{lem_0}, $|A_0(C_j^\star)|$ is non-zero in this domain. Integrating both integrals in the inequality \eqref{Proof_of_lemma_1_2} by parts and then dividing by $|A_0(C_j^\star)|$ yield that
\begin{align}
&\frac{e^{2\al}}{1+q}\left|\int_0^{2/\ep}(\al \ep)e^{-\al\ep t} \frac{A_0(C_j^\star+t)}{A_0(C_j^\star)}dt+\frac{e^{-2\al}A_0(C_j^\star+\frac{2}{\ep})}{A_0(C_j^\star)}-1\right|\nonumber\\
&\geq \left|-\int_0^{2/\ep} \al \ep e^{\al \ep t} \frac{A_0(C_j^\star+t)}{A_0(C_j^\star)}dt+\frac{e^{2\al}A_0(C_j^\star+\frac{2}{\ep})}{A_0(C_j^\star)}-1\right|,\quad j=1,2.
\end{align}
Combining it with the definition of $u$ \eqref{defn_u}, we obtain the result \eqref{proof_of_lem_1_1}.

\textbf{Step 2: Proof of the inequality  \eqref{proof_of_lem_1_1}.} Recalling from \cite{Romanov73}, we obtain
\begin{equation}
|u(C_j^\star,t)|=\bigg|\exp\bigg(\int_0^t Re\frac{A_0'(C_j^\star+s)}{A_0(C_j^\star+s)}ds\bigg)\bigg|\leq e^{-at}, \quad\forall t\in\left[0,\frac{2}{\ep}\right],\quad j=1,2,\label{ubound}
\end{equation}
where $a$ is defined in \eqref{defn_of_a}. 
Next, we square both sides of \eqref{proof_of_lem_1_1} (note both sides are positive), use $u(C_j^\star,2\ep^{-1})=-e^{-2a/\ep}$ and the upper bound \eqref{ubound}, to obtain that the following implies \eqref{proof_of_lem_1_1},  
\begin{equation}
e^{2\al}(1+e^{-2\al-2a/\ep})\geq 1+e^{2\al-2a/\ep}+(1+q)\al\ep\int_0^{2/\ep}(e^{2\al-\al\ep t}+e^{\al\ep t})e^{-at}dt+q(1+e^{2\al-2a/\ep}).
\end{equation}
Now substituting $t=(1+s)/\ep$ and dividing both sides by $\al e^{\al-a/\ep}$ yields the equivalent inequality 
\begin{equation}\label{proof_of_lem_1_3}
e^{a/\ep}\frac{a}{\ep}\int_0^1\frac{\sinh \al s}{\sinh\al}\sinh\frac{as}{\ep}ds\geq 1+\frac{q}{4(1+q)}\left(\frac{e^\al }{\sinh \al}+\frac{e^{2a/\ep}}{e^\al\sinh\al}\right). 
\end{equation}
The left hand side is calculated through integration by parts as
\begin{align}
e^{a/\ep}\frac{a}{\ep}\int_0^1\frac{\sinh \al s}{\sinh\al}\sinh\frac{as}{\ep}ds=
\frac{e^{2a/\ep}+1}{2}\left(\frac{1-\frac{\al\epsilon}{a} \frac{\tanh(a/\epsilon)}{\tanh(\alpha)}}{1-(\al\ep/a)^2}\right) =:F\left(\al,\frac{a}{\ep}\right).
\end{align}
As a result, the following lower bound of the function $F$ yields the inequality \eqref{proof_of_lem_1_1}
\begin{equation}\label{goal}
F\left(\al,\frac{a}{\ep}\right)>1+\frac{q}{4(1+q)}\left(\frac{e^\al}{\sinh\al}+\frac{e^{2a/\ep}}{e^\al\sinh\al}\right).
\end{equation}
The remaining part is devoted to proving this lower bound. Recall some properties of the function $F$ from \cite{Romanov73}.
The function $F(\al,\beta)$ is decreasing in terms of $\al$ and is increasing in terms of $\beta$. 
\begin{figure}
\centering
\includegraphics[width=6cm]{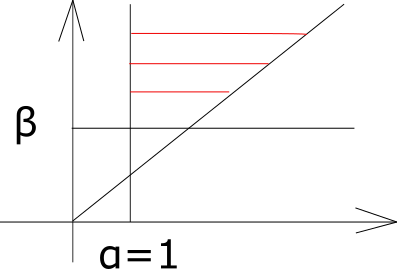}
  \caption{domain of interest}\label{region_D}
\end{figure}
Next we distinguish between two regimes: 
 $K$ sufficiently large such that $Ka\geq 100$ where $a$ is defined in \eqref{defn_of_a}, we define 
\begin{align}\label{low_mode_high_mode}
1)\text{ Low modes: } \al\ep\leq Ka;\\
2)\text{ High modes: } \al\ep\geq Ka. 
\end{align}
In the low mode case 1), to derive the lower bound, we consider the minimum of the function $F(\al,\beta)$ in the domain (here $B_K$ is a constant chosen sufficiently large relative to $K$)
\begin{align*}
D:=\{(\al,\beta)|\beta\geq B_K, K\beta\geq \al,
\al\geq 1\}.
\end{align*}
A figure of the region $D$ can be found in Figure \ref{region_D}.  Due to monotonicity, the minimum of $F(\al,\beta)$ is achieved on the half-line $\al=K\beta, \beta\geq B_K$. On this line segment, explicit calculation yields
\begin{equation}\label{goal_0}
\min_{(\al,\beta)\in D}F(\al,\beta)\geq \min_{\beta\geq B_K}F(K\beta,\beta)=\min_{\beta\geq B_K}\frac{e^{2\beta}+1}{2}\frac{1-K\frac{\tanh(\beta)}{\tanh(K\beta)}}{1-K^2}\geq \min_{\beta\geq B_K}{ \frac{e^{2\beta}}{2(K+1)}-\frac{1}{2(K-1)}}.
\end{equation}
Choosing $q$ sufficiently small relative to $K^{-1}$ and $\nu$ sufficiently small relative to $a$ and $K^{-1}$, \eqref{goal_0} yields \eqref{goal}.  

\noindent
For the high mode case 2),  
we estimate the function $F$ as follows:
\begin{align}
F(\al,a/\ep)=\frac{e^{2a/\ep}+1}{2}\frac{\frac{\al\epsilon}{a}\frac{\tanh(a/\epsilon)}{\tanh(\alpha)}-1}{(\al\ep/a)^2-1}\geq \frac{e^{2a\al^{1/3}\nu^{-1/3}}+1}{2}\frac{K\frac{\tanh(a\al^{1/3}\nu^{-1/3})}{\tanh\alpha}-1}{\al^{4/3}\nu^{2/3}/a^2-1}\geq\frac{K/2}{\al^{4/3}}\frac{e^{\al^{1/3}}+1}{2}, 
\end{align}
which implies \eqref{goal}. 

\noindent
\textbf{Step 3: Proof of inequalities \eqref{H_+_int} and \eqref{H_-_int}.}
Recall \eqref{defn_C_j_star}. By arguments similar to that used to prove \eqref{Evans_function_lower_bound_1}, a suitable lower bound $b$ as follows yields \eqref{H_+_int} and \eqref{H_-_int}
\begin{equation}\label{proof_of_lem_1_4}
e^{\al}\ep\left|\int_0^{2/\ep} e^{-\al \ep t }H_-(C_j^\star+t)dt\right|\geq b>0.
\end{equation}
Recalling the definition of $A_0$\eqref{A_0} and $u$ \eqref{defn_u},  then an integration by parts yields that the following implies \eqref{proof_of_lem_1_4}:
\begin{align}\label{proof_of_lem_1_5}
e^{\al}\ep\left|u\bigg(C_j^\star,\frac{2}{\ep}\bigg)e^{-2\al}-1\right|\geq \ep e^\al\left|\int_0^{2/\ep}e^{-\al\ep t}\al\ep\left|u(C_j^\star, t)\right|dt\right|+\frac{b}{|A_0(C_j^\star)|}.
\end{align}
Using the upper bound \eqref{ubound} and $u(C_j^\star,2\ep^{-1})=-e^{-2a/\ep}$, we see that \eqref{H_+_int} and \eqref{H_-_int}  hold if the following is satisfied:
\begin{align}\label{proof_of_lem_1_6}
e^\al(1-e^{-2\al-2a/\ep})\geq e^\al\int_0^{2/\ep} e^{-\al \ep t}\al\ep e^{-at}dt+\frac{b}{|A_0(C_j^\star)|\ep}.
\end{align}
A calculation shows that \eqref{proof_of_lem_1_6} holds if
\begin{equation}\label{proof_of_lem_1_7}
\bigg(\frac{a}{\al\ep +a}\bigg)(e^{\al}-e^{-\al-2a/\ep})\geq \frac{b}{|A_0(C_j^\star)|\ep}.
\end{equation}
To prove the inequality  \eqref{proof_of_lem_1_7}, we distinguish between the high modes and low modes \eqref{low_mode_high_mode} again. If $\al$ is small, i.e., $\al\ep\leq Ka$, the inequality \eqref{proof_of_lem_1_7} is satisfied if
$b$ is small 
\begin{align}
b\leq \frac{e^\al |A_0(C_j^\star)|\ep}{(K+1)2}.
\end{align}
{For the high modes $\al\ep\geq Ka\geq 100$, the inequality \eqref{proof_of_lem_1_7} holds if 
\begin{align*}
b\leq \frac{e^\al a|A_0(C_j^\star)|\ep}{2(\al\ep+a)}\approx\frac{e^\al|A_0(C_j^\star)|\epsilon}{\al\ep}.
\end{align*}}
Combining the estimates in different regimes, we obtain that \eqref{H_+_int} and \eqref{H_-_int} hold as long as the lower bound $b$ in \eqref{proof_of_lem_1_4} is smaller than
\begin{align*}
b\lesssim \frac{e^\al|A_0(C_j^\star)|\epsilon}{\lan\al\ep\ran}.
\end{align*}
Combining it with the asymptotic expansion of $A_0$ in Lemma \ref{lem:IntAiry} yields the estimate \eqref{D_al_c_lower_bound}. This completes the proof part (i) of the lemma.

\noindent
\textbf{Step 4: Part (ii) -- Proof of the inequalities \eqref{Evans_function_lower_bound_2}, \eqref{H_+_int_high_mode} and \eqref{H_-_int_high_mode} in the high mode case.}
Let us comment on the proof of \eqref{Evans_function_lower_bound_2}, \eqref{H_+_int_high_mode}, and \eqref{H_-_int_high_mode}.
Here we use the observation [(3.5), \cite{Romanov73}]: for sufficiently large $R>0$, and for all $z$ in $G_R:=\{z||z|\geq R,-13\pi/12\leq \mathrm{ph} z\leq \pi/12\}$, the following inequality holds for some universal constant $B > 0$, 
\begin{align}\label{key_high_mode}
Re\frac{A_0'(z)}{A_0(z)}\leq -R^{1/2}\cos\frac{7\pi}{24}+\frac{B}{R}. 
\end{align}
Following the estimate \eqref{key_high_mode}, we have that the Langer variables in this case satisfy
\begin{align*}
C_j^\star=\frac{-1\pm c_r}{\ep}-\kappa\frac{\al\nu}{\ep}i,\quad |C_j^\star|\gtrsim \al^{4/3}\nu^{2/3}.
\end{align*}
Now from \eqref{key_high_mode} and the definition of $a$ in \eqref{defn_of_a}
gives $a\gtrsim  \kappa^{1/2}\al^{2/3}\nu^{1/3}=\kappa^{1/2}\al\ep \gtrsim 1$.  
Therefore, by choosing $K$ sufficiently large we have $\alpha \eps \leq K(\kappa) a$.
This reduces to case (1) in the previous steps. The inequalities \eqref{Evans_function_lower_bound_2}, \eqref{H_+_int_high_mode} \eqref{H_-_int_high_mode} with implicit constants depending on $\kappa$ follow by the same arguments as above. Note that $\lan \al\ep\ran $ will not appear. 
\end{proof}

\subsection{Evans function estimate in the connection region}
Recall the contours $\Gamma_t^{\pm}$, $\Gamma_E$, and $\Gamma_j^{\pm}$. On the contour, it is clear that we may write $c_i$ as a function of $c_r$; denote this function $c_i = \Gamma(c_r)$. 
\begin{lemma}[Connection region Evans function] 
There exists a universal $\nu_0$ such that for $\nu \in (0,\nu_0]$, the Evans function $D(\alpha,c)$ is non-zero except in the region
\begin{align}
\set{z \in \Complex\big| \Gamma(z_r) >  z_i }. 
\end{align} 
Moreover, on $\Gamma_t^{\pm}, \Gamma_j^{\pm}$ the lower bounds \eqref{Evans_function_lower_bound_1}, \eqref{H_+_int}, and \eqref{H_-_int} all hold. 
\end{lemma}
\begin{proof} 
The proof is similar to Lemma \ref{lem:Evans_function_est} but some changes are required because $c_i$ is no longer constant and the $ C_j^\star$'s, defined as 
\begin{align}
C_1^\star=\frac{-1-id}{\ep}=\frac{-1- c_r}{\ep}+\frac{-c_i-\al\nu}{\epsilon} i,\\
C_2^\star=\frac{-1-i\overline{d}}{\ep}=\frac{-1+ c_r }{\ep}+\frac{-c_i-\al\nu}{\ep} i,
\end{align}
are no longer in the region specified in Lemma \ref{lem_0}.
As above, one must bound $a$ from below on the contour: 
\begin{equation}\label{a_co2}
-\max_{c \in \Gamma_2^{\pm}\cup \Gamma_t^\pm} Re\frac{A_0'(C_j^\star+s)}{A_0(C_j^\star+s)}>0,\quad \forall s\in [0,{2}{\ep}^{-1}]. 
\end{equation}
For this, we use again \eqref{key_high_mode}. Since the angle between the region $\Gamma_2^+\cup \Gamma_t^+$ ($\Gamma_2^-\cup \Gamma_t^-$) and the positive imaginary (negative) axis are small, the argument $C_j^\star+s$ in the equation \eqref{a_co2} is inside the domain $G_R$. This completes the proof of \eqref{a_co2}; the rest of the argument follows similarly to Lemma \ref{lem:Evans_function_est}. 
\end{proof} 

\section*{Acknowledgments}
The authors would like to thank Anna Mazzucato and Vlad Vicol for helpful discussions, especially regarding Corollary 1.8.

\bibliographystyle{abbrv}
\bibliography{bibliography}

\end{document}